\documentclass[12pt]{article}

\usepackage{amsfonts,amsmath,amsthm,amssymb}
\usepackage{graphicx}
\usepackage[a4paper, total={6in, 9in}]{geometry}
\usepackage[caption=false]{subfig}
\usepackage{authblk}
\usepackage{bm}
\usepackage[authoryear]{natbib}
\theoremstyle{plain}

\newtheorem{proposition}{Proposition}[section]
\newtheorem{theorem}{Theorem}[section]
\newtheorem{lemma}{Lemma}[section]
\newtheorem{corollary}{Corollary}[section]

\theoremstyle{definition}

\newtheorem{example}{Example}[section]
\newtheorem{assumption}{Assumption}[section]

\theoremstyle{remark}
\newtheorem{remark}{\textbf{Remark}}[section]

\DeclareMathOperator*{\btheta}{\bm{\theta}}

\newcommand{\tr}{\operatorname{tr}}
\newcommand{\Var}{\operatorname{Var}}
\newcommand{\Cov}{\operatorname{Cov}}
\newcommand{\E}{\operatorname{E}}

\begin{document}

\title{A Composite Divergence Approach to Robust Multivariate Estimation under Cellwise and Casewise Contamination}

\author[1]{Abhik Ghosh}
\author[2]{Claudio Agostinelli}
\author[1]{Ayanendranath Basu}

\affil[1]{Indian Statistical Institute, Kolkata, India.}
\affil[2]{University of Trento, Trento, Italy.}

\maketitle

\begin{abstract}
Composite likelihood (CL) methods provide a computationally efficient alternative to full likelihood inference  for complex multivariate models by replacing the joint likelihood with a product of lower-dimensional marginal or conditional components. 
Like the MLE, however, the maximum CL estimator (MCLE) is highly sensitive to data contamination.
On the other hand, robust divergence-based procedures such as the minimum density power divergence (DPD) estimator 
require the full joint density and so scale poorly to complex multivariate models. 
We introduce the composite DPD (CDPD), a genuine statistical divergence built entirely from the low-dimensional component densities 
defining a CL, combining the computational scalability of CL with the robustness of the DPD. The resulting minimum CDPD estimator (MCDPDE) robustifies the MCLE without requiring integration over the full multivariate sample space.

We establish consistency, asymptotic normality, and the influence function of the MCDPDE under regularity conditions 
on the component models alone, without requiring correct specification of the full joint distribution.
We also show that it is qualitatively robust for every positive value of its tuning parameter, unlike the MCLE recovered as the limit.  Because its components can be chosen at the pairwise or cell level, the framework guards simultaneously against casewise and cellwise contamination. 
Operating directly on component densities rather than elliptical distance structures, it extends robust inference beyond the elliptical models to which most existing cellwise-robust procedures are confined.

We develop computational algorithms for Gaussian pairwise CLs and for a non-elliptical multivariate models 
based on McKay's bivariate gamma distribution for ordered positive data, which are implemented in the accompanying \textsf{R} package \texttt{mvdpd} available on CRAN. 
Simulation studies and real-data applications show that the MCDPDE achieves substantial robustness gains over the MCLE 
while retaining competitive efficiency under the assumed model.

\end{abstract}

\section{Introduction} 
\label{SEC:intro}

Modern statistical applications increasingly require inference for complex multivariate models with rich dependence structures. 
In many such settings, the full likelihood is either analytically intractable or computationally prohibitive in large  dimensions. 
Composite likelihood (CL) methods provide a powerful solution to this problem 
by replacing the full likelihood with a product of lower-dimensional marginal or conditional likelihoods.
It was originally motivated by the pseudo-likelihood approach of \cite{besag:1974spatial},  
later formalized by \cite{lindsay:1988composite}. 
Although a CL is generally not a probability density for the full observation vector, 
each component is a genuine likelihood associated with a lower-dimensional event.
The maximum CL estimator (MCLE) thus retains standard likelihood-based properties,
such as consistency, asymptotic normality, and efficiency often close to the maximum likelihood estimator (MLE) when the latter is available
\citep{lindsay:1988composite,varin/etc:2005,varin/etc:2011}, 
while requiring correct specification of only the component distributions rather than the entire joint distribution \citep[][Sec.~2.1]{xu/reid:2011}.
These advantages have led to widespread use of CL methods in spatial statistics, network models, genetics, and high-dimensional multivariate data among others
\citep{varin/etc:2011,rodriguez/etc:2024,schwob/etc:2024,ranalli/etc:2024, liu/etc:2025,rimella/etc:2025}.

Despite these advantages, the MCLE inherits a fundamental limitation of likelihood-based inference,
namely acute sensitive to outliers and data contamination. 
This issue becomes particularly serious  in multivariate datasets featuring complex, latent contamination structures. 
Alongside classical \textit{casewise} contamination, where an entire vector observation is corrupted, 
practical data frequently exhibit \textit{cellwise} contamination, in which only a subset of coordinates in an observation is affected, 
e.g., through recording errors, sensor failures, heterogeneous data pipelines. 
Since a single corrupted cell can enter several CL components at once, 
it may cause severe bias in both standard and composite likelihood methods.

A large body of robust statistical methods have been developed 
to reduce sensitivity to contamination while preserving efficiency under the assumed model. 
Classical robust multivariate scatter estimators range from the minimum volume ellipsoid (MVE) and minimum covariance determinant (MCD)
\citep{Rousseeuw1985,Rousseeuw1999} to M-, S-, $\tau$-, GM-, and MM-estimators  
\citep{Huber2009robust,Tyler1987,Tyler2014,Ruppert1992,Maronna2002,Maronna2016}. 
These methods rely heavily on the Mahalanobis-distance (MD) and are largely confined to elliptical families,
with limited applicability to general non-elliptical multivariate models  \citep{Maronna2016,Tatsuoka2000}. 
Even recent cellwise robust procedures, such as the cellwise MCD \citep{raymaekers2024cellwise}, are primarily developed for 
Gaussian or elliptically symmetric settings and do not directly extend to general dependence structures.

A complementary approach to robust inference is based on robust likelihood modifications, 
including weighted likelihood  \citep{Agostinelli2019} and minimum divergence methods \citep{Basu2011,pardo2018statistical}. 
They directly compare the assumed model distribution with the empirical distribution 
and can be formulated for any general parametric models. Important examples of divergences used in this context are the $\phi$-divergences, 
the density power divergence (DPD) of \cite{Basu1998},  and related families 
\citep{moon2014multivariate,chakraborty2026robust,nkum2024robust,rekavandi2024learning}.
Among these, the minimum DPD estimator (MDPDE) has become popular as a general-purpose robust parametric procedure, 
which is consistent and asymptotically normal, has a bounded influence function for positive values of its tuning parameter $\beta$,
and contains the MLE as a limiting case as $\beta\downarrow 0$ \citep{Basu2011,basu2026}.

However, their direct application to multivariate models face the same computational obstacle as the full-likelihood approach, 
since the objective function typically involves the full joint density and its integrals over the entire sample space, 
a daunting task when the joint density is high-dimensional or analytically intractable. This raises a natural question:
\emph{Can the computational advantages of CL be combined with the robustness of divergence-based inference in a principled 	way?}

To the best of our knowledge, such a general framework is not currently available in the literature. 
Existing attempts are tied to specific model classes, such as a composite $\tau$-estimator for linear mixed models
\citep{Agostinelli/Yohai:2016composite} and  a DPD-based method for a specific Ising models \citep{liu/etc:2025dpd}, 
that exploit the special structure of the densities. 
The estimators developed by \cite{Castilla/etc:2017,Castilla/etc:2020,Castilla/etc:2021} are closer in spirit, 
but they applied the DPD directly to the composite likelihood itself. 
Because a CL is generally not a probability density, this may not define a genuine divergence, 
and even when it does, it still requires integrating over the full multivariate sample space, 
thereby loosing the very computational advantage CL was meant to provide (see Section \ref{SEC:MCDPDE} for more discussions). 

The present paper addresses this gap by introducing a general framework for robust CL inference through 
the \emph{composite density power divergence} (CDPD). We replace the (possibly intractable) full-data density 
in the DPD with a collection of valid low-dimensional component densities used to build the CL. 
The resulting CDPD remains a genuine statistical divergence regardless of whether the CL itself integrates to one
and requires integration only over individual components, thereby combining CL's computational scalability with the robustness of the DPD.
The resulting minimum CDPD estimator (MCDPDE) provides a robust extension of the MCLE for general parametric multivariate models. 
It extends naturally to general non-elliptical dependence structures.
Moreover, when pairwise or low-dimensional components are used, 
it guards against both casewise and cellwise contamination simultaneously.
Our main contributions are summarized as follows.

\begin{enumerate}
	\item[(i)] We introduce the CDPD, a genuine divergence defined entirely through component densities, 
	that generalizes the composite Kullback--Leibler divergence (CKLD) of \cite{varin/etc:2005} 
	to introduce robustness through a tuning parameter $\beta>0$.
	It extends the classical DPD to CL settings without specifying the full joint distribution. 
			
	\item[(ii)] We establish the asymptotic theory of the resulting MCDPDEs, 
	including consistency and asymptotic normality under regularity conditions on the component models alone, 
	with no requirement that the full joint distribution be tractable or even correctly specified. 
	We show the estimator is both locally B-robust  and qualitatively robust for every $\beta>0$, 
	and describe how the robustness is lost at $\beta=0$.
		
	\item[(iii)] We develop computational algorithms, with an accompanying \texttt{R} package, 
	for two representative model classes, namely the pairwise CLs based on bivariate Gaussian and 
	McKay's bivariate gamma distribution \citep{mckay1934sampling}.   
	The second one demonstrates natural applicability of our framework well beyond Gaussian and elliptical settings.
	
	\item[(iv)] We show via extensive simulation, under both casewise and cellwise contamination, 
	that the MCDPDE substantially improves robustness relative to the MCLE while maintaining high efficiency under the assumed model. 
	An application to real-life ordered positive-valued data (non-elliptical) further illustrate its practical utility.   
	Even under Gaussian models, the MCDPDE performs competitively with, or slightly outperforms, existing cellwise-robust methods 
	while remaining computable at contamination levels where the cellwise MCD fails outright.
\end{enumerate}

The remainder of the paper is organized as follows. 
Section \ref{SEC:method} introduces the CDPD and the MCDPDE with illustrative examples, and relates them to existing composite divergence constructions.
Section \ref{SEC:theory} establishes the asymptotic theory. 
Section \ref{SEC:simulations} evaluates finite-sample performance under Gaussian and non-elliptical (multivariate ordered gamma) models with casewise and cellwise contamination. Section \ref{SEC:data_ultramarathon} presents an application to ultra-marathon race data, and Section \ref{SEC:discusions} concludes with discussions and directions for future work. 
Technical details, proofs, illustrations on benchmark Gaussian datasets, and additional tables and figures are provided in Appendices~\ref{APP:Computation}--\ref{APP:empResults}, available in the Supplementary Material.

The \textsf{R} package \texttt{mvdpd}, available on CRAN, implements Gaussian and multivariate ordered gamma model while the vignettes therein provide the code to reproduce all the figures and tables from the real datasets.

\section{The Minimum Composite DPD Estimation}
\label{SEC:method}

\subsection{Model Set-up and Composite Likelihood Estimation}
\label{SEC:Model}

Consider a $d$-dimensional random vector $\boldsymbol{X} = (X_1, \ldots, X_d)^\top \in \mathbb{R}^d$  
with true (unknown) population density $g$  and distribution function $G$.
We model $G$ by a family of $d$-dimensional parametric distributions 
$\mathcal{F}^\ast =\left\{ F_{\boldsymbol{\theta}} : \boldsymbol{\theta} \in \Theta \subseteq \mathbb{R}^p\right\}$,
with corresponding family of densities 
$\mathcal{F} =\left\{ f_{\boldsymbol{\theta}} : \boldsymbol{\theta} \in \Theta \subseteq \mathbb{R}^p\right\}$.
We assume throughout that the dominating measure of the model family coincides with that of the true distribution 
and hence does not depend on $\btheta$. Our objective is to estimate the unknown parameter $\boldsymbol{\theta}$ based on  
$n$ independent and identically distributed (IID) observations  $\mathcal{X}_n=\{\boldsymbol{x}_1,\ldots,\boldsymbol{x}_n\}$ on $\bm{X}$.
The classical MLE maximizes the full joint likelihood 
$\mathcal{L}(\boldsymbol{\theta}) = \prod_{i=1}^{n} f_{\boldsymbol{\theta}}(\boldsymbol{x}_i)$
with respect to $\boldsymbol{\theta}\in\Theta$. 

In many complex multivariate settings, the joint density $f_{\boldsymbol{\theta}}$ is either computationally intractable 
or unavailable in closed form. Composite likelihood methods address this difficulty by replacing the full likelihood 
with a product of lower-dimensional likelihood contributions \citep{lindsay:1988composite,varin/etc:2011}. 
Formally, let $\{\mathcal{A}_1, \ldots, \mathcal{A}_K\}$  denote a collection of (marginal or conditional) events, 
and denote the likelihood associated with $\mathcal{A}_k$ by  
$\mathcal{L}_k(\boldsymbol{\theta}|\boldsymbol{X})  \propto f_{\boldsymbol{\theta},k}(\boldsymbol{X}) = f_{\boldsymbol{\theta}}(\boldsymbol{X}\in\mathcal{A}_k)$, 
the density of $\boldsymbol{X}$ restricted to $\mathcal{A}_k$ under the joint model $f_{\boldsymbol{\theta}}$, for $k=1, \ldots, K$.
The composite likelihood (CL) of $\boldsymbol{X}$ is then defined as \citep{lindsay:1988composite,varin/etc:2011}
$\mathcal{L}_C(\boldsymbol{\theta}|\boldsymbol{X}) = \prod_{k=1}^K \mathcal{L}_k(\boldsymbol{\theta}|\boldsymbol{X})^{w_k}$,
where $w_1, \ldots, w_K\in\mathbb{R}^+$ are pre-specified weights independent of $\btheta$. 
Given an IID sample   $\mathcal{X}_n$ of size $n$, the maximum CL estimator (MCLE) of $\boldsymbol{\theta}$ maximizes the CL,
$\mathcal{L}_C(\boldsymbol{\theta}|\mathcal{X}_n) = \prod_{i=1}^n \mathcal{L}_C(\boldsymbol{\theta}|\boldsymbol{x}_i)$, or the log-CL,
\begin{equation}
\ell_C(\boldsymbol{\theta}|\mathcal{X}_n) = 
\sum_{i=1}^n \sum_{k=1}^K w_k \ell_k(\boldsymbol{\theta}|\boldsymbol{x}_i),
\qquad \ell_k(\boldsymbol{\theta}|\boldsymbol{x}) = \ln \mathcal{L}_k(\boldsymbol{\theta}|\boldsymbol{x}),
\label{EQ:log-CL}
\end{equation}
with respect to $\boldsymbol{\theta}\in\Theta$, assuming a maximizer exists. 
Equivalently, it solves the composite score equation given by
\begin{equation*}
\sum_{i=1}^n \sum_{k=1}^K w_k \boldsymbol{u}_{\boldsymbol{\theta},k}(\boldsymbol{x}_i) = \boldsymbol{0}_p,
\end{equation*}
where $\boldsymbol{u}_{\boldsymbol{\theta},k}(\boldsymbol{x}) 
= \nabla \ell_k(\boldsymbol{\theta}|\boldsymbol{x}) = \nabla\ln  f_{\boldsymbol{\theta},k}(\boldsymbol{x}) $
is the score function corresponding to the $k$th likelihood component, and $\nabla$ denotes the gradient with respect to $\btheta$. 
Under standard regularity conditions, this estimating equation is unbiased and the MCLE is Fisher consistent 
whenever the composite likelihood identifies the true parameter.
The event subsets $\{\mathcal{A}_k\}$  and hence component likelihoods may correspond to different marginal or conditional structures, 
including univariate, pairwise, or higher-order marginal distributions, as well as conditional distributions. 
When every event $\mathcal{A}_k$ carries equal weight, the weights are usually dropped and the unweighted CL is used
(equivalent to taking all $w_k$ equal to a common value).
The MCLE also admits a divergence interpretation. Let $g_k$ and $f_{\boldsymbol{\theta},k}$ denote the true and model densities, 
respectively, associated with the $k$th component. The composite Kullback--Leibler divergence (CKLD) is then defined as 
\citep{varin/etc:2005}  
\begin{eqnarray}\label{EQ:CKLD}
\mbox{KLD}_C(g, f_{\boldsymbol{\theta}}) = \sum_{k=1}^K w_k \E_G\left[
\log \frac{g_k(\boldsymbol{X})}{f_{\boldsymbol{\theta},k}(\boldsymbol{X})}\right]. 
\end{eqnarray}
The MCLE can be seen to be a minimizer of the estimated CKLD obtained by replacing the expectation under $G$ 
with the empirical average under the empirical distribution function $G_n$. 
The associated population version, the minimum CL functional of $\boldsymbol{\theta}$, is thus defined as 
\begin{equation}
\boldsymbol{\theta}^\ast = \arg\min_{\boldsymbol{\theta}\in\Theta} \mbox{KLD}_C(g, f_{\boldsymbol{\theta}}),
\end{equation}
whenever such a minimizer exists. 
Under standard regularity assumptions, the MCLE is $\sqrt n$-consistent and asymptotically normal, 
with covariance matrix determined by the Godambe information matrix
$\boldsymbol{I}_G(\boldsymbol{\theta}) = \boldsymbol{J}(\boldsymbol{\theta}) 
\boldsymbol{V}^{-1}(\boldsymbol{\theta})\boldsymbol{J}(\boldsymbol{\theta})$,
where $\boldsymbol{J}(\boldsymbol{\theta})  = - \E_g\left[\nabla^2 \ell_C(\boldsymbol{\theta}|\boldsymbol{X})\right]$
and $\boldsymbol{V}(\boldsymbol{\theta})  = \mbox{Var}_g\left[\nabla \ell_C(\boldsymbol{\theta}|\boldsymbol{X})\right]$
\citep{lindsay:1988composite,varin2008composite}. 

\subsection{The Composite DPD and its minimization}
\label{SEC:MCDPDE}

The density power divergence (DPD) of \citet{Basu1998} provides a principled framework for constructing robust estimators 
by replacing the Kullback-Leibler divergence (KLD) with a family of divergence measures indexed by a tuning parameter $\beta$. 
Formally, given a tuning parameter $\beta> 0$, the DPD between the true joint density $g$ and the model density $f_{\boldsymbol{\theta}}$ 
is defined as 
\begin{eqnarray}
D_{\beta}(g,f_{\boldsymbol{\theta}}) &=& \int f_{\boldsymbol{\theta}}^{1+\beta}(\boldsymbol{x}) \;d\boldsymbol{x} 
- \left(1+\frac{1}{\beta}\right)\int f_{\boldsymbol{\theta}}^{\beta}(\boldsymbol{x})g(\boldsymbol{x})\;d\boldsymbol{x} 
+ \frac{1}{\beta}\int g^{1+\beta}(\boldsymbol{x})\;d\boldsymbol{x}
\nonumber\\
&=& \int f_{\boldsymbol{\theta}}^{1+\beta}(\boldsymbol{x}) \;d\boldsymbol{x} 
- \left(1+\frac{1}{\beta}\right) \E_{G}\left[f_{\boldsymbol{\theta}}^{\beta}(\boldsymbol{X})\right] 
+ \frac{1}{\beta} \E_G \left[g^\beta(\boldsymbol{X})\right].\nonumber
\label{EQ:dpd}
\end{eqnarray}
The limiting case $\beta\downarrow0$ gives the KLD,
\begin{eqnarray}
\label{EQ:KLD}
KLD(g, f_{\boldsymbol{\theta}}) = D_0(g, f_{\boldsymbol{\theta}}) 
= \int \log\left[\frac{g(\boldsymbol{x})}{f_{\boldsymbol{\theta}}(\boldsymbol{x})}\right] g(\boldsymbol{x})d\boldsymbol{x}
= \E_G\left[\log \frac{g(\boldsymbol{X})}{f_{\boldsymbol{\theta}}(\boldsymbol{X})}\right].
\end{eqnarray}
The minimum DPD functional (MDPDF) $\boldsymbol{U}_\beta(G)$, with tuning parameter $\beta\geq 0$, is thus defined as 
a minimizer of $D_{\beta}(g,f_{\boldsymbol{\theta}})$ with respect to $\boldsymbol{\theta}\in\Theta$, whenever it exists. 
The corresponding MDPDE of $\boldsymbol{\theta}$ based on a sample $\mathcal{X}_n$ is defined as $\boldsymbol{U}_\beta(G_n)$, 
i.e., a minimizer of the estimated DPD measures obtained by replacing the expectation with respect to $G$ in (\ref{EQ:dpd}) 
by the empirical average under $G_n$. Replacing the $\bm{\theta}$-free third term in (\ref{EQ:dpd}), 
with the simpler constant $\beta^{-1}$,  the objective  function for the MDPDE turns out to be \citep{Basu1998}
\begin{equation}
 \label{EQ:ObjFunc_MDPDE}
\int f_{\boldsymbol{\theta}}^{1+\beta}(\boldsymbol{x}) \;d\boldsymbol{x} - \left(1+\frac{1}{\beta}\right)\frac{1}{n}\sum_{i=1}^{n}f_{\boldsymbol{\theta}}^{\beta}(\boldsymbol{X}_{i}) + \frac{1}{\beta}.
\end{equation}
This MDPDE coincides with the MLE as $\beta\downarrow 0$, and is highly robust to case-wise contamination  
at a small cost in efficiency for moderate $\beta>0$.
However, the associated optimization problem is already difficult to solve for multivariate Gaussian models at modest sample sizes 
and larger $\beta>0$ \citep{chakraborty2026robust}. For more complex models, it requires evaluation of 
$\int f_{\boldsymbol{\theta}}^{1+\beta}(\boldsymbol{x})d\boldsymbol{x}$, which can be computationally demanding. 
This difficulty is amplified in models where the joint density is unavailable.

To overcome such computational challenges for the otherwise well-performing robust MDPDEs, 
\citet{chakraborty2025componentwise} developed a componentwise MDPDE for robust location and scatter estimation under elliptical models. 
Although it is efficient to compute and performs well under Gaussian models, it fails to capture more complex dependence structures,
and the resulting scatter estimate need not be positive definite at finite sample sizes.
\cite{Castilla/etc:2017} instead proposed minimizing the DPD between the true density $g$ and the CL $\mathcal{L}_C$, directly, 
in place of $f_{\boldsymbol{\theta}}$. This construction has two major problems. 
First, $\mathcal{L}_C$ need not integrate to one over $\bm{x}$; it is not, in general, a genuine probability density, and so 
the resulting object need not be a genuine divergence at all. Second, even where it is well defined,
its objective function still requires evaluating $\int \mathcal{L}_C^{1+\beta}(\boldsymbol{\theta}|\boldsymbol{x})d\boldsymbol{x}$ 
over the full $d$-dimensional sample space, reinstating exactly the computational burden that CL was meant to remove. 

A natural solution is obtained by applying the DPD construction to the component likelihoods 
$f_{\boldsymbol{\theta},k}$ and $g_k$ underlying the CL. 
Writing the DPD in \eqref{EQ:dpd} as an expectation,
\begin{equation*}
D_{\beta}(g,f_{\boldsymbol{\theta}}) = \E_G \left[\int f_{\boldsymbol{\theta}}^{1+\beta} 
- \left(1+\frac{1}{\beta}\right)f_{\boldsymbol{\theta}}^{\beta}(\boldsymbol{X}) + \frac{1}{\beta} g^\beta(\boldsymbol{X})\right],\nonumber
\end{equation*}
we note that the term inside the brackets is itself non-negative, with minimum zero attained only at 
$g(\boldsymbol{x}) = f_{\boldsymbol{\theta}}(\boldsymbol{x})$ for almost all $\boldsymbol{x}$ \citep{Basu2011,basu2026}. 
Guided by the relationship between the DPD and KLD, and the form of the CKLD in \eqref{EQ:CKLD}, 
we define the composite DPD (CDPD) by replacing the joint densities inside this bracket with the component densities 
and taking a weighted sum of their expectations under $G$, exactly as in \eqref{EQ:CKLD}. 
Formally, for any $\beta > 0$, the CDPD is given by  
\begin{eqnarray}
D_{\beta}^C(g,f_{\boldsymbol{\theta}})  = \sum_{k=1}^K w_k \E_G\left[\int f_{\boldsymbol{\theta},k}^{1+\beta} 
- \left(1+\frac{1}{\beta}\right)f_{\boldsymbol{\theta},k}^{\beta}(\boldsymbol{X}) 
+ \frac{1}{\beta} g_k^\beta(\boldsymbol{X})\right].
\label{EQ:CDPD}
\end{eqnarray}
This definition extends the CKLD in exactly the same manner that the ordinary DPD extends the KLD. 
In particular, as $\beta\downarrow 0$, this CDPD coincides in the limit with the CKLD  in (\ref{EQ:CKLD}).

\begin{remark}[Validity of the CDPD]
Since each component $f_{\boldsymbol{\theta},k}$ (and $g_k$) is a proper probability density  on the support of $\mathcal{A}_k$, 
each term $D_\beta(g_k,f_{\boldsymbol{\theta},k})$ is a valid DPD. Therefore,
$D_\beta^C(g,f_{\boldsymbol{\theta}})\ge0$, with equality if and only if $g_k=f_{\boldsymbol{\theta},k}$ a.e. for all $k$,
whenever all weights are positive. Hence the proposed CDPD construction defines a genuine statistical divergence 
regardless of whether $\mathcal{L}_C$  integrates to one. 
This distinguishes the CDPD from previous approaches of \cite{Castilla/etc:2017} 
that directly replace the joint density $f_{\boldsymbol{\theta}}$ in the DPD \eqref{EQ:dpd} by $\mathcal{L}_C$.
\qed
\end{remark}

The minimum CDPD functional (MCDPDF) of $\boldsymbol{\theta}$ with tuning parameter $\beta\geq0$, 
denoted as $\boldsymbol{\theta}_\beta^\ast = \boldsymbol{T}_\beta(G)$, is then defined as a minimizer of 
$D_{\beta}^C(g,f_{\boldsymbol{\theta}})$ with respect to $\boldsymbol{\theta}\in\Theta$.
Since the last term in (\ref{EQ:CDPD}) is independent of $\boldsymbol{\theta}$, 
the MCDPDF equivalently minimizes $\E_G[V_\beta(\boldsymbol{X}|\boldsymbol{\theta})]$ with respect to $\boldsymbol{\theta}\in\Theta$, 
where 
\begin{equation}
V_{\beta}(\boldsymbol{x}|{\boldsymbol{\theta}}) = \sum_{k=1}^K w_k V_{k,\beta}(\boldsymbol{x}|{\boldsymbol{\theta}}),
\mbox{ with }
V_{k,\beta}(\boldsymbol{x}|{\boldsymbol{\theta}}) = \int f_{\boldsymbol{\theta},k}^{1+\beta} 
	- \left(1+\frac{1}{\beta}\right)f_{\boldsymbol{\theta},k}^{\beta}(\boldsymbol{x}) + \frac{1}{\beta} .
	\label{EQ:V_theta}
\end{equation}
The extra term $\frac{1}{\beta}$ ensures $\lim\limits_{\beta\rightarrow 0} V_{k,\beta}(\boldsymbol{x}|{\boldsymbol{\theta}}) 
= - \log f_{\boldsymbol{\theta},k}(\boldsymbol{x})$ for each $k$,
so that the MCDPDF at $\beta=0$ coincides with the MCL functional in the limit. 
It is also straightforward to verify that, by construction, the MCDPDF  is Fisher consistent, 
i.e., $\boldsymbol{\theta}_\beta^\ast = \boldsymbol{T}_\beta(G)=\boldsymbol{\theta}_0$ 
whenever $G_k = F_{\boldsymbol{\theta}_0,k}$ for all $k=1, \ldots, K$, and $\boldsymbol{\theta}_0 \in \Theta$,
even if the full joint density is misspecified.

Given an IID sample $\mathcal{X}_n$ of size $n$, the minimum CDPD estimator (MCDPDE) of $\boldsymbol{\theta}$, 
with tuning parameter $\beta\geq 0$, can be defined as $\widehat{\boldsymbol{\theta}}_{n,\beta} = \boldsymbol{T}_\beta(G_n)$,
which minimizes the empirical objective 
\begin{eqnarray}
	\widehat{\boldsymbol{\theta}}_{n,\beta} = \arg\min_{\boldsymbol{\theta}\in\Theta} \frac{1}{n} \sum_{i=1}^n V_{\beta}(\boldsymbol{x}_i|{\boldsymbol{\theta}})  
	= \arg\min_{\boldsymbol{\theta}\in\Theta} \sum_{k=1}^K \frac{w_k}{n} \sum_{i=1}^n  V_{k,\beta}(\boldsymbol{x}_i|{\boldsymbol{\theta}}),
\label{EQ:MCDPDE}
\end{eqnarray} 
where $V_{\beta}$ and $V_{k,\beta}$s are as defined in (\ref{EQ:V_theta}). 
Under differentiability of component densities $f_{\btheta,k}$ in $\btheta$, the MCDPDE satisfies the estimating equation
\begin{equation}
	\label{EQ:MCDPDE_EstEqn}
	\frac{1}{n}\sum_{i=1}^{n} \boldsymbol{\psi}_\beta(\boldsymbol{x}_i|\boldsymbol{\theta}) = \boldsymbol{0}_p,
\end{equation}
where
\begin{align}
\boldsymbol{\psi}_\beta(\boldsymbol{x}|\boldsymbol{\theta}) & = & \frac{1}{1+\beta}\nabla V_\beta(\boldsymbol{x}|\boldsymbol{\theta}) 
= \frac{1}{1+\beta} \sum_{k=1}^K w_k \nabla V_{k,\beta}(\boldsymbol{x}|{\boldsymbol{\theta}}) 
\nonumber\\
& = & \sum_{k=1}^K w_k \left[\int \boldsymbol{u}_{\boldsymbol{\theta},k} f_{\boldsymbol{\theta},k}^{1+\beta} 
- \boldsymbol{u}_{\boldsymbol{\theta},k}(\boldsymbol{x}) f_{\boldsymbol{\theta},k}^{\beta}(\boldsymbol{x})\right].
\label{EQ:Psi_theta}
\end{align}
Since each $f_{\boldsymbol{\theta},k}$ is a genuine probability density, 
$\E\left[\boldsymbol{\psi}_\beta(\boldsymbol{X}|\boldsymbol{\theta})\right] = \boldsymbol{0}$ for any $\beta\geq 0$,
where the expectation is taken under a joint distribution with component densities $f_{\boldsymbol{\theta},k}$.
The estimating equation is thus unbiased whenever only the model component distributions are correctly specified.

\begin{remark}
As $\beta \downarrow 0$, the MCDPDE reduces to the MCLE. For $\beta > 0$, 
the component scores are downweighted according to the magnitude of the corresponding component density contribution, 
yielding the robustness properties  established formally in Section~\ref{SEC:IF}.
As with the ordinary MDPDE, $\beta$ governs the trade-off between robustness and efficiency. 
A fully data-adaptive choice of $\beta$ in the composite setting, for example, 
by adapting minimum-distance or bootstrap-based selection criteria developed for the ordinary MDPDE, 
is an important direction that we hope to pursue in future.
\qed
\end{remark}

Since the MCDPDE objective involves only lower-dimensional integrals, chosen to be tractable via the composite likelihood's event subsets, 
it is considerably easier to minimize than the objective of the corresponding MDPDE based on the full joint distribution, 
and the same holds for the estimating equation. Because each component density typically involves only a subset of $\btheta$, 
the estimating equations often separate in a way that supports simple fixed-point iteration, 
so the MCDPDE inherits the same computational advantages as the classical MCLE. 
See Appendix \ref{APP:Algo_MVN} for a representative computational algorithm for MCDPDE under the Gaussian model.

\subsection{Illustrative Examples}
\label{SEC:Examples}

The proposed MCDPDE is applicable whenever the model admits a collection of tractable marginal or conditional component densities 
for constructing a CL. The following examples illustrate this generality across continuous, discrete, and high-dimensional models. 
In each case, robustness is achieved without requiring evaluation of the full joint density.

\subsubsection{Elliptically Symmetric Distributions}

Elliptically symmetric distributions provide a natural starting point because they include many widely used multivariate models, 
including Gaussian and Student-$t$ families. A $d$-dimensional elliptical density can be written as \citep{Kelker1970,Chmielewski1981}
\begin{equation}
	\centering
	f_{\boldsymbol{\theta}}(\boldsymbol{x})= c_d |\boldsymbol{\Sigma}|^{-\frac{1}{2}}h\left((\boldsymbol{x}-\boldsymbol{\mu})^\top\boldsymbol{\Sigma}^{-1}(\boldsymbol{x}-\boldsymbol{\mu})\right),
	~~~~~\boldsymbol{x}\in\mathbb{R}^d,
	\label{EQ:density_ES}
\end{equation}
where $\boldsymbol{\mu} = (\mu_1, \ldots, \mu_d)^\top\in\mathbb{R}^d$ is a location vector, 
$\boldsymbol{\Sigma}$ is a $d\times d$ positive definite scatter matrix, 
$h$ is the (possibly dimension-dependent) \textit{density generator}, and $c_d$ is the normalizing constant given by 
\begin{equation*}
c_d = \frac{\Gamma(d/2)}{2\pi^{d/2}} \left[\int_0^\infty r^{d-1}h(r^2) dr \right]^{-1}.
\end{equation*}
The density in (\ref{EQ:density_ES}) depends on $\boldsymbol{x}$ through the squared Mahalanobis distance (MD)
$\Delta(\boldsymbol{x}|\boldsymbol{\mu}, \boldsymbol{\Sigma})
=(\boldsymbol{x}-\boldsymbol{\mu})^\top\boldsymbol{\Sigma}^{-1}(\boldsymbol{x}-\boldsymbol{\mu})$. 
In parametric multivariate modelling, we often assume that $h$ is known 
and estimate the parameters in $\bm{\mu}$ and $\bm{\Sigma}$ from observed data.
However, it is often convenient to reparametrize in terms of the correlation matrix $\bm{R}$, 
with the $(j,k)$-th entry being $\rho_{jk} = \mbox{Corr}(X_j, X_k)$ for $j, k = 1, \ldots, d$.
Then, the $(j,k)$-th entry of $\bm{\Sigma}$ is given by $\sigma_{jk} = \rho_{jk}\sigma_j\sigma_k$ ($j, k = 1, \ldots, d$),
where $\rho_{jj} =1$ and thus $\sigma_{jj} = \sigma_j^2$ denote the variance of $X_j$ for each $j$. 
Writing $\boldsymbol{\theta}_j =(\mu_j, \sigma_j^2)^\top$ for all $j=1, \ldots, d$ and 
$\boldsymbol{r} = (\rho_{12}, \rho_{13}, \ldots, \rho_{d-1, d})^\top$, which stacks the upper triangular entries of $\boldsymbol{R}$ 
(excluding the diagonals) column-wise into a vector of dimension $\frac{1}{2}d(d-1)$, our working parameter space is 
\begin{equation*}
\Theta = \left\{ \boldsymbol{\theta} = (\boldsymbol{\theta}_{1}^\top, \ldots, \boldsymbol{\theta}_{d}^\top, \boldsymbol{r}^\top)^\top \in (\mathbb{R}\times\mathbb{R}^{+})^d\times [-1, 1]^{\frac{1}{2}d(d-1)}: \boldsymbol{R} \text{ is positive definite} \right\}.
\end{equation*}
Since elliptical distributions are fully characterized by their pairwise marginals,
a natural composite construction is thus obtained from the bivariate margins $\boldsymbol{X}_{jk} = (X_j, X_k)$ for $1\leq j < k\leq d$, 
yielding $K=\frac{1}{2}d(d-1)$ components. In particular, under (\ref{EQ:density_ES}), 
the pairwise component density of $\boldsymbol{X}_{jk} = (X_j, X_k)$ is given by 
\begin{equation}
f_{\boldsymbol{\theta},{jk}}(\boldsymbol{x}_{jk})  = \frac{c_{2*}}{\sigma_j\sigma_k\sqrt{1-\rho_{jk}}} 
		h_{jk}\left(\Delta(\boldsymbol{x}_{jk}|\boldsymbol{\mu}_{jk}, \boldsymbol{\Sigma}_{jk})\right), \qquad \boldsymbol{x}_{jk} = (x_j, x_k)^\top\in \mathbb{R}^2,
\label{EQ:density_ES_pairwise}
\end{equation}
where 
$\boldsymbol{\mu}_{jk} = (\mu_j, \mu_k)^\top$, 
$\boldsymbol{\Sigma}_{jk} = 
\begin{bmatrix}
	\begin{array}{cc}
		\sigma_j^2 & \sigma_j\sigma_k \rho_{jk}\\
		\sigma_j\sigma_k \rho_{jk} & \sigma_k^2
	\end{array}
\end{bmatrix}$, 
$h_{jk}$ is the bivariate density generator induced by $h$ for each $1 \leq j < k\leq d$, 
and $c_{2*}$ is the corresponding normalizing constant. 
This depends only on the $\bm{x}_{jk}$ and on a subset of $\btheta$ (unless $h_{jk}$ itself depends on further parameters).  
Then, the resulting pairwise CL is  
\begin{eqnarray}
	\mathcal{L}_C(\boldsymbol{\theta}|\boldsymbol{X}) = \prod_{1 \leq j < k\leq d} \frac{c_{2*}}{\sigma_j\sigma_k\sqrt{1-\rho_{jk}}} 
		h_{jk}\left(\Delta(\boldsymbol{x}_{jk}|\boldsymbol{\mu}_{jk}, \boldsymbol{\Sigma}_{jk})\right),	
\label{EQ:CL_Ellp}
\end{eqnarray}
where equal weights have been used for simplicity  ($w_{jk} \equiv 1$).
The associated MCDPDE is obtained by replacing each bivariate likelihood contribution with its corresponding CDPD criterion. 
This construction remains valid even when the full joint elliptical model is not assumed, 
provided that the relevant pairwise component models are correctly specified.

\begin{example}[Multivariate Gaussian Model]
\label{EX:MVN}
For the multivariate Gaussian distribution, $h(u) = e^{-u/2}$, independent of $d$, and $c_d = (2\pi)^{-d/2}$. 
Since every bivariate Gaussian margin has the same generator, $h_{jk} = h_2 = e^{-u/2}$, and $c_{2*} = c_2 = 1/(2\pi)$, 
the pairwise CDPD components have closed forms. For $\beta > 0$ the MCDPDE minimizes \eqref{EQ:MCDPDE} with
\begin{equation*}
  V_{\beta}(\boldsymbol{x}|{\boldsymbol{\theta}}) = \sum\limits_{1 \leq j < k\leq d} V_{jk,\beta}\left(\boldsymbol{x}_{jk}| \boldsymbol{\theta}_{j}, \boldsymbol{\theta}_{k}, \rho_{jk}\right),
\end{equation*}
where direct integration gives
\begin{equation}
	V_{jk, \beta}\left(\boldsymbol{x}_{jk}|\boldsymbol{\theta}_{j}, \boldsymbol{\theta}_{k},\rho_{jk}\right) 
	= \kappa_{jk}(\beta) \left[	\frac{1}{(1+\beta)} 
	- \frac{1+\beta}{\beta} e^{-\frac{\beta}{2}\Delta(\boldsymbol{x}_{jk}|\boldsymbol{\mu}_{jk}, \boldsymbol{\Sigma}_{jk})}\right]
	+\frac{1}{\beta},
\label{EQ:Vjk-gaussian}
\end{equation}
with $\kappa_{jk}(\beta) = \sigma_j^{-\beta}\sigma_k^{-\beta}(1-\rho_{jk}^2)^{-\beta/2}(2\pi)^{-\beta}$.
As $\beta \downarrow0$, $V_{jk, \beta}\left(\boldsymbol{x}_{jk}|\boldsymbol{\theta}_{j}, \boldsymbol{\theta}_{k},\rho_{jk}\right)$
tends to the log-CL, given by 
\begin{equation*}
\ell_{jk}(\boldsymbol{\theta}|\boldsymbol{x}_{jk}) = - \frac{1}{2}
\left[\log(2\pi \sigma_j^2\sigma_k^2(1-\rho_{jk}^2)) + \Delta(\boldsymbol{x}_{jk}|\boldsymbol{\mu}_{jk}, \boldsymbol{\Sigma}_{jk}) \right].
\end{equation*}
Hence the MCDPDE continuously reduces to the pairwise MCLE when $\beta=0$.
For $\beta>0$, the contribution of an observation is modulated by the factor 
$\exp\left\{-\frac{\beta}{2}\Delta(\boldsymbol{x}_{jk})\right\}$,
which decreases with the pairwise MD. Thus, unlike the MCLE, extreme observations have reduced influence. 
The resulting estimating equations and the computational algorithm are given in Appendix~\ref{APP:Algo_MVN}.
\qed
\end{example}

\begin{example}[Multivariate Student's $t$ Model]
The multivariate $t_\nu(\bm{\mu},\bm{\Sigma})$ distribution is a natural heavier-tailed alternative within the elliptical family, 
having density generator $h(u) \propto (1+u/\nu)^{-(\nu+d)/2}$, for pre-specified degrees of freedom $\nu > 0$. 

A useful property of the multivariate $t$ distribution is that every marginal sub-vector is again Student's $t$ distributed with 
the same degrees of freedom. Consequently, each pairwise component required by the composite construction is available in closed form as
\begin{equation*}
f_{\btheta,jk}(\bm{x}_{jk}) = \frac{k_2(\nu)}{\sigma_j \sigma_k (1-\rho_{jk}^2)^{1/2}}
	\left[1 + \frac{\Delta(\bm{x}_{jk} \mid \bm{\mu}_{jk}, \bm{\Sigma}_{jk})}{\nu}\right]^{-(\nu+2)/2},
\end{equation*}
and $k_2(\nu) = \Gamma\left(\frac{\nu+2}{2}\right)/(\Gamma\left(\frac\nu2\right) \pi \nu)$. Using the substitution $y = (x_{jk}-\mu_{jk})/\sqrt\nu$ followed by the	standard radial-integral identity 
$\int_{\mathbb{R}^2} (1+u^\top A u)^{-a}du = |A|^{-1/2} \pi\, \Gamma(a-1)/\Gamma(a)$ for $a>1$, 
we can simplify the objective function of the MCDPDE. Its components has the form, with $a_\beta = (1+\beta)(\nu+2)/2$,
\begin{align}
V_{jk, \beta}\left(\boldsymbol{x}_{jk}|\boldsymbol{\theta}_{j}, \boldsymbol{\theta}_{k},\rho_{jk}\right) 
& = \frac{(1+\beta)k_2(\nu)^\beta}{\sigma_j^\beta\sigma_k^\beta(1-\rho_{jk}^2)^{\beta/2}(2\pi)^\beta}		
\Bigg\{ \nu\pi\, k_2(\nu)\frac{\Gamma(a_\beta-1)}{\Gamma(a_\beta)} \nonumber \\
& - \frac{1}{\beta} \left[1 + \frac{\Delta(\bm{x}_{jk}\mid\bm{\mu}_{jk},\bm{\Sigma}_{jk})}{\nu}\right]^{-\beta(\nu+2)/2}\Bigg\} 
+ \frac1\beta.
	\label{EQ:Vjk-t}
\end{align}
As $\nu \to \infty$, $k_2(\nu) \to (2\pi)^{-1}$ and \eqref{EQ:Vjk-t} reduces to \eqref{EQ:Vjk-gaussian}, recovering the Gaussian case. 
Interestingly, unlike the exponential downweighting $e^{-\beta \Delta/2}$ in the Gaussian case,
observations here are downweighted at the polynomial rate $[1+\Delta/\nu]^{-\beta(\nu+2)/2}$. 
The resulting MCDPDE combines robustness arising from the heavy-tailed model itself 
with the additional protection induced by the DPD tuning parameter.
The associated estimating equations follow by differentiating \eqref{EQ:Vjk-t} with respect to $(\mu_j,\sigma_j^2,\rho_{jk})$,
which we skip for brevity.  
\qed
\end{example}

\subsubsection{Non-Elliptical Continuous Multivariate Distributions}

Composite likelihood methods are particularly attractive for multivariate models when the dependence structure is complicated beyond elliptical symmetry but lower-dimensional marginals remain tractable. 
We illustrate this point using a class of ordered multivariate distributions generated through
gamma increments on the cumulative hazard scale, introduced by \citet{balakrishnan2016multivariate}. 
Such models arise naturally for ordered event times, reliability data, hydrology, environmental applications,
and biomedical survival studies.

\begin{example}[A Multivariate Gamma Model for Ordered Positive Data]
\label{EX:GammaExp}
Let $0<X_1<\cdots<X_d$ denote ordered positive observations. The general construction of \citet{balakrishnan2016multivariate}
begins with a baseline distribution  having density $f$ with support $(a, b)$ 
and survival function $\overline F$, and defines the transformed increments
$Y_1=-\log \overline F(X_1)$, $Y_j=-\log\frac{\overline F(X_j)}{\overline F(X_{j-1})}$, $j=2,\ldots,d$.
Assuming these increments are independently gamma distributed with shape parameters $\delta_1,\ldots,\delta_d >0$, 
the resulting joint density is
\begin{eqnarray}
f_{\boldsymbol{\theta}}(\boldsymbol{x})  =  \frac{(-\log \overline{F}(x_1))^{\delta_1-1}f(x_d)}{\prod_{j=1}^{d}\Gamma(\delta_j)} 
\prod_{j=1}^{d-1}\left(- \log \frac{\overline{F}(x_{j+1})}{\overline{F}(x_j)}\right)^{\delta_{j+1} -1 }
\frac{f(x_j)}{\overline{F}(x_j)}, 
\label{EQ:BR_model}
\end{eqnarray}
for $\boldsymbol{x} \in  \{ \boldsymbol{x}\in(a, b)^d : x_1 < x_2 < \cdots < x_d\}$, 
and the parameter vector $\btheta$ contains the gamma-shape parameters together with any unknown parameters of the baseline distribution.

For illustration, we restrict to an exponential baseline distribution with $f(u) = \lambda e^{-\lambda u}$ and  $(a, b) = (0, \infty)$,
where $\lambda>0$ denotes its (baseline)  hazard rate parameter. Thus, the joint model density \eqref{EQ:BR_model} simplifies to 
\begin{eqnarray}
	f_{\boldsymbol{\theta}}(\boldsymbol{x})  = e^{-\lambda x_d} \prod_{j=1}^{d}\frac{\lambda^{\delta_j} (x_j - x_{j-1})^{\delta_j-1}}{\Gamma(\delta_j)}, 
	\label{EQ:density_GammaExp}
\end{eqnarray}
with $x_0=0$ and parameter vector  $\boldsymbol{\theta}  = (\boldsymbol{\delta}^\top, \lambda)^\top\in\Theta = \mathbb{R}^{(d+1)+}$. 
The representation in \eqref{EQ:density_GammaExp} shows that the model is constructed through cumulative sums of 
independent gamma increments; we  refer to this particular distribution having density \eqref{EQ:density_GammaExp}
as the \textit{multivariate ordered Gamma distribution}.  Although the full likelihood is available, 
inference based on the joint density becomes increasingly cumbersome as $d$ grows.
Furthermore, verifying the adequacy of this joint model is quite hard for any given sample data,
although pairwise (ordered) marginals can be readily verified in practice.

A key advantage of the CDPD framework is that only low-dimensional components are required. 
In particular, each pairwise marginal $(X_j,X_k)$ are distributed as McKay's bivariate gamma distribution  
\citep{mckay1934sampling} having density  ($1\leq j < k \leq d$)
\begin{equation}
f_{\boldsymbol{\theta},{jk}}(\boldsymbol{x}_{jk})  = 	\frac{\lambda^{\delta_k^*}}{\Gamma(\delta_j^*)\Gamma(\delta_k^*-\delta_j^*)}
x_j^{\delta_j^* -1}(x_k - x_j)^{\delta_k^* - \delta_j^* -1}e^{-\lambda x_k}, \quad 0<x_j<x_k, 
\label{EQ:McKay_2density}
\end{equation}
where the updated parameters are $\delta_j^* = \delta_1 + \cdots + \delta_j$ and $\delta_k^* - \delta_j^* = \delta_{j+1}+\cdots+\delta_k$. 
The pairwise CL therefore involves only $K=\frac{d(d-1)}{2}$ bivariate gamma components. The corresponding MCDPDE minimizes
\begin{equation*}
\frac{1}{n}\sum\limits_{i=1}^n V_{\beta}(\boldsymbol{x}_i|{\boldsymbol{\theta}})  
= \frac{1}{n}\sum\limits_{i=1}^n \sum\limits_{1 \leq j < k\leq d} 
V_{jk,\beta}\left(\boldsymbol{x}_{jk}^{(i)}| \btheta\right),
\end{equation*}
where 
\begin{equation*}
V_{jk, \beta}\left(\boldsymbol{x}_{jk}|\btheta\right) 
= \lambda^{2\beta} C_{\beta}\left(\delta_j^*, \delta_k^* - \delta_j^*\right)- \frac{1+\beta}{\beta} f_{\boldsymbol{\theta},{jk}}^\beta(\boldsymbol{x}_{jk}) + \frac{1}{\beta}, \quad 0<x_j<x_k. 
\end{equation*}
for  all $1\leq j<k \leq d$, with
\begin{equation}
C_{\beta}(\delta_1, \delta_2) = \frac{\Gamma(\beta\delta_1-\beta+\delta_1)\Gamma(\beta\delta_2-\beta+\delta_2)}{
\{\Gamma(\delta_1)\Gamma(\delta_2)\}^{\beta+1}(1+\beta)^{\beta(\delta_1+\delta_2-2)+\delta_1 +\delta_2}}, \quad \delta_1, \delta_2 >0.
\label{EQ:McKay_C}
\end{equation}
This constant $C_{\beta}(\delta_1, \delta_2)$ is well-defined only when $(1+\beta)\delta_1>\beta$ and $(1+\beta)\delta_2>\beta$,
which are the natural integrability condition required by the DPD criterion. Thus the MCDPDE is defined only for tuning parameters 
$\beta > \max\{0, \frac{\delta_j^*}{1-\delta_j^*}, \frac{\delta_k^* - \delta_j^*}{1-\delta_k^* + \delta_j^*} : 1\leq j < k \leq  d\}$.

Thus, unlike the ordinary MDPDE based on the full $d$-dimensional density, 
the proposed estimator requires only evaluations of bivariate gamma densities 
and avoids high-dimensional integration while retaining robustness against atypical observations.
This example demonstrates that the CDPD framework applies beyond elliptical models and 
remains useful whenever low-dimensional component distributions are available even if the full likelihood is inconvenient.
The estimating equations, obtained by direct differentiation, are presented in Appendix \ref{APP:Algo_GammaExp}.
\qed 
\end{example}
A different source of intractability arises when a flexible dependence structure makes the full joint density costly to evaluate 
or differentiate in high dimensions, as is common for copula-based models with non-Gaussian dependence \citep{joe2014dependence}. 
Let $X_j$ have marginal distribution function  $F_{\eta_j, j}$ and density $f_{\eta_j, j}$, 
and suppose the joint law is generated by a copula $C_\phi$ with density $c_\phi$, as 
	\begin{equation*}
	f_{\btheta}(\bm{x}) = c_{\phi}\big(F_{\eta_1,1}(x_1),\ldots,F_{\eta_d,d}(x_d)\big)
	\prod_{j=1}^d f_{\eta_j,j}(x_j;\eta_j), \quad \btheta = (\eta_1,\ldots,\eta_d, \phi)^\top.
	\end{equation*}
For flexible dependence structures, including vine copulas and factor copulas, the full density may involve high-dimensional products 
or latent-variable integrations. Although the likelihood $f_{\btheta}$ itself may be available, evaluating the MDPDE requires
a $d$-dimensional integral $\int f_{\btheta}^{1+\beta}(\bm{x})d\bm{x}$ which is generally unavailable in closed form.

In contrast, the CDPD construction (and hence the MCDPDE) based on pairwise components requires only $K=d(d-1)/2$ bivariate integrals
of the form $\int f_{\btheta,jk}^{1+\beta}\, dx_j\, dx_k$, where 
\begin{equation*}
f_{\btheta,jk}(x_j,x_k) = c_{\phi_{jk}, jk}\big(F_{\eta_j,j}(x_j), F_{\eta_k, k}(x_k)\big) f_{\eta_j,j}(x_j) f_{\eta_k, k}(x_k), \quad 1 \le j < k \le d,
\end{equation*}
with $\phi_{jk}$ denoting the bivariate dependence parameter of the pair $(X_j,X_k)$, 
e.g., the corresponding entry of a vine copula, or a sub-correlation of an elliptical copula \citep{panagiotelis2012pair, aas2009pair}. 
These are two-dimensional integrals and can be evaluated easily using standard numerical quadrature.
Even without a closed form,  evaluating these $K = d(d-1)/2$ separate \emph{bivariate} integrals  
is vastly cheaper than the $d$-dimensional integration required for computing MDPDEs. 
Therefore, the CDPD provides a computationally scalable robust alternative for high-dimensional copula inference.

For the Gaussian-copula, $\phi_{jk} = \rho_{jk}$ and $c_{\rho, jk}(u,v)$ is the bivariate Gaussian copula density 
with correlation $\rho_{jk}$, so that $z_j = \Phi^{-1}(F_{\eta_j, j}(x_j))$ is standard Gaussian and the copula factor reduces exactly 
to the bivariate-Gaussian dependence structure of Example \ref{EX:MVN} evaluated at $(z_j,z_k)$. 
If the marginals are also Gaussian, $V_{jk,\beta}$ reduces 	exactly to \eqref{EQ:Vjk-gaussian}. 

\subsubsection{Discrete Multivariate models}

The CDPD framework  is equally valuable for discrete multivariate models with an intractable normalizing constant. The following example illustrates the CDPD for pairwise binary Markov random fields, connecting directly to the Ising-model construction of \cite{liu/etc:2025dpd}.

\begin{example}[Binary Markov Random Fields]\label{EX:BMRF}
Let $\bm{X} = (X_1,\ldots,X_d)^\top \in \{0,1\}^d$ follow an Ising-type model with joint density
\begin{equation*}
f_{\btheta}(\bm{x}) = \frac{1}{Z(\btheta)} \exp\left(\sum_{j=1}^d \alpha_j x_j + \sum_{j<k} \gamma_{jk} x_j x_k\right), 
\end{equation*}
where evaluating the normalizing constant $Z(\btheta)$ requires summing over $2^d$ configurations and is intractable for even moderate $d$. 
Following the pseudo-likelihood idea of \citet{besag:1974spatial}, we construct the composite model from the full conditional distributions
\begin{equation*}
p_j = P_{\btheta}(X_j = 1 \mid \bm{X}_{-j} = \bm{x}_{-j}) = \frac{\exp(\eta_j(\bm{x}_{-j}))}{1+\exp(\eta_j(\bm{x}_{-j}))}, 
\quad \eta_j(\bm{x}_{-j}) = \alpha_j + \sum_{k \ne j} \gamma_{jk} x_k,
\end{equation*}
where $\bm{X}_{-j}$ (or $\bm{x}_{-j}$) is the $(d-1)$-dimensional subvector of $\bm{X}$ (or $\bm{x}$) after dropping its $j$-th component $X_j$ (or $x_j$).
The corresponding unweighted ($w_j \equiv 1$) MCDPDE objective is  
$\frac{1}{n}\sum\limits_{i=1}^n \sum\limits_{j=1}^d V_{j,\beta}(\bm{x}_i \mid \btheta)$, where 
\begin{equation}
V_{j,\beta}(\bm{x} \mid \btheta) = p_j^{1+\beta} + (1-p_j)^{1+\beta} 
- \left(1+\frac1\beta\right) p_j^{\beta x_j} (1-p_j)^{\beta(1-x_j)} + 	\frac1\beta,
\label{eq:Vj-ising}
\end{equation}
summing over the two-point support of $X_j$ in place of integrating; 
this is exactly the DPD objective used for robust binary/logistic regression.
Thus, the CDPD estimator avoids computation of the global normalizing constant $Z(\btheta)$, 
and inherits the robustness properties of DPD-based estimation under logistic models as illustrated in \citet{ghosh2016robust}.
The estimating equations reduce to a weighted pseudo-likelihood score equation by using $\bm{u}_{\btheta, j}(x_j, \bm{x}_{-j}) = \nabla_{\btheta} \log P_{\btheta}(X_j = x_j \mid	\bm{X}_{-j} = \bm{x}_{-j}) = (x_j-p_j)\begin{pmatrix} 1 \\
\bm{x}_{-j}
\end{pmatrix}$,
and can be solved using iterative reweighted procedures.
This example generalizes the Ising-specific construction of \cite{liu/etc:2025dpd} as a special case (with $n=1$) of the present framework. 
\qed
\end{example}

The same construction as in Example \ref{EX:BMRF} also extends to composite pseudo-likelihoods and CDPD for multivariate count data.
For example, under a Poisson Markov random field, it can be achieved with $p_j$ replaced by 
the corresponding conditional Poisson probabilities  and the two-point sum in \eqref{eq:Vj-ising} replaced by 
a  sum over the count support, which is still low-dimensional, and hence tractable.

\subsection{Handling Missing Data}
\label{SEC:missing}

An additional practical advantage of marginal-component CDPD construction is its natural accommodation of incomplete observations. 
Suppose that for observation $\boldsymbol{x}_i$, only the components indexed by $\mathcal K_i\subseteq\{1,\ldots,K\}$ 
are completely observed. For example, under the pairwise marginal construction, $\mathcal K_i$ consists of all pairs $(j,k)$ 
for which both variables are  observed (not missing). 
The MCDPDE objective and estimating equation in \eqref{EQ:MCDPDE}--\eqref{EQ:MCDPDE_EstEqn} extend immediately 
to this setting by summing only over the observed components for each observation, so that we get
\begin{equation}
\widehat{\btheta}_{n,\beta} = \arg\min_{\btheta\in\Theta} \frac1n \sum_{i=1}^n 
								\sum_{k \in \mathcal K_i} w_k\, V_{k,\beta}(\bm{x}_i \mid\btheta).
\label{EQ:MCDPDE-missing}
\end{equation}

Under missing completely at random (MCAR), each observed component contributes an unbiased estimating function, 
so the theoretical properties of the MCDPDE continue to hold when the complete collection of components is replaced by the observed subset. 
The CDPD framework may also provide some protection in settings where missingness is driven by the extremeness of the underlying values, 
a form of nonignorable missingness that is more naturally viewed as contamination than as a conventional MCAR mechanism \citep{little2019statistical}. 
Because the MCDPDE downweights observations that are poorly supported by the fitted model, 
the same mechanism that yields robustness to case-wise and cell-wise contamination is expected to attenuate, though not eliminate, 
the resulting bias. A formal analysis of this setting is beyond the scope of the present work and is left for future research.

\section{Theoretical Properties}
\label{SEC:theory}

\subsection{Notation and Assumptions}

Throughout this section, $\btheta_\beta^* = \bm{T}_\beta(G) = \arg\min_{\btheta\in \Theta} \E_G[V_\beta(\bm{X} \mid \btheta)]$ 
denotes the minimum CDPD functional evaluated at the true distribution $G$, which is not required to belong to the parametric family. 
The corresponding empirical estimator based on $\mathcal X_n=\{\bm x_1,\ldots,\bm x_n\}$ is denoted by $\widehat{\btheta}_{n,\beta}$.
Recall that $\bm{u}_{\btheta,k}(\bm{x}) = \nabla_{\btheta} \log f_{\btheta,k}(\bm{x})$ is the component score function, 
and define the component curvature matrix $\bm{I}_{\btheta,k}(\bm{x}) = -\nabla_{\btheta}^2 \log f_{\btheta,k}(x)$. 
The dimensions of the observation vector and parameter vector are denoted by $d$ and $p$, respectively. 
The index $k=1,\ldots,K$ refers generically to any collection of marginal or conditional components used in constructing the CL.
Further, define 
\begin{equation}\label{EQ:J-K}
	\boldsymbol{J}_\beta(\btheta) = \E_G\left[\nabla \boldsymbol{\psi}_\beta(\boldsymbol{X}|\boldsymbol{\theta})\right],
\mbox{ and }
	\boldsymbol{K}_\beta(\btheta
	= \E_G[\boldsymbol{\psi}_\beta(\boldsymbol{X}|\boldsymbol{\theta})\boldsymbol{\psi}_\beta^\top(\boldsymbol{X}|\boldsymbol{\theta})],
\end{equation}
where $\bm{\psi}_\beta$ is as defined in \eqref{EQ:Psi_theta}. The following assumptions provide sufficient conditions for consistency and asymptotic normality of the MCDPDE. 
They extend the standard Cram\'er-type  regularity conditions used for the ordinary MDPDE \citep{Basu1998,Basu2011,basu2026}, 
general extremum estimators \citep{newey1994large}, and smooth $M$-estimators \citep{Huber2009robust,van2000asymptotic} 
to the composite setting.

\begin{assumption}[Identifiability and existence]\label{ASS:Asmp1}
	For each $k$, the support of $f_{\btheta,k}$ does not depend on $\btheta$, and is the same as the support of $g_k$.
	The CL is identifiable with respect to $\btheta$, i.e., $f_{\btheta_1, k} \equiv f_{\btheta_2, k}$ for all $k=1, \ldots, K$,
	implies $\btheta_1 = \btheta_2$. Further, the minimizer $\btheta_\beta^* = \bm{T}_\beta(G)$ exists, is unique, and lies in the interior of the parameter space $\Theta$.
\end{assumption}

\begin{assumption}[Smoothness]\label{ASS:Asmp2}
	For almost all $\bm{x}$ and each component $k = 1, \ldots, K$, the mapping $\btheta \mapsto f_{\theta,k}(\bm{x})$ 
	is thrice continuously differentiable on an open neighborhood $\mathcal N$ of $\btheta_\beta^*$.
\end{assumption}

\begin{assumption}[Differentiation under the integral sign]\label{ASS:Asmp3}
For each $k$, the quantities $\int f_{\btheta,k}^{1+\beta}(\bm{x})\,d\bm{x}$ and $\int f_{\btheta,k}^{\beta}(\bm{x})g_k(\bm{x})\,d\bm{x}$ 
may be differentiated in $\btheta$ under the integral sign, up to third order, uniformly for $\btheta \in \mathcal N$.
\end{assumption}

\begin{assumption}[Envelope conditions]\label{ASS:Asmp4}
	There exist $G$-integrable envelop functions dominating, uniformly over	$\mathcal N$, the third order partial derivatives of the objective $V_\beta(\bm{x}\mid\btheta)$. Also, the variability matrix $\bm{K}_\beta(\btheta_\beta^*)$ exists finitely.
\end{assumption}

\begin{assumption}[Non-singularity]\label{ASS:Asmp5}
The sensitivity matrix $\bm{J}_\beta(\btheta_\beta^*)$ exists and is positive definite.
\end{assumption}

These assumptions are formulated entirely in terms of the component densities $f_{\btheta,k}$ 
rather than the full joint density $f_{\btheta}$. In particular, they do not require the true distribution $G$ 
to belong to the parametric family or even to satisfy the complete joint model specification. 
Only the component distributions entering the composite construction need to be sufficiently regular. 
This distinction is important for complex problems, where the full dependence structure may be difficult to model or evaluate, 
whereas lower-dimensional marginals or conditionals remain tractable.

\subsection{Asymptotic Properties}
\label{SEC:asymptotics}

We now establish the large-sample properties of the MCDPDE, assuming the dimensions $d$ and $p$ to be fixed, while $n\rightarrow\infty$. 
The results are developed under possible model misspecification, namely when the true distribution $G$ does not necessarily belong to 
the assumed family. This follows the general framework of extremum estimators \citep{newey1994large,van2000asymptotic} 
and extends the corresponding results for ordinary MDPDEs \citep{Basu1998,Basu2011,basu2026} to the composite setting.
In particular, the following theorem shows that replacing the full density by a collection of lower-dimensional components 
preserves consistency and asymptotic normality, provided the component-based objective satisfies suitable regularity assumptions.
Please refer to Appendix \ref{APP:Proof_Asymp} for its proof.

\begin{theorem}\label{THM:MCDPDE_asymp}
Under Assumptions \ref{ASS:Asmp1}--\ref{ASS:Asmp5} at any given $\beta\geq 0$, 
the MCDPDE  $\widehat{\boldsymbol{\theta}}_{n, \beta}$ exists with probability tending to one, 
$\widehat{\btheta}_{n,\beta} \xrightarrow{\mathcal{P}} \btheta_\beta^*$ and 
\begin{equation*}
\sqrt{n}\left(\widehat{\boldsymbol{\theta}}_{n, \beta} - \boldsymbol{\theta}_{\beta}^*\right) \mathop{\rightarrow}^\mathcal{D} 
N_p\left(\boldsymbol{0}_p, 
\boldsymbol{J}_\beta^{-1}({\btheta}_\beta^*)\boldsymbol{K}_\beta({\btheta}_\beta^*)\boldsymbol{J}_\beta^{-1}({\btheta}_\beta^*)\right), 
\qquad \mbox{ as } n \to \infty.
\end{equation*}
\end{theorem}
We next simplify the matrices $\bm{J}_\beta$ and $\bm{K}_\beta$, from \eqref{EQ:J-K} under the practically important case 
where each composite component is correctly specified, i.e.,  $g_k = f_{{\btheta}_0,k}$ for every $k = 1, \ldots, K$, 
and some ${\btheta}_0 \in \Theta$. This assumption is weaker than requiring $G=F_{{\btheta}_0}$, 
because only the lower-dimensional distributions entering the composite construction need to be correct. 
Under this condition, we still have ${\btheta}_\beta^* = {\btheta}_0$ for every $\beta\geq0$, and the MCDPDE is Fisher consistent.
Further, it is straightforward to see that,  under such component-wise correct specification, we get  
\begin{equation}
\bm{J}_\beta({\btheta}_0) = \sum_{k=1}^K w_k 
\int \bm{u}_{{\btheta}_0,k}(\bm{x})\,\bm{u}_{{\btheta}_0,k}^\top(\bm{x})\, f_{{\btheta}_0,k}^{1+\beta}(\bm{x})\,d\bm{x},
\label{EQ:Jbeta-model}
\end{equation}
and
\begin{equation}
\bm{K}_\beta({\btheta}_0) = \sum_{k=1}^K \sum_{l=1}^K w_k w_l \left[
\bm{C}_{kl}({\btheta}_0) - \bm{m}_k({\btheta}_0)\, \bm{m}_l^\top({\btheta}_0)\right],
\label{EQ:Kbeta-model}
\end{equation}
where
\begin{equation*}
  \bm{m}_k(\btheta) = \int \bm{u}_{\btheta,k}(\bm{x}) f_{\btheta,k}^{1+\beta}(\bm{x}) d\bm{x}
\end{equation*}  
and
\begin{equation*}
\bm{C}_{kl}(\btheta) = \E_G\left[
  \bm{u}_{\btheta,k}(X) \bm{u}_{\btheta,l}^\top(\bm{X}) f_{\btheta,k}^\beta(\bm{X})f_{\theta,l}^\beta(\bm{X})\right],
\end{equation*}  
the last expectation being taken under the joint distribution  of $\bm{X}$ implied by $G$ restricted to the (typically low-dimensional) 
set of variables spanned by $\mathcal{A}_k \cup \mathcal{A}_l$.
The dependence between different composite components is therefore fully captured through the cross-covariance terms $\bm C_{kl}$.

The following corollaries further confirm that the MCDPDEs include both the MCLE (at $\beta=0$) and the ordinary MDPDE (at $K=1$) 
as special cases of a single asymptotic theory, developed in Theorem \ref{THM:MCDPDE_asymp}.
\begin{corollary}[Full likelihood, $K=1$]\label{COR:Asymp_K1}
When $K=1$ and $w_1=1$, the CL reduces to the full 	likelihood, $\mathcal{L}_C = f_{\btheta}$, 
and the MCDPDE coincides with the ordinary MDPDE.
In this case, \eqref{EQ:Jbeta-model}--\eqref{EQ:Kbeta-model} reduce to 
$\bm{J}_\beta({\btheta}_0) = \int \bm{u}_{{\btheta}_0}(\bm{x})u_{{\btheta}_0}^\top(\bm{x}) f_{{\btheta}_0}^{1+\beta}(\bm{x})d\bm{x}$ and
\begin{equation*}
\bm{K}_\beta({\btheta}_0) =	\int \bm{u}_{{\btheta}_0}(\bm{x}) \bm{u}_{{\btheta}_0}^\top(\bm{x}) f_{{\btheta}_0}^{1+2\beta}(\bm{x})d\bm{x} 
-	\bm{\xi}_\beta({\btheta}_0)\bm{\xi}_\beta^\top({\btheta}_0),
\end{equation*}
with $\bm{\xi}_\beta({\btheta}_0)= \int \bm{u}_{{\btheta}_0}(\bm{x}) f_{{\btheta}_0}^{1+\beta}(\bm{x})d\bm{x}$, recovering exactly the asymptotic variance of the ordinary MDPDE at the model \citep{Basu1998,Basu2011,basu2026}.
\end{corollary}

\begin{corollary}[$\beta = 0$]\label{COR:Asymp_beta0}
As $\beta \downarrow 0$, $\bm{m}_k({\btheta}_0) \to 0$ for every $k$, since the ordinary composite score has $G$-mean zero at the model.
Then, the MCDPDE estimating equation becomes the ordinary composite score equation, 
and hence \eqref{EQ:Jbeta-model}--\eqref{EQ:Kbeta-model} reduce to \\
$\bm{J}_0({\btheta}_0)= \sum_k w_k \int \bm{u}_{{\btheta}_0,k}(\bm{x}) \bm{u}_{{\btheta}_0,k}^\top(\bm{x}) f_{{\btheta}_0,k}(\bm{x})d\bm{x} 
= - \E_g[\nabla_{\btheta}^2 \log \mathcal{L}_C({\btheta}_0 \mid \bm{X})]$ and 
$\bm{K}_0({\btheta}_0) = \mathrm{Var}_g[\nabla_{\btheta} \log \mathcal{L}_C({\btheta}_0 \mid \bm{X})]$, 
exactly recovering the classical Godambe sandwich covariance matrix $\boldsymbol{I}_G(\boldsymbol{\theta})$ of the MCLE 
as defined in Section \ref{SEC:Model}.
\end{corollary}

\begin{figure}[!b]
	\centering
	\subfloat[Independent variables ($\rho=0$)]{
		\includegraphics[width=0.4\linewidth]{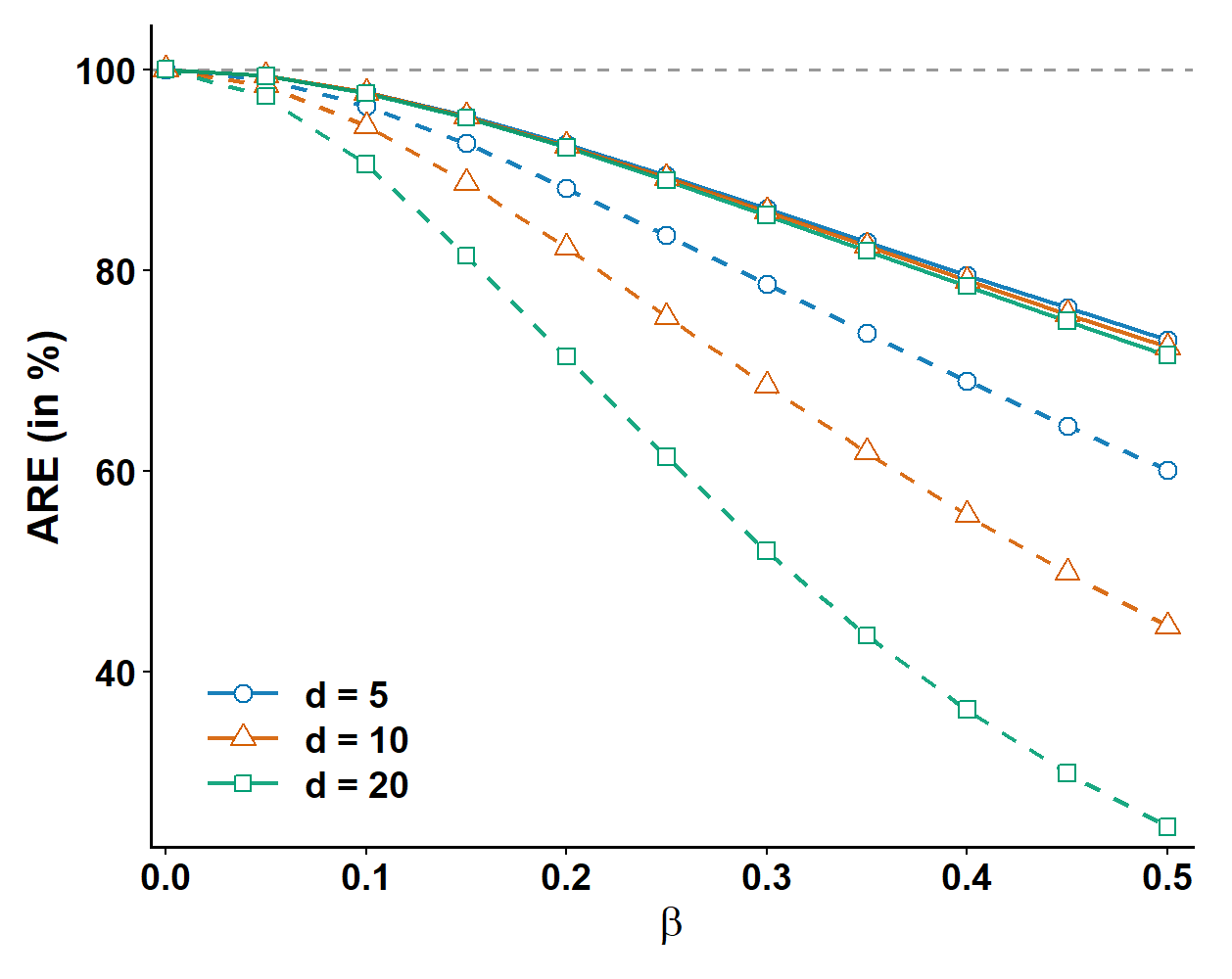}
		\label{FIG:ARE_MVN0}
	}
	\subfloat[Dependent variables ($\rho=0.5$)]{
		\includegraphics[width=0.4\linewidth]{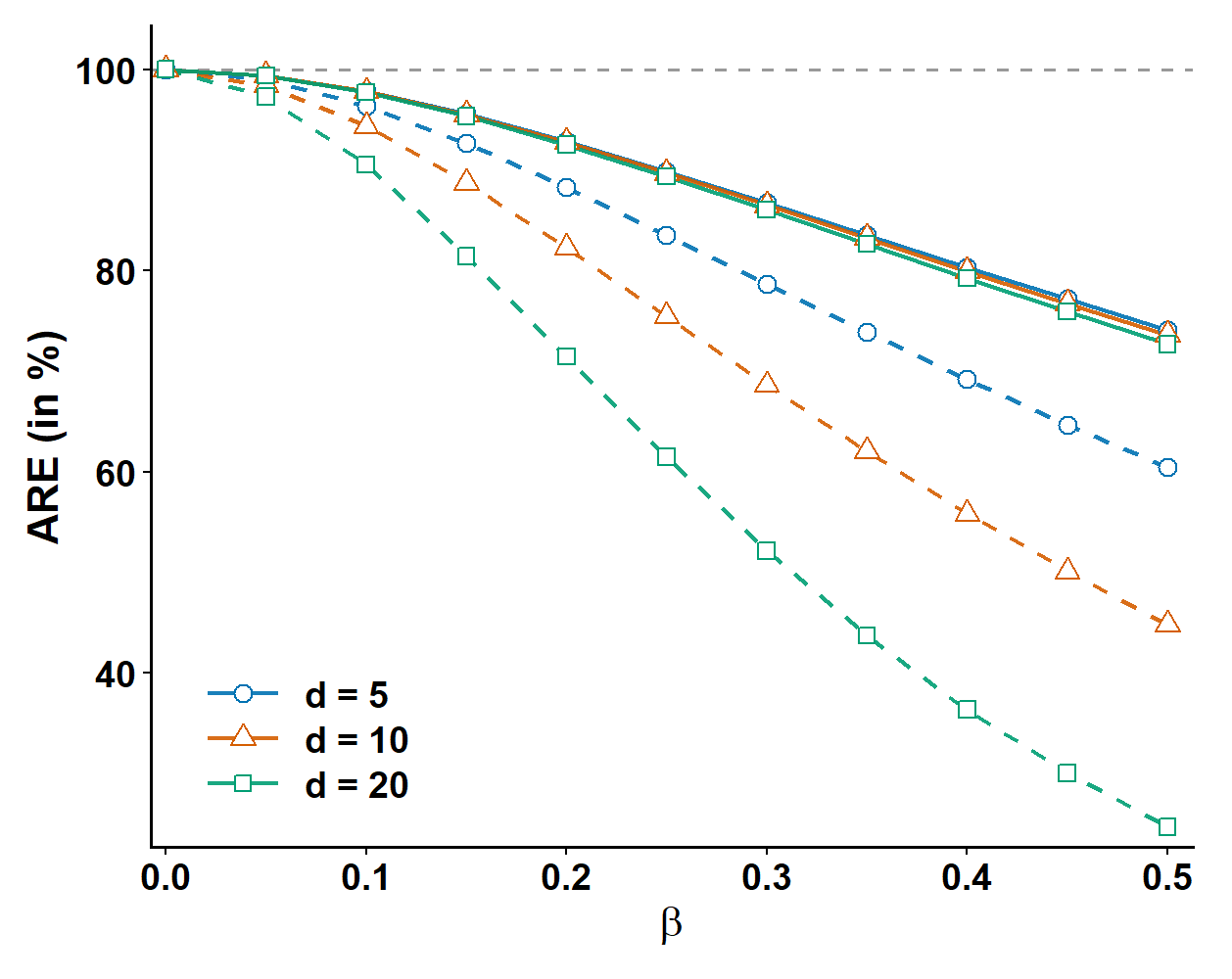}
		\label{FIG:ARE_MVN05}
	}
	\caption{AREs of the MCDPDEs (solid line) and MDPDEs (dashed line), relative to the MLE, under $d$-dimensional  Gaussian model at the true parameter value
		$\mu_j = 0, \sigma_j^2=1$ and $\rho_{jk}=\rho$ for all $j, k = 1, \ldots, d$.}
	\label{FIG:ARE_MVN}
\end{figure}

\begin{example}[Continuation of Example \ref{EX:MVN}]\label{EX:MVN-efficiency}
	Let us reconsider the pairwise Gaussian CL of Example \ref{EX:MVN}	with unweighted components ($w_{jk} \equiv 1$), 
	and compute the asymptotic variance of the MCDPDE following Theorem \ref{THM:MCDPDE_asymp}.
	Detailed derivations are given in Appendix \ref{APP:Algo_MVN}.
	Figure \ref{FIG:ARE_MVN} compares the MCDPDE and the MDPDEs in terms of their asymptotic relative efficiency (ARE),  defined as 
	the ratio of the MLE's total variance (the trace of its asymptotic covariance matrix) to that of the corresponding estimator. 
	For Gaussian models, the MCLE coincides with the MLE and therefore does not suffer any efficiency loss (i.e., its ARE is 100\%). 
	The MCDPDE exhibits a gradual efficiency loss as $\beta$ increases, but maintains substantially higher AREs than the MDPDE 
	over a broad range of $\beta$, particularly for larger data dimension $d$. 
	Moreover, while the efficiency loss of the MDPDE increases significantly with  $d$, that of the MCDPDE remains nearly stable. 
	These results demonstrate the more favorable robustness-efficiency 	trade-off of the MCDPDE relative to the standard MDPDE.
	\qed
\end{example}

\subsection{Robustness: Influence Functions and Sensitivity}
\label{SEC:IF}

We now study the local robustness properties of the MCDPDE through the influence function (IF) 
of its associated statistical functional \citep{hampel1974influence,hampel2011robust}. 
The IF measures the first-order effect of an infinitesimal contamination at a point $\bm{y}$ on the estimator functional.
For the MCDPDF $\boldsymbol{\theta}_\beta^\ast = \boldsymbol{T}_\beta(G)$,
the influence function at $G$ is defined as the Gateaux derivative \citep{hampel2011robust}
\begin{equation*}
\mbox{IF}(\bm{y}; \bm{T}_\beta, G) = \lim\limits_{\epsilon\rightarrow 0}\frac{\bm{T}_\beta(G_\epsilon) - \bm{T}_\beta(G)}{\epsilon}
= \frac{\partial}{\partial\epsilon}\bm{T}_\beta(G_\epsilon)\bigg|_{\epsilon=0},
\qquad \bm{y}\in\mathbb{R}^d,
\end{equation*}
where $G_\epsilon = (1-\epsilon)G + \epsilon \wedge_{\bm{y}}$ with $\epsilon$ denoting the contamination proportion and 
$\wedge_{\bm{y}}$ denoting the degenerate distribution at the point $\bm{y}$. 

In order to explicitly compute this IF, let us recall that the MCDPDF  $\boldsymbol{\theta}_\beta^\ast = \boldsymbol{T}_\beta(G)$
satisfies the estimating equation $\int \boldsymbol{\psi}_\beta(\boldsymbol{x}|\boldsymbol{\theta})dG(\bm{x}) = \boldsymbol{0}_p$.
Thus, $\bm{T}_\beta(G_\epsilon)$ satisfies
\begin{equation*}
\int \boldsymbol{\psi}_\beta(\boldsymbol{x}|\bm{T}_\beta(G_\epsilon))dG_\epsilon(\bm{x}) = \boldsymbol{0}_p.
\end{equation*}
Differentiating this equation with respect to $\epsilon$ at $\epsilon = 0$ and simplifying suitably, 
we get the IF of the MCDPDF as presented in the following theorem. 
The subsequent corollary presents the simplification under the assumption of component-wise correct specification.

\begin{theorem} \label{THM:IF}
Suppose Assumptions \ref{ASS:Asmp1}--\ref{ASS:Asmp5} hold and the functional $\boldsymbol T_\beta(G)$ is Gateaux differentiable at $G$.
Then, the influence function of the MCDPDF at $G$ is given by 
\begin{equation}
{\rm IF}(\bm{y}; \bm{T}_\beta, G) = -\bm{J}_\beta^{-1}({\btheta}_\beta^*)\, \bm{\psi}_\beta(\bm{y} \mid {\btheta}_\beta^*),
\label{EQ:IF-fixed}
\end{equation}
where $\boldsymbol{J}_\beta(\btheta) = \E_G\left[\nabla \boldsymbol{\psi}_\beta(\boldsymbol{X}|\boldsymbol{\theta})\right]$ 
as defined in \eqref{EQ:J-K}.
\end{theorem}

\begin{corollary} \label{CORR:IF}
Under the assumptions of Theorem \ref{THM:IF}, and under component-wise correct specification with parameter $\boldsymbol\theta_0$,
the	influence function of the MCDPDE functional $\bm{T}_\beta$ at $F_{{\btheta}_0}$ simplifies to
\begin{equation}
{\rm IF}(\bm{y}; \bm{T}_\beta, F_{{\btheta}_0}) = - \bm{J}_\beta^{-1}({\btheta}_0)\,\bm{\psi}_\beta(\bm{y}\mid{\btheta}_0),
\label{EQ:IF-fixed-model}
\end{equation}
where $\bm{J}_\beta$ is now as given in \eqref{EQ:Jbeta-model}.
\end{corollary}

The above expression immediately shows that the robustness of the MCDPDE is governed by the boundedness of the weighted composite score
$\boldsymbol\psi_\beta(\bm y|\boldsymbol\theta)$. For $\beta>0$, each component contribution contains the factor 
$\boldsymbol u_{\boldsymbol\theta,k}(\bm y) f_{\boldsymbol\theta,k}^{\beta}(\bm y)$, and hence a sufficient condition for B-robustness is
\begin{equation*}
\sup_{\bm y}
\left\|
\boldsymbol u_{\boldsymbol\theta,k}(\bm y)
f_{\boldsymbol\theta,k}^{\beta}(\bm y)
\right\|
<\infty ,
\qquad k=1,\ldots,K.
\end{equation*}
Under this condition, which holds for most common statistical models, 
the IF in \eqref{EQ:IF-fixed} is bounded, establishing local robustness of the MCDPDE. 
This is precisely the robustness mechanism of the ordinary MDPDE, where the density-power factor downweights observations 
having low model density \citep{Basu1998,Basu2011,basu2026}.

The boundedness of the IF establishes local robustness, but it is also important to quantify the extent of this robustness. 
Following \citet{hampel2011robust}, we consider the gross-error sensitivity (GES) of the MCDPDE functional given by
\begin{equation*}
\gamma^*(\bm T_\beta,G) = \sup_{\bm{y}\in\mathcal{X}}\left|\mbox{IF}(\bm{y};\bm{T}_\beta,G) \right| 
= \sup_{\bm{y}\in\mathcal{X}}\left|\bm{J}^{-1}_\beta({\btheta}_\beta^\ast)\bm{\psi}_\beta(\bm{y}\mid{\btheta}_\beta^\ast)\right|
\propto \sup_{\bm{y}\in\mathcal{X}}\left|\bm{u}_{\btheta,k}(\bm{y}) f_{\btheta,k}^{\beta}(\bm{y})\right|.
\end{equation*}
For $\beta>0$, the multiplicative density factor $f_{\btheta,k}^{\beta}(\bm{y})$ shrinks the contribution of observations lying 
in regions of low model density. Consequently, even when the component score $\bm{u}_{\btheta,k}(\bm{y})$ diverges in the tails, 
the product $\bm{u}_{\btheta,k}(\bm{y})f_{\btheta,k}^{\beta}(\bm{y})$ may remain bounded. 
The tuning parameter $\beta$ therefore controls the magnitude of the gross-error sensitivity, 
yielding a finite gross-error sensitivity under mild tail conditions which further decreases with increasing values of $\beta>0$.
This is precisely the mechanism responsible for the local robustness of DPD-based procedures \citep{Basu1998,Basu2011,basu2026}.

In contrast, for $\beta=0$, the density attenuation disappears and 
$\boldsymbol\psi_0(\bm y|\boldsymbol\theta_0)= -\sum_{k=1}^{K} w_k \boldsymbol u_{\boldsymbol\theta_0,k}(\bm y)$,
which is the ordinary CL score. Since score functions are unbounded in most regular models, 
the corresponding MCLE has unbounded IF and infinite gross-error sensitivity, recovering the well-known non-robustness of the CMLE.

\subsection{Qualitative Robustness}
\label{SEC}

Influence function boundedness characterizes infinitesimal robustness, but robustness with respect to finite perturbations of 
the underlying distribution can be captured by the notion of qualitative robustness introduced by \citet{hampel1971general}.
A statistical functional $\bm{T}$ is said to be qualitatively robust at $G$ if for every $\epsilon>0$, 
there exists a neighborhood $\mathcal{N}(G)$ of $G$ under the weak topology such that $G'\in\mathcal{N}(G)$ 
implies that the Prokhorov distance between the distributions of $\bm{T}_n = \bm{T}(G_n)$ under $G'$ and $G$ are less than $\epsilon$
for all sufficiently large sample sizes. In other words, small changes in the generating distribution lead to only small changes 
in the sampling distribution of the estimator for all sufficiently large sample sizes.

For $\beta>0$, qualitative robustness follows from the general theory of robust
$M$-functionals \citep{hampel2011robust,Huber2009robust}. Under Assumption
\ref{ASS:Asmp1}, the population objective function  $Q_\beta(G,\btheta) = \E_G\!\left[V_\beta(\bm X\mid\btheta)\right]$
is well defined and possesses a unique minimizer $\btheta_\beta^\ast=\bm T_\beta(G)$.
Moreover, Assumption \ref{ASS:Asmp4} guarantees the existence of $G$-integrable envelope functions dominating
$V_\beta(\bm X\mid\btheta)$ and its first two derivatives uniformly over a neighborhood of $\btheta_\beta^\ast$.
Combined with the continuity of $V_\beta(\bm x\mid\btheta)$ in $\btheta$, these conditions imply that the mappings
$G\longmapsto \E_G[V_\beta(\bm X\mid\btheta)]$  and $G\longmapsto \E_G[\bm\psi_\beta(\bm X\mid\btheta)]$ 
are continuous under weak convergence (see, e.g., \citealp[Chapter~3]{Huber2009robust}). Further, since the sensitivity matrix 
$\bm J_\beta(\btheta_\beta^\ast) = \E_G\!\left[ \nabla_{\btheta} \bm\psi_\beta(\bm X\mid\btheta_\beta^\ast)\right]$
is finite and nonsingular by Assumption \ref{ASS:Asmp5}, the implicit function theorem \citep{Rudin1976} yields an open neighborhood
$\mathcal N(G)$ of $G$ on which the estimating equation $\E_{G'}[\bm\psi_\beta(\bm X\mid\btheta)]=\bm0$
admits a unique continuously differentiable solution $\btheta=\bm T_\beta(G')$.
Consequently, the MCDPDF $\bm T_\beta$ is weakly continuous at $G$, so that $G_n \Rightarrow G$ implies 
$\bm T_\beta(G_n)\rightarrow\bm T_\beta(G)$.
Hence, the sequence of MCDPDEs is qualitatively robust for every $\beta>0$.

In contrast, the above argument breaks down when $\beta=0$. 
The weak-continuity result established for $\beta>0$ ultimately relies on the Portmanteau characterization of weak convergence,
which states that $G_n\Rightarrow G$ implies $\E_{G_n}[h(\bm{X})]\to \E_G[h(\bm X)]$ for every bounded continuous $h$.
Assumption \ref{ASS:Asmp4} extends this guarantee to $\bm\psi_\beta$ only when 
$\bm\psi_\beta(\bm{x}\mid\btheta)$ remains \emph{bounded} in $\bm x$. 
This is possible for $\beta>0$, because the density-power weight $f_{\btheta,k}^\beta(\bm{X})$ multiplying each component score 
in $\bm\psi_\beta(\bm X\mid\btheta)$ decays fast enough to keep $\bm\psi_\beta(\cdot\mid\btheta)$ itself bounded in $\bm x$
for the regular parametric models covered by our assumptions. 
Thus, observations far from the model are automatically downweighted before they can affect $\E_G[\bm \psi_\beta(\bm X\mid\btheta)]$. 
For $\beta=0$, however, $\bm\psi_0(\bm X\mid\btheta)$ reduces to the ordinary composite score, 
which is generically \emph{unbounded} (e.g., linear in $x$ for a location parameter), and no such downweighting occurs. 
As a result, the mapping $G\mapsto \E_G[\bm\psi_0(\bm X\mid\btheta)]$ is not weakly continuous in general. 
Indeed, a vanishing fraction of contamination escaping to infinity can permanently alter the expectation of the score 
despite weak convergence of the underlying distributions, explaining why the MCLE, like the ordinary MLE, is not qualitatively robust
\citep{hampel2011robust,Huber2009robust}.

This result complements the IF analysis. A bounded IF ensures protection against infinitesimal gross errors, 
whereas qualitative robustness guarantees stability under finite weak perturbations of the data-generating distribution. 
Together, these properties provide a robustness characterization of the MCDPDE analogous to the classical theory for the MDPDEs 
\citep{Basu1998,Basu2011,basu2026}.

\subsection{Estimating Standard Errors in practice}

The asymptotic theory developed in Section \ref{SEC:asymptotics} provides the limiting distribution of the MCDPDE. 
For practical inference, however, this result must be supplemented by a consistent estimator of its asymptotic covariance matrix
given in Theorem \ref{THM:MCDPDE_asymp}. Then, standard errors of individual parameters can be obtained from 
the square roots of the diagonal elements of  such an estimator of the asymptotic covariance matrix.

There are two practically relevant situations. First, if the complete joint distribution of $\bm X$ is correctly specified, 
then both $\boldsymbol{J}_\beta$ and $\boldsymbol{K}_\beta$ can be evaluated from their population expressions given earlier. 
Replacing $\btheta_\beta^*$ by the MCDPDE $\widehat{\btheta}_{n,\beta}$ yields the consistent plug-in estimator
of the asymptotic covariance matrix of MCDPDE given by
\begin{equation*}
\widehat{\boldsymbol{\Sigma}}_\beta = \boldsymbol{J}_\beta^{-1}(\widehat{\btheta}_{n,\beta}) 
\boldsymbol{K}_\beta(\widehat{\btheta}_{n,\beta})\boldsymbol{J}_\beta^{-T}(\widehat{\btheta}_{n,\beta}).
\end{equation*}

Second, under component-wise correct specification, the matrix $\boldsymbol{J}_\beta$ continues to have the closed form 
given in Equation \eqref{EQ:Jbeta-model}. However, $\boldsymbol{K}_\beta$ also involves cross-component covariance terms
$\bm C_{kl}({\btheta}_0)$, which depend on the unknown dependence structure linking the components. 
Therefore, unlike $\boldsymbol{J}_\beta$, the full $\boldsymbol{K}_\beta$ cannot generally be recovered from the component models alone.
In such cases, when the joint distribution may itself be unavailable or computationally intractable, 
we recommend using the  model-robust estimators of $\bm{J}_\beta$ and $\bm{K}_\beta$ as given by
\begin{equation}
\widehat{\bm{J}}_\beta = \frac1n \sum_{i=1}^n \nabla_{\btheta} \bm{\psi}_\beta(\bm{x}_i\mid \widehat{\btheta}_{n,\beta}), \qquad 
\widehat{\bm{K}}_\beta = \frac1n \sum_{i=1}^n \bm{\psi}_\beta(\bm{x}_i \mid \widehat{\btheta}_{n,\beta})
\bm{\psi}_\beta^\top(\bm{x}_i \mid \widehat{\btheta}_{n,\beta}),
	\label{EQ:empirical-sandwich}
\end{equation}
leading to the empirical sandwich estimator $\widehat{\boldsymbol\Sigma}_\beta = \widehat{\boldsymbol J}_\beta^{-1}
\widehat{\boldsymbol K}_\beta \widehat{\boldsymbol J}_\beta^{-1}$ of the asymptotic variance in Theorem \ref{THM:MCDPDE_asymp}. This estimator requires no knowledge of the unknown dependence among components and 
is therefore the natural default estimator for the MCDPDE's asymptotic variances and standard errors.
Consistency of these estimators in \eqref{EQ:empirical-sandwich} for $\bm{J}_\beta(\btheta_\beta^*)$ and $\bm{K}_\beta(\btheta_\beta^*)$ 
follows from the laws of large numbers together with continuity of $\nabla_{\btheta}\bm{\psi}_\beta$ and $\bm{\psi}_\beta$, 
and consistency of $\widehat{\btheta}_{n,\beta}$. 
The following proposition presents these consistency results under one additional assumption, 
while the detailed proof is deferred to Appendix \ref{APP:Proof-Varest}.

\begin{assumption}	
\label{ASS:Asmp6}
There exist $G$-integrable envelop functions dominating, uniformly over	$\mathcal N$, 
the first and second order partial derivatives of the objective $V_\beta(\bm{x}\mid\btheta)$, and 
\begin{equation*}
  \E_G \left[\sup\limits_{\btheta\in\mathcal N}\|	\bm{\psi}_\beta(\bm X|\btheta)	\|^2\right]<\infty.
\end{equation*}
\end{assumption}

\begin{proposition}[Consistency of the empirical sandwich estimator]
\label{PROP:sandwich}
Under Assumptions \ref{ASS:Asmp1}--\ref{ASS:Asmp6},
$\widehat{\boldsymbol J}_\beta\xrightarrow{\mathcal{P}}	\boldsymbol J_\beta(\btheta_\beta^*)$ and
$\widehat{\boldsymbol K}_\beta\xrightarrow{\mathcal{P}}	\boldsymbol K_\beta(\btheta_\beta^*)$.	
Consequently, $\widehat{\boldsymbol\Sigma}_\beta \xrightarrow{\mathcal{P}} 	\boldsymbol\Sigma_\beta$ as $n\rightarrow\infty$.
\end{proposition}

This result extends the classical Godambe information approach for CL inference \citep{varin/etc:2011} to the robust MCDPDE setting. 
Importantly, it does not require analytical expressions for the cross-component covariance terms $\bm{C}_{kl}$,
and therefore remains applicable even when the underlying joint distribution is unavailable.
Detailed discussions on computational aspects for two representative models are provided in Appendix \ref{APP:Computation}.

\begin{remark}[Choice of the tuning parameter $\beta$]	
The tuning parameter $\beta$ controls the trade-off between robustness and statistical efficiency of the CMDPDE, 
exactly as in the ordinary MDPDE \citep{Basu1998,Basu2011,basu2026}. 
Larger values of $\beta$ produce stronger downweighting of observations lying in low-density regions, 
but increase the asymptotic variance under the model.
Based on our empirical illustrations presented in Section \ref{SEC:simulations}--\ref{SEC:data_ultramarathon}, 
values in the range $0.1\leq\beta\leq0.5$ often provide a useful compromise between efficiency and robustness.
However, a fully data-driven choice is also possible for CMDPDEs extending the results for those used in case of ordinary MDPDEs.
	
In particular, following \citet{warwick2005choosing}, one may select $\beta$ by minimizing an empirical estimate of 
the mean squared error, where the bias and variance components are estimated using the robust functional 
and the empirical sandwich covariance matrix given above. Iterative improvements of this idea were proposed by \citet{basak2021optimal}.
Either of these approaches can be routinely applied for the MCDPDEs as well to obtain a data-driven optimal value of $\beta$
over a finite grid minimizing the estimated risk criterion. This provides a completely automatic tuning procedure without requiring
knowledge of the unknown dependence structure among the composite components.
\qed
\end{remark}

\section{Simulation Studies} 
\label{SEC:simulations}

\subsection{Illustration Under the Multivariate Gaussian Family}
\label{SEC:simu_MVN}

We first investigate the finite-sample performance of the proposed MCDPDE under the multivariate Gaussian model, 
where closed-form expressions for the objective function and the estimating equations are available (Example~\ref{EX:MVN}). 
Besides the performance of the MCDPDE, this setting allows us to compare it with the popular cellwise MCD commonly used for Gaussian models.

We consider dimensions $d=10,20,30,40,50$ and set the sample size as $n=10d$.
The uncontaminated observations are generated from $N_d(\boldsymbol{0},\bm{\Sigma})$,
where $\bm{\Sigma}=\bm{D}^{1/2}\bm{R}\bm{D}^{1/2}$, with $\bm{D}=\sigma^2\mathbb{I}_d$, $\sigma^2=1$, 
and $\bm{R}$ a randomly generated correlation matrix having condition number equal to $100$. 
This construction creates a moderately ill-conditioned covariance structure,
which is challenging for high-dimensional covariance estimation.

Two contamination mechanisms, namely casewise and cellwise contamination, are considered 
producing two fundamentally different types of outlying behaviour in multivariate data \citep{alqallaf2009propagation, Agostinelli2015}.
These contamiations are incorporated in the simulated data using the functions \texttt{generate.casecontam} 
and \texttt{generate.cellcontam} from the \textsf{R} \citep{CRAN} package \texttt{GSE} \citep{GSE}. Specifically, for casewise contamination, a proportion $\epsilon\in\{0.1,0.2,0.3\}$ of randomly selected observations is replaced 
by contaminated observations generated as $\bm{Y}_i=y\bm{v}_{\min}+\bm{Z}_i$, 
where $\bm{v}_{\min}$ denotes the eigenvector associated with the smallest eigenvalue of $\bm{\Sigma}$,
$\bm{Z}_i\sim N_d(\bm{0},0.01^2\mathbb{I}_d)$, and the contamination magnitude $y$ is varied from $0$ to $100$. For cellwise contamination, a proportion $\epsilon\in\{0.02,0.05,0.10,0.15,0.20\}$ of randomly selected (individual) cells is replaced 
independently by observations generated from $N(y,0.01^2)$, where $y$ varies from $0$ to $15$.
This contamination mechanism mimics sparse measurement errors and isolated variable-level anomalies.

\begin{table}[!b]
	\centering
	\caption{Monte-Carlo Efficiency of different competing estimators computed relative to the MLE
	based on empirical MSE and Wasserstein-2 distance at several data dimension $d$}
	\begin{tabular}{r|rrrrr||rrrrr}
		\hline  
		Method & \multicolumn{5}{c}{$d$} & \multicolumn{5}{c}{$d$}\\
		   & 10 & 20 & 30 & 40 & 50 & 10 & 20 & 30 & 40 & 50 \\  
		\hline
		& \multicolumn{5}{c||}{MSE-based} & \multicolumn{5}{|c}{Wasserstein-2 distance-based}\\ \hline
		\hline
		 MCLE   & 1.00 & 1.00 & 1.00 & 1.00 & 1.00 & 1.00 & 1.00 & 1.00 & 1.00 & 1.00 \\
		 MCD & 0.73 & 0.76 & 0.79 & 0.84 & 0.82  & 0.77 & 0.84 & 0.86 & 0.87 & 0.88 \\
		cellMCD & 0.98 & 0.94 & 0.95 & 0.95 & 0.94 &  0.90 & 0.90 & 0.89 & 0.88 & 0.88 \\ 
		\hline		
		\multicolumn{11}{l}{Minimum Composite DPD Estimators (MCDPDE)}\\
		 $\beta=0.1$ & 1.00 & 0.95 & 0.96 & 0.96 & 0.97 & 0.99 & 0.98 & 0.98 & 0.98 & 0.98 \\ 
		0.3 & 0.93 & 0.89 & 0.92 & 0.91 & 0.93 & 0.91 & 0.90 & 0.89 & 0.89 & 0.89 \\
		0.5 & 0.88 & 0.84 & 0.86 & 0.86 & 0.88 & 0.80 & 0.78 & 0.77 & 0.76 & 0.76 \\ 
		\hline
		\multicolumn{11}{l}{Ordinary Minimum DPD Estimators (MDPDE)}\\
		$\beta=0.1$ & 0.98 & 0.89 & 0.85 & 0.78 & 0.72 & 0.98 & 0.94 & 0.92 & 0.88 & 0.84 \\ 
		0.3 & 0.69 & 0.46 & 0.23 & 0.04 & 0.02 & 0.79 & 0.69 & 0.56 & 0.33 & 0.26 \\ 
		0.5 & 0.14 & 0.08 & 0.02 & 0.01 & 0.01 & 0.44 & 0.37 & 0.26 & 0.24 & 0.23 \\ 
		\hline
		\multicolumn{11}{l}{Componentwise Minimum DPD Estimators (CMDPDE)}\\
		$\beta=0.1$ & 1.00 & 0.98 & 0.99 & 0.99 & 0.99 & 0.99 & 0.99 & 0.99 & 0.99 & 0.99 \\
		0.3 & 0.95 & 0.91 & 0.93 & 0.93 & 0.94 & 0.92 & 0.91 & 0.90 & 0.91 & 0.90 \\
		0.5 & 0.87 & 0.83 & 0.85 & 0.85 & 0.86 & 0.83 & 0.81 & 0.80 & 0.80 & 0.80 \\ 		
		\hline
	\end{tabular}
\label{TAB:normal_eff}
\end{table}

\begin{figure}[!ht]
	\centering
\subfloat[Casewise contamination]{
	\includegraphics[width=0.49\linewidth]{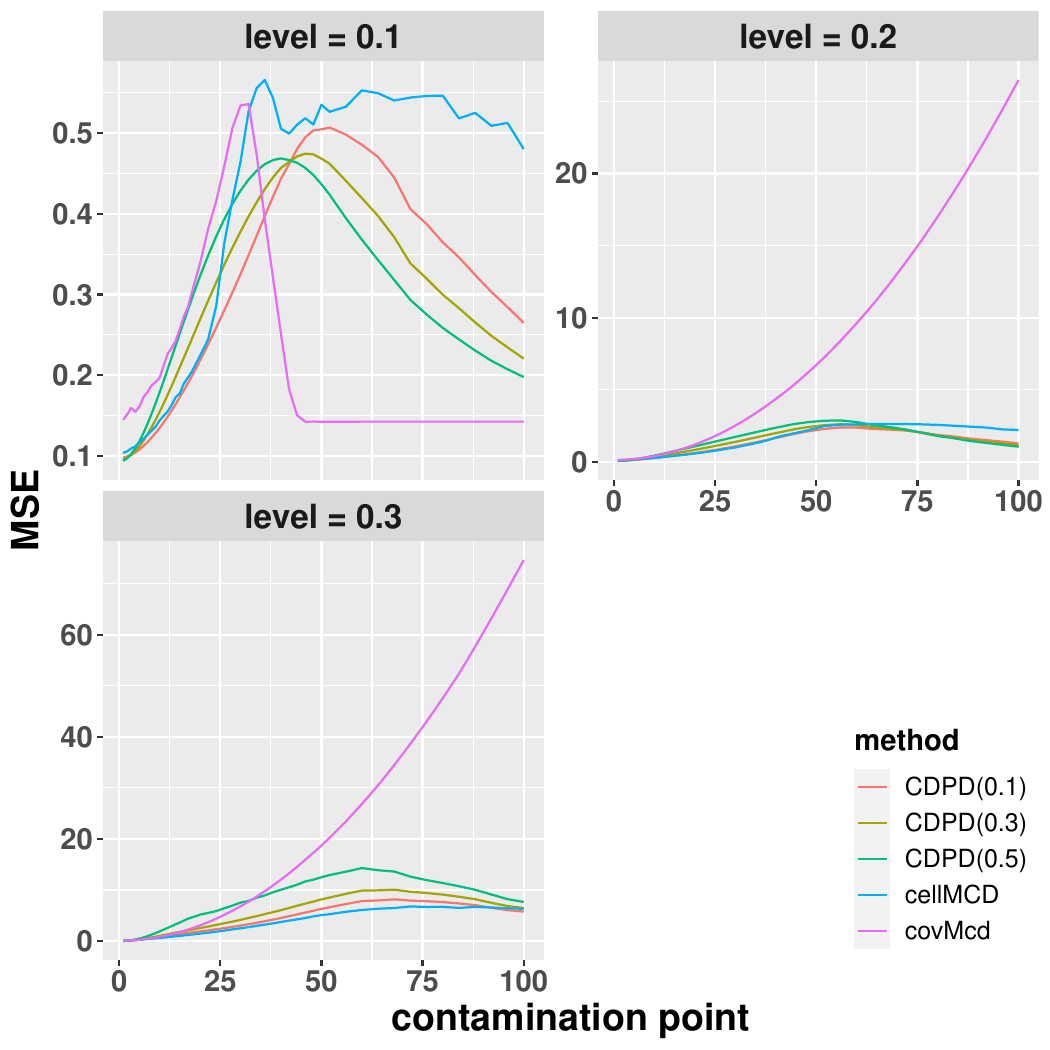}
	\includegraphics[width=0.49\linewidth]{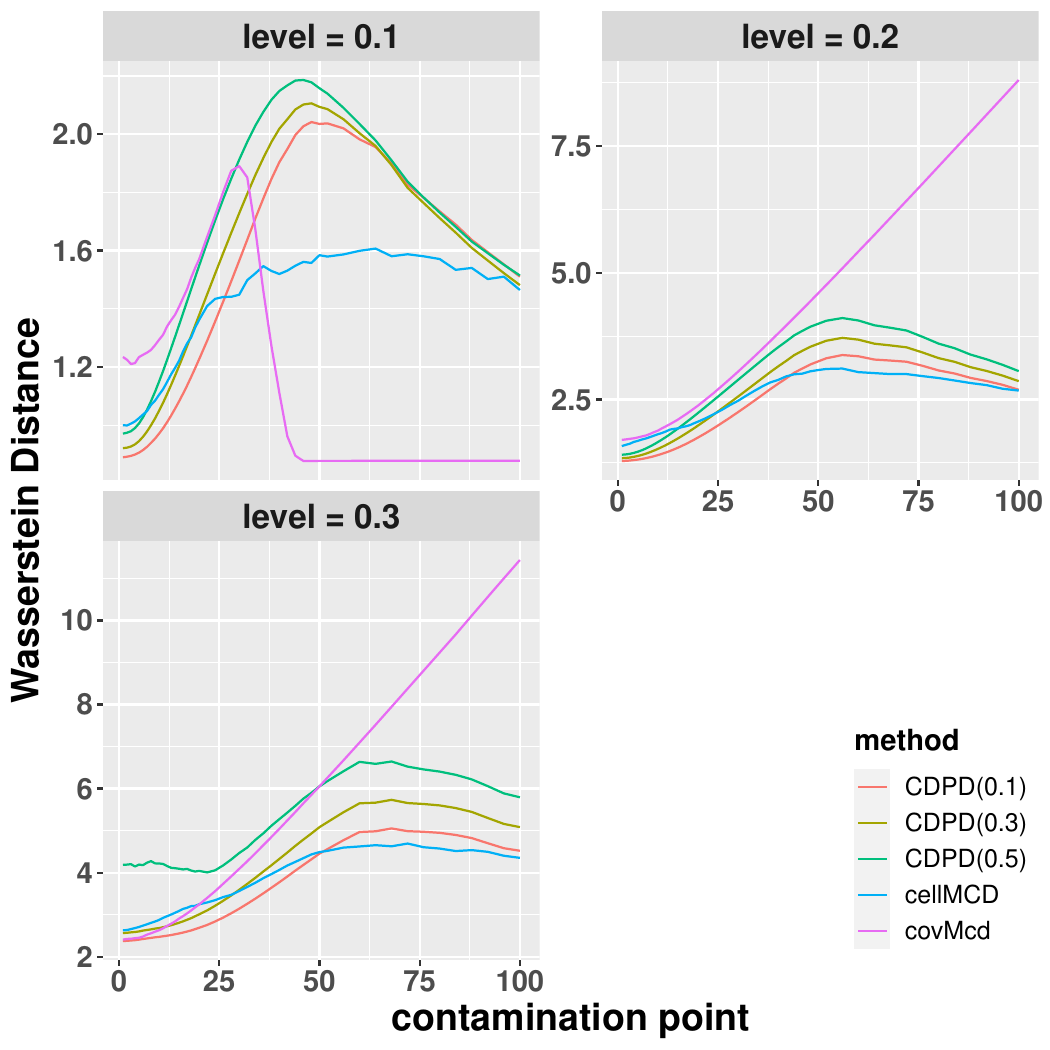}
		\label{FIG:normal-case-20}
	}
\\
\subfloat[Cellwise contamination]{
	\includegraphics[width=0.49\linewidth]{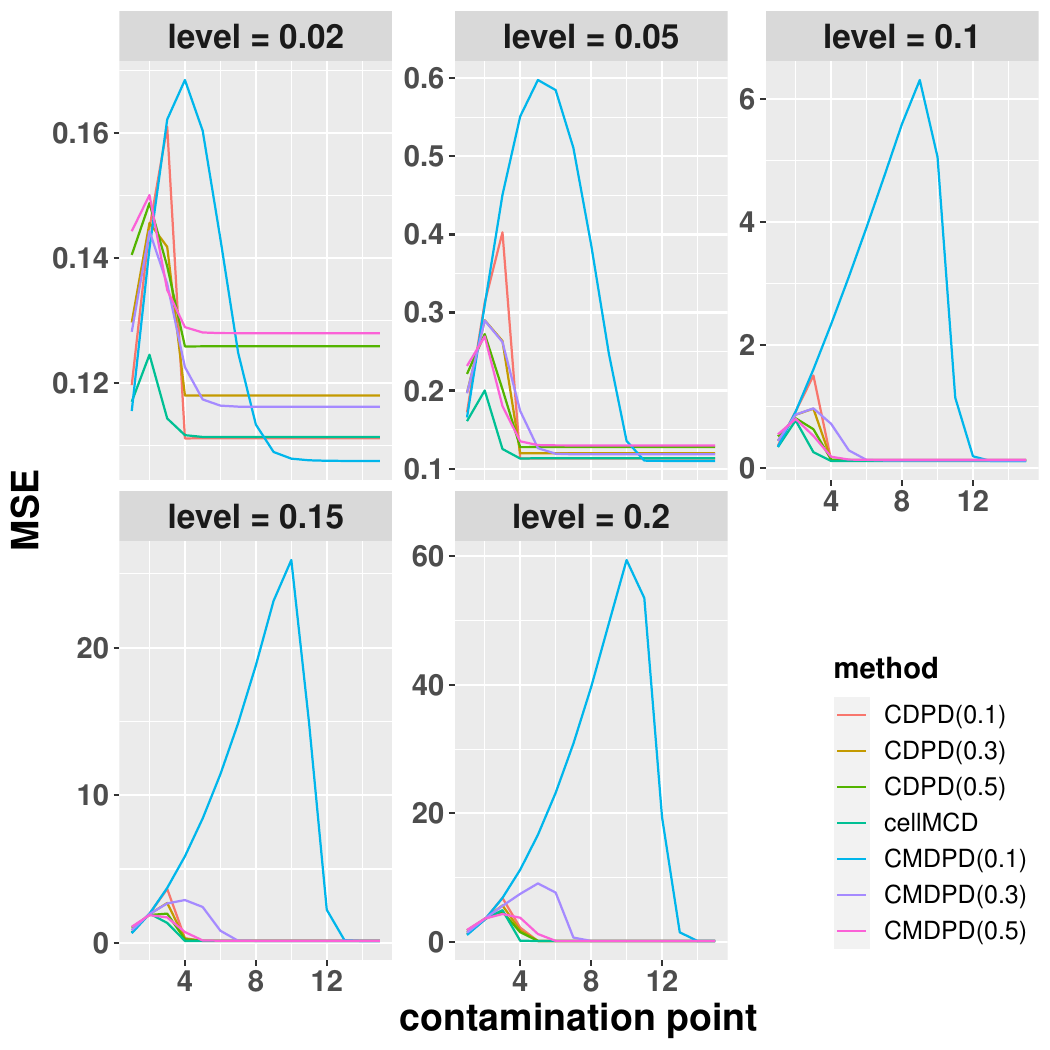}
	\includegraphics[width=0.49\linewidth]{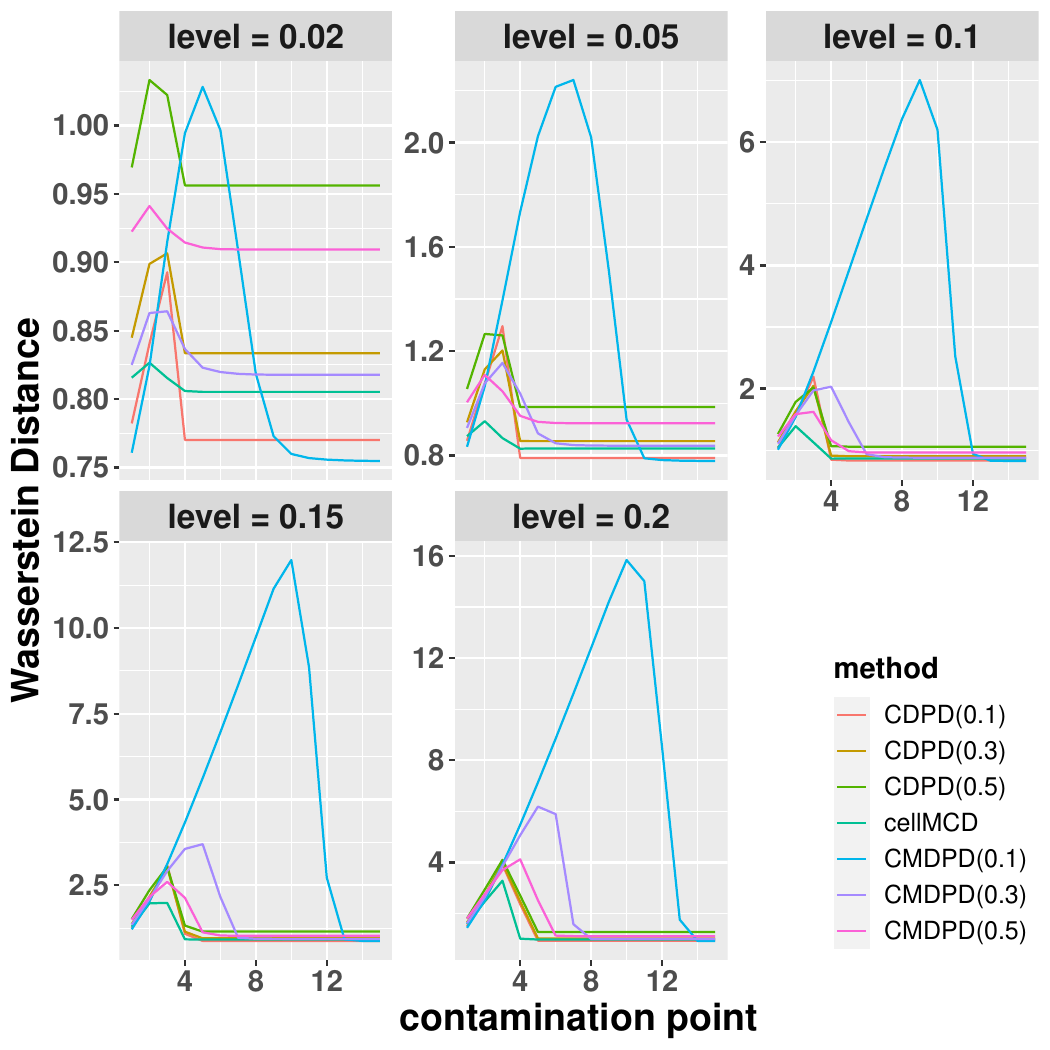}
	\label{FIG:normal-cell-20}
}
\caption{Plots of empirical MSE and Wasserstein-2 distance of different competing estimators over the contamination point $y$
	for data dimension $d=20$.}
\label{FIG:normal_cont20}
\end{figure}

For every configuration, $100$ Monte Carlo replications are performed. The performance of the proposed MCDPDEs,
at several values of $\beta>0$, are compared with the MCLE (corresponding to $\beta=0$).
The MCDPDE is computed using the fixed-point algorithm described in Appendix~\ref{APP:Algo_MVN} 
with the different initialization strategies discussed therein. Among these, we report the results obtained using the recommended 
\textit{robust filtering initialization}, as it provides the best efficiency-robustness trade-off across all our empirical studies.
We also compare the proposed estimator with several benchmark methods, including MCD, cellMCD, the ordinary MDPDE, and 
the component-wise MDPDE (CMDPDE) of \cite{chakraborty2025componentwise}. 
These benchmark estimators are initialized at the true parameter values, whenever needed, to assess their best possible performance.

For each method, we report the average empirical MSE of the location estimator and the average empirical Wasserstein-2 distance, 
computed using the location and scatter estimates, under both clean and contaminated data settings. 
The relative efficiencies of all estimators with respect to the MLE under clean data are reported in Table~\ref{TAB:normal_eff}.
The results under contaminated data are presented in Figure~\ref{FIG:normal_cont20} for dimension $d=20$, 
and in Figures~\ref{FIG:normal_cont10}--\ref{FIG:normal_cont50} in Appendix~\ref{APP:empResults} for the remaining dimensions.
For casewise contamination, we compare the proposed MCDPDE with MCD and cellMCD, 
whereas for cellwise contamination we compare it with cellMCD and CMDPDE for different values of $\beta$. We also conducted simulations with true covariance matrix having smaller condition numbers; 
since the qualitative conclusions remained essentially unchanged, these results are omitted for brevity.

As expected, we see from Table \ref{TAB:normal_eff}, the MCLE (equivalently MCDPDE with $\beta$=0) achieves an ARE of one, 
confirming no efficiency loss due to replacing the full likelihood by the pairwise CL is negligible in this setting. For the MCDPDEs at $\beta=0.1$ as well, the efficiency loss is almost negligible, with its performance remaining comparable to the MCLE 
even as the dimension increases. Increasing $\beta$ to 0.3 and 0.5 leads to a gradual reduction in efficiency, particularly 
for the Wasserstein-based criterion, reflecting the additional variability introduced by stronger downweighting of observations. 
Nevertheless, the efficiency remains competitive with existing robust estimators. In contrast, the ordinary MDPDE becomes 
increasingly inefficient in high dimensions, whereas the proposed MCDPDE maintains substantially better efficiency 
because it exploits the low-dimensional pairwise structure rather than requiring the full multivariate density. 
The componentwise MDPD estimator shows similar behavior, but the proposed estimator additionally preserves 
the dependence information through the pairwise components.

Under casewise contamination, the classical MCLE exhibits a rapid increase in both error measures as the contamination magnitude increases. 
This deterioration is particularly severe for larger dimensions, where a small number of contaminated observations can influence 
a large number of pairwise likelihood contributions simultaneously. In contrast, all MCDPDEs with $\beta>0$ remain remarkably stable, 
with substantially flatter error curves across the contamination range. The improvement is particularly pronounced for larger contamination 
proportion, where the density-power weighting mechanism effectively limits the contribution of observations lying in low-density regions. 
Among the robust procedures considered, the MCDPDE provides performance comparable to or better than the cellMCD estimator, 
while avoiding the restriction to Gaussian scatter structures imposed by covariance-based robust procedures.

A similar pattern is observed under cellwise contamination, although the nature of the robustness challenge is different. 
Since a single contaminated cell affects only a subset of pairwise components, the impact is less catastrophic than casewise contamination; 
however, the cumulative effect becomes substantial as the contamination proportion or magnitude increases. 
The MCLE again shows increasing instability, whereas the MCDPDEs maintain controlled estimation error over the entire contamination range. 
Increasing $\beta$ systematically improves robustness, with $\beta=0.3$ and $\beta=0.5$ providing the strongest protection in 
highly contaminated scenarios. The performance is competitive with specialized cellwise procedures such as cellMCD and CMDPDE, 
while the proposed approach has the important advantage of being directly extendable beyond elliptical models 
as shown in the next subsection.

\subsection{Illustration under the Multivariate Ordered Gamma Model}
\label{SEC:simu_McKay}

We next study the finite-sample performance of the proposed MCDPDE under the multivariate ordered gamma model of Example~\ref{EX:GammaExp}, 
which provides a flexible framework for modeling ordered positive-valued multivariate observations. 
Unlike the Gaussian model considered previously, this distribution lies outside the class of elliptically symmetric models. 
Consequently, existing robust multivariate procedures based on robust scatter estimation, 
including the cellwise MCD and its variants, are not directly applicable. This example therefore illustrates a key advantage of the proposed CDPD methodology that it extends robust estimation 
to flexible multivariate models for which no general robust alternative is currently available. 
We compare the MCLE ($\beta=0$) with MCDPDEs at $\beta\in\{0.1,0.2,0.3\}$ computed based on the pairwise McKay marginals, 
as well as the ordinary MLE obtained based on the full joint density,
evaluating their empirical bias and mean squared error under both pure and contaminated data.

The data are generated from the multivariate ordered gamma model with density given in \eqref{EQ:density_GammaExp}. 
We consider dimension $d=9$ with  parameters $$\boldsymbol{\delta}=(3,2.7,2.4,2.1,1.8,1.5,1.2,0.9,0.6)^\top$$ and $\lambda=0.5$.
The sample size is fixed at $n=100$, and $100$ Monte Carlo replications are generated. To investigate robustness against cellwise contamination, for each observation, 
every coordinate is first independently selected for contamination with probability $\epsilon\in\{0,0.025, 0.05,0.10\}$.
For a contaminated cell, the original value is multiplied by a random factor $C\sim U(0,5)$, 
creating either downward or upward perturbations while preserving positivity. 
After contamination, the observation vector is reordered to preserve the required ordering of the data. This contamination mechanism mimics errors in sequential measurements, ordered event times, reliability experiments, 
and environmental applications, where occasional extreme values may arise at individual measurement points.
The boxplots of the parameter estimates over 100 replications are shown in Figures \ref{FIG:marathon-est1}
and \ref{FIG:marathon-est2} (Appendix \ref{APP:empResults}).

\begin{figure}[!ht]
	\centering
	\includegraphics[width=0.49\linewidth]{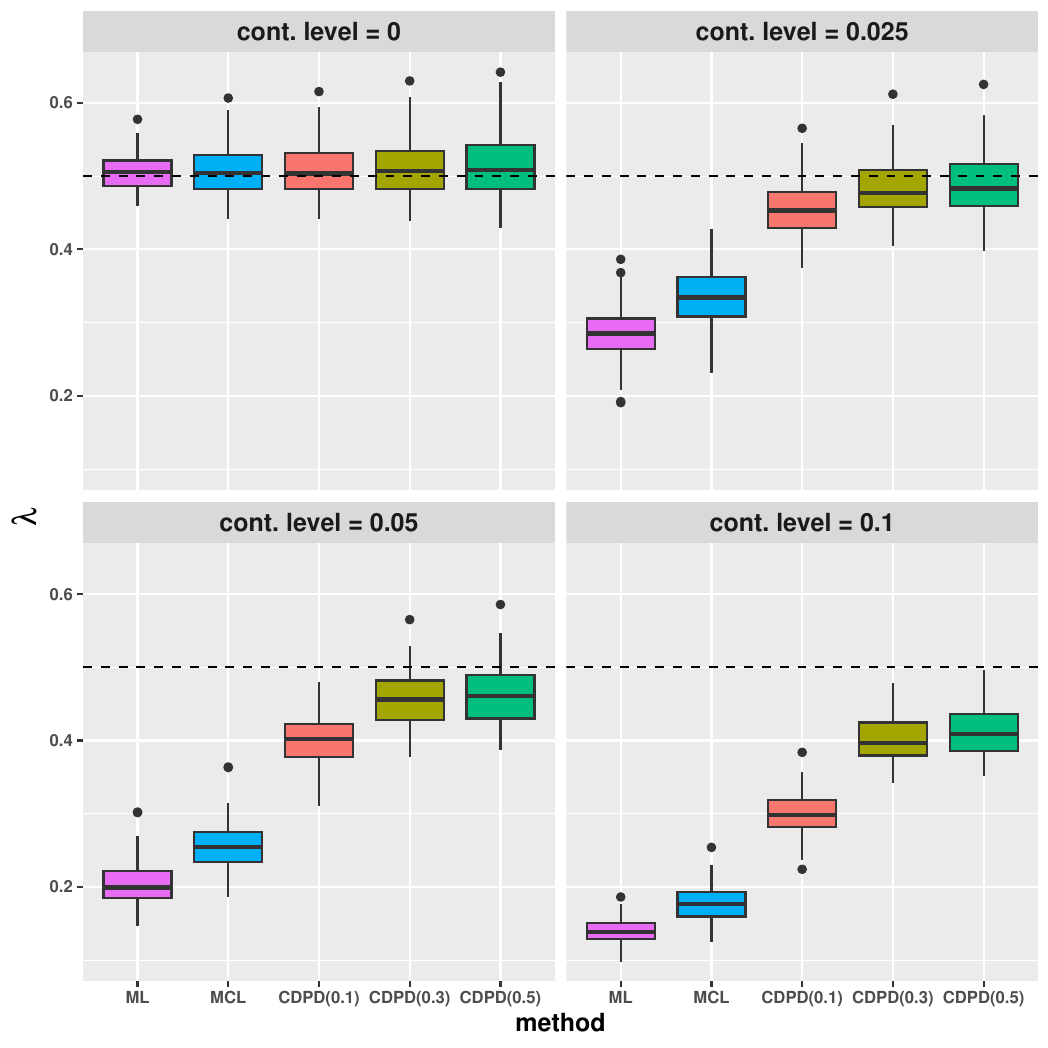}
	\includegraphics[width=0.49\linewidth]{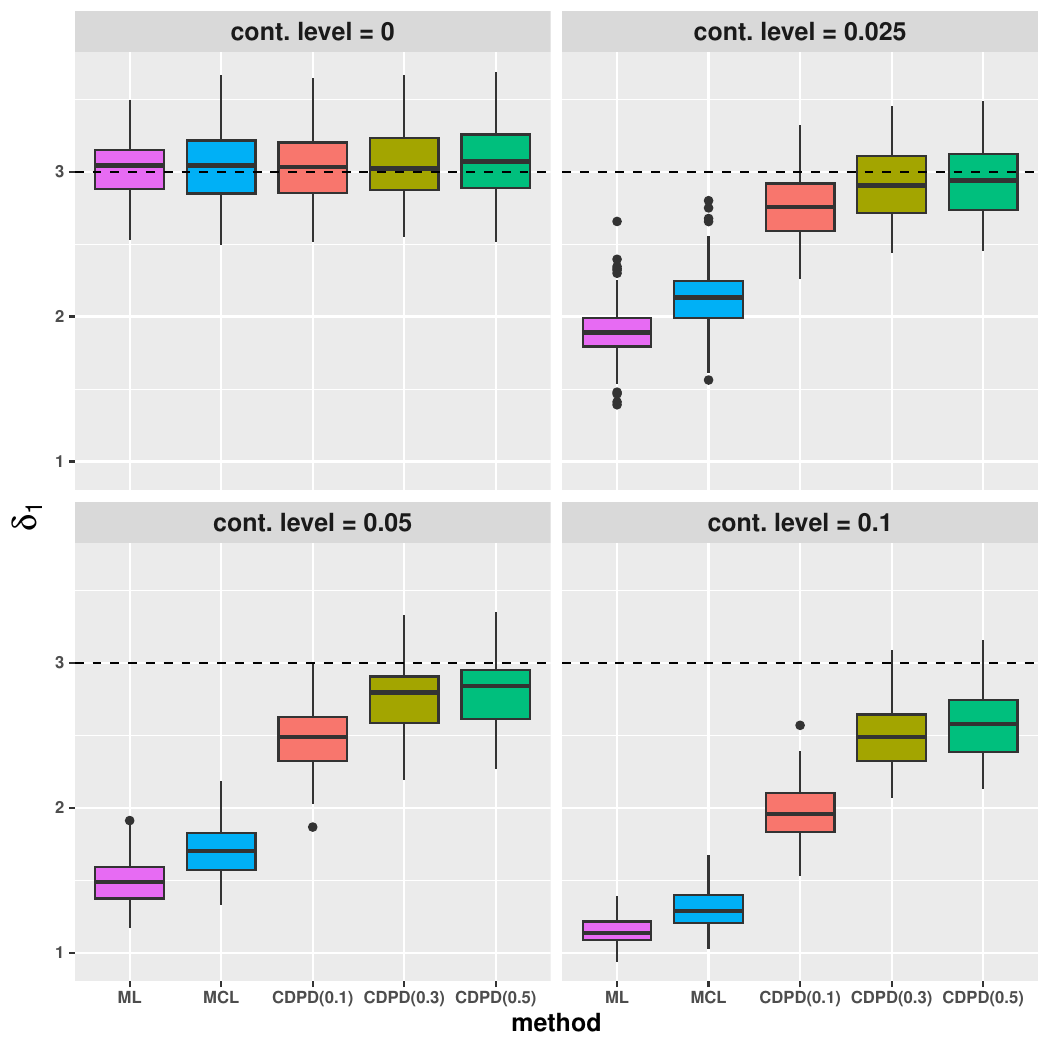}
	\includegraphics[width=0.49\linewidth]{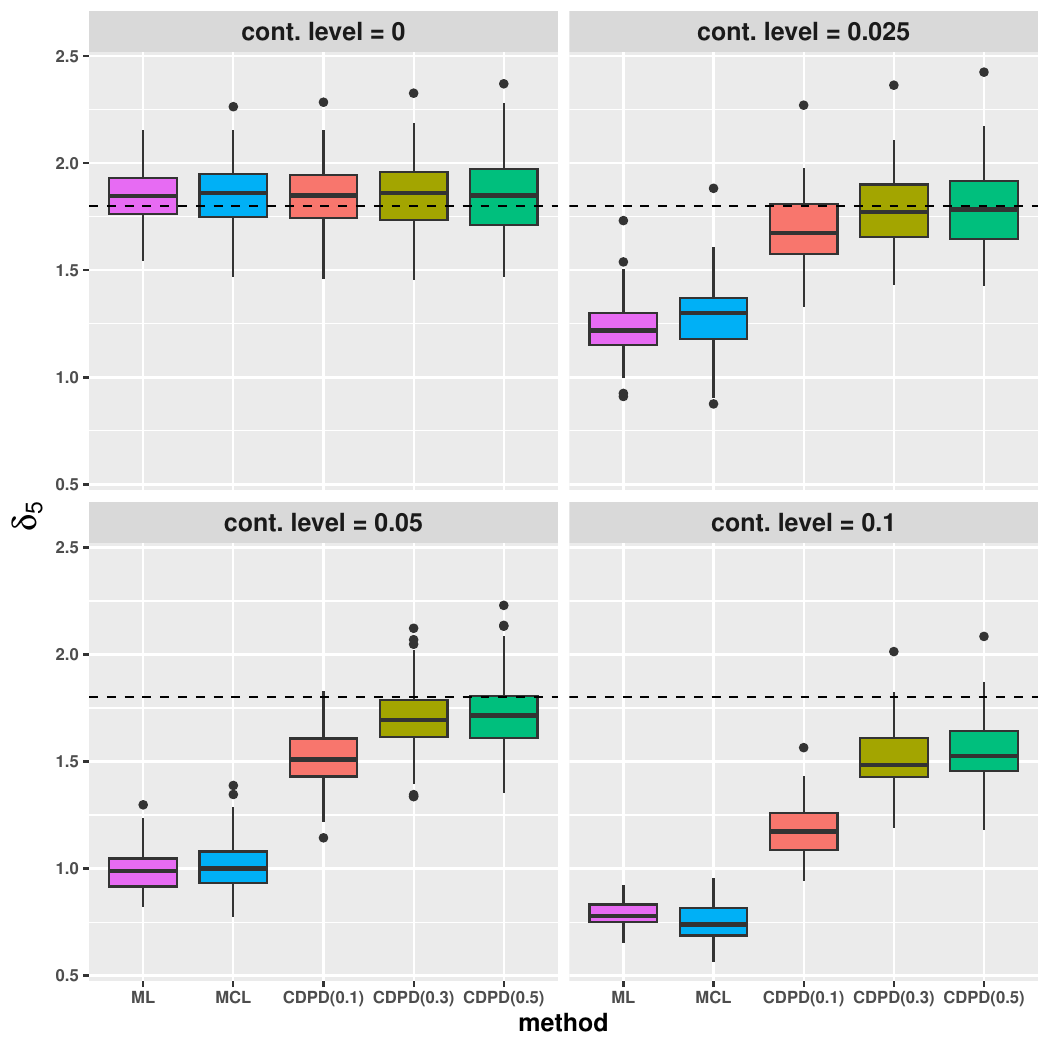}
	\includegraphics[width=0.49\linewidth]{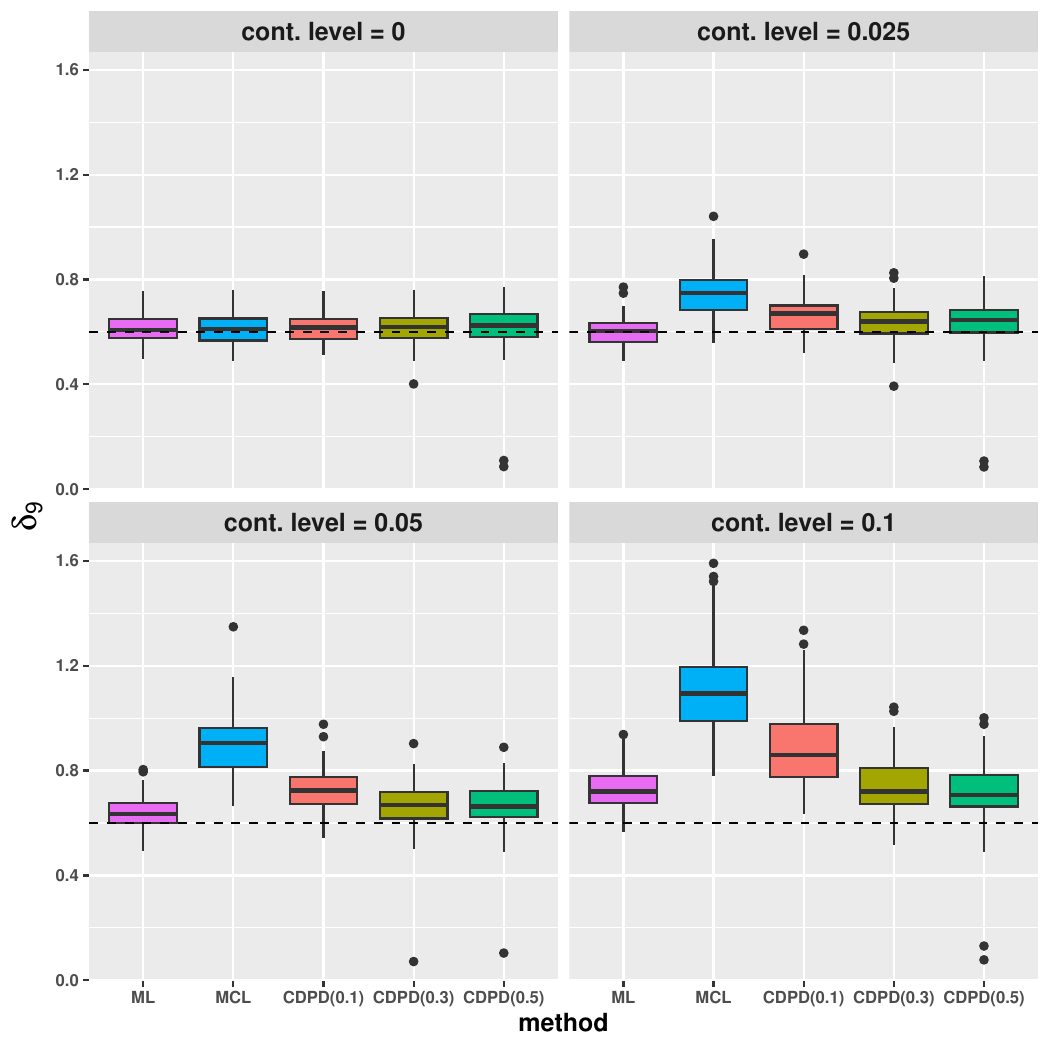}
	\caption{Boxplot of estimated parameters ($\lambda$, $\delta_1$, $\delta_5$, $\delta_9$) 
		obtained across Monte Carlo replications under the Multivariate Ordered Gamma Model.}
	\label{FIG:marathon-est1}
\end{figure}

These simulated results shows that, under uncontaminated data, the MLE, the MCLE, and the proposed MCDPDEs all recover 
the true parameter values accurately, indicating that the proposed robustification entails only a modest efficiency loss 
when the model assumptions are satisfied. As the contamination level increases, however, the MLE and the MCLE exhibit rapidly 
increasing variability and noticeable bias across nearly all model parameters. 
In contrast, the proposed MCDPDEs remain substantially more stable, with their sampling distributions continuing to be 
centered close to the true values even under the highest contamination levels considered. 
Moreover, the robustness improves systematically with the tuning parameter $\beta$, 
as larger values of ($\beta$) produce increasingly concentrated estimates and further suppress the influence of contaminated observations. The observed improvement is consistent across the entire parameter vector, demonstrating that the proposed estimator provides 
uniform protection against cellwise contamination. Although larger values of $\beta$ lead to a slight increase in variability 
under uncontaminated data, this loss is small relative to the substantial gains in robustness under contaminated settings. 
Thus, the proposed MCDPDE successfully balances statistical efficiency and robustness, 
yielding reliable estimation across a broad range of contamination scenarios.

\section{Application: Analyzing Ultramarathon Race Performances}
\label{SEC:data_ultramarathon}

Let us now illustrate the applicability of the proposed MCDPDE methodology beyond  elliptical models 
by analyzing an interesting real-life example involving ordered positive-valued multivariate data arising from ultramarathon competitions.
In particular, we analyze the ``\textit{Two Centuries of Ultra Marathon Races}'' dataset, 
publicly available on Kaggle \footnote{\texttt{https://www.kaggle.com/datasets/fatihyavuzz/two-centuries-of-um-races}}, 
which contains race results from ultramarathon events held worldwide over approximately two centuries. Since finishing times are naturally ordered positive variables, this dataset provides an appropriate setting 
for the  multivariate ordered gamma model introduced in Example~\ref{EX:GammaExp}.
Unlike the Gaussian example of the previous subsection, no competing robust estimator is currently available for 
this non-elliptical multivariate model; existing cellwise robust procedures based on the MCD or related robust scatter estimators 
rely fundamentally on elliptical symmetry and therefore cannot be applied in this setting. 

We restrict attention to mens' 50-mile races. For each race event (identified by the event name and date), 
finishing times were converted to hours and the ten fastest finishers were retained, 
producing one ordered $10$-dimensional observation per race. 
Events with fewer than ten recorded finishers or with tied finishing times among the top ten were discarded, 
yielding a collection of complete ordered observations suitable for analysis under the model of Example~\ref{EX:GammaExp}. 
We then consider a reduced dataset consisting of the fastest 200 races 
which is analyzed using the MLE, MCLE, standard multivariate MDPDEs and our MCDPDE computed based on pairwise McKay bivariate gamma distribution.

\begin{figure}[!ht]
	\centering
	\subfloat[Actual data]{
		\includegraphics[width=0.3\linewidth]{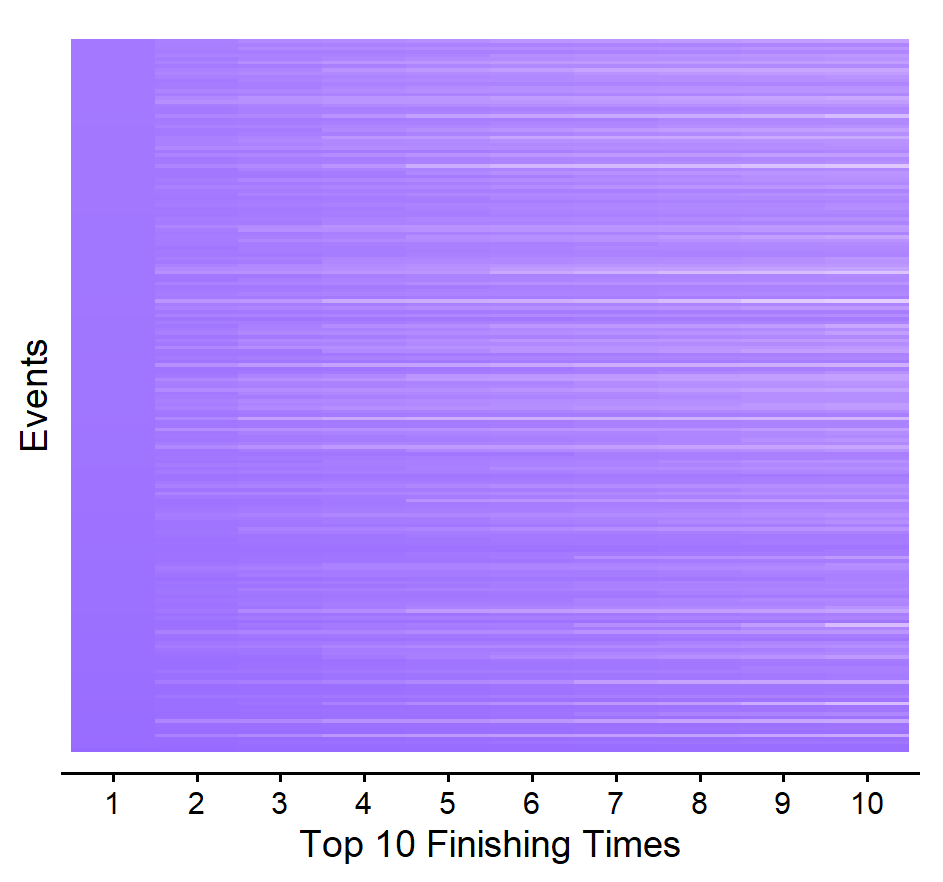}
	}
	\subfloat[Casewise Contamination]{
		\includegraphics[width=0.3\linewidth]{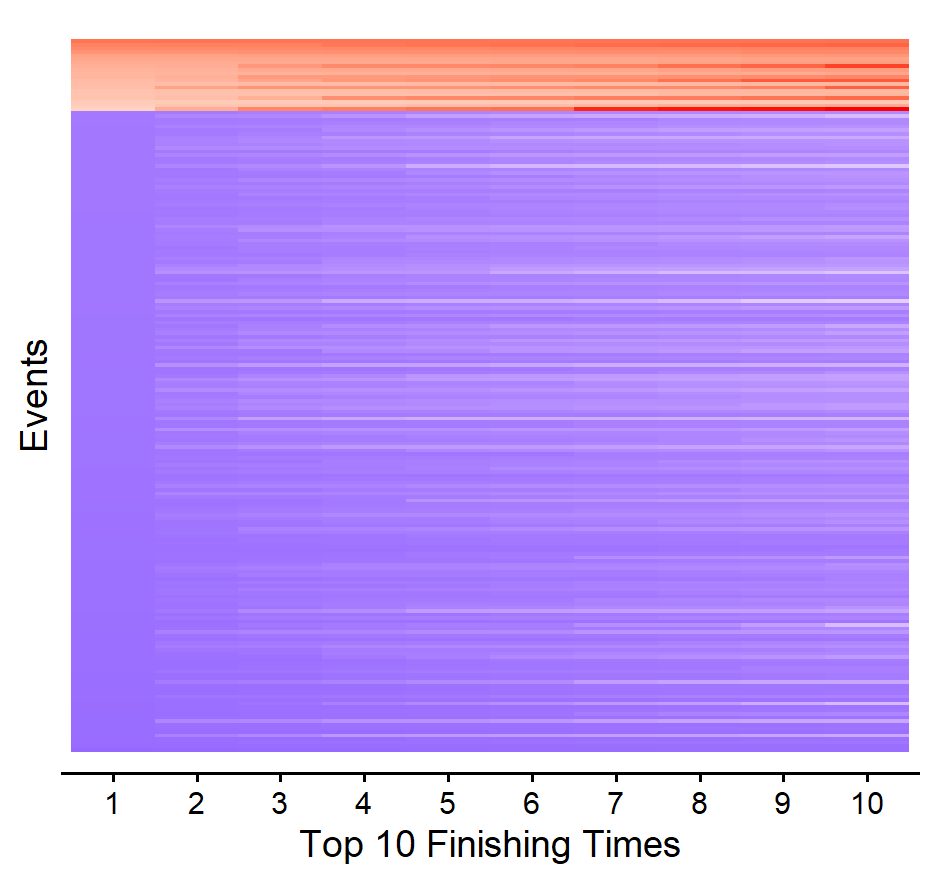}
	}
	\subfloat[Cellwise Contamination]{
		\includegraphics[width=0.36\linewidth]{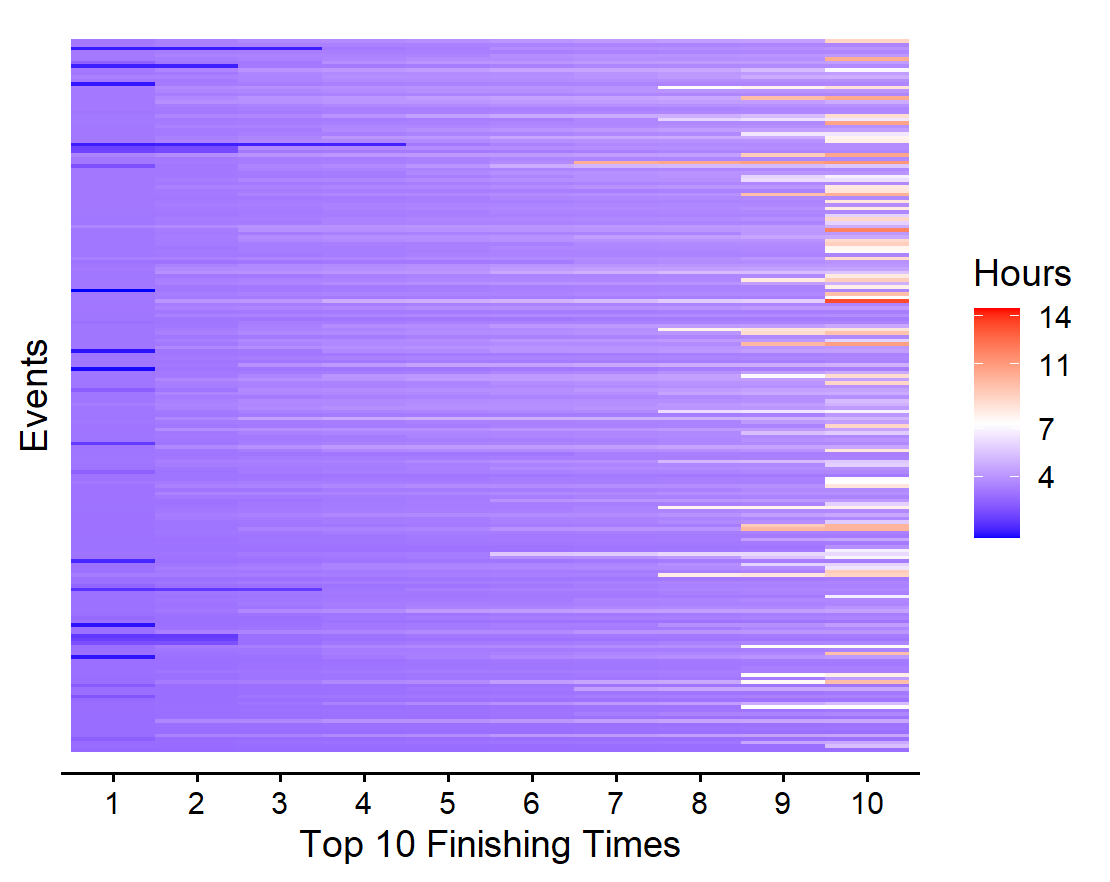}
	}
	\caption{Heatmaps of the clean and contaminated ultramarathon data}
	\label{FIG:marathon-heatmap}
\end{figure}

To examine performance against contamination, we replace the last 20 rows by the 20 worst races from the whole datasets, producing a natural rowwise contamination. 
Further, we study the robustness against cellwise contamination, by artificially corrupting cells in 20 randomly selected rows 
in exactly the same manner as done in our simulation studies, except now the random multiplicative factor $C\sim U(0,3)$.
This construction mimics recording errors or localized timing anomalies affecting only a subset of finishers within a race.
For a visual understanding, heatmaps of these contaminated data are presented in Figure \ref{FIG:marathon-heatmap}, 
along with the same for clean data.   The resulting parameter estimates, under both clean and contaminated data, 
are provided in Table \ref{TAB:marathon-estimates} in the Supplementary material.
It can be seen that the estimates of $\lambda$ and $\delta_1$ are significantly larger compared to the remaining $\delta_j$s,
which ranges from 0.65 to 1.96 across methods under clean data; in particular, the MLE under clean data 
are $\widehat{\lambda} = 11.53$ and $\widehat{\bm{\delta}} = (35.59, 1.37, 1.06, 0.96, 0.88, 0.89, 0.82, 0.76, 0.77, 0.76)^\top$. 
So, we study the absolute relative changes under both types of contamination with respect to the clean data estimates for $\lambda$ and $\delta_1$, 
which are presented in Figure \ref{FIG:marathon-RE}.

\begin{figure}[!ht]
	\centering
	\subfloat[Casewise Contamination]{
		\includegraphics[width=0.4\linewidth]{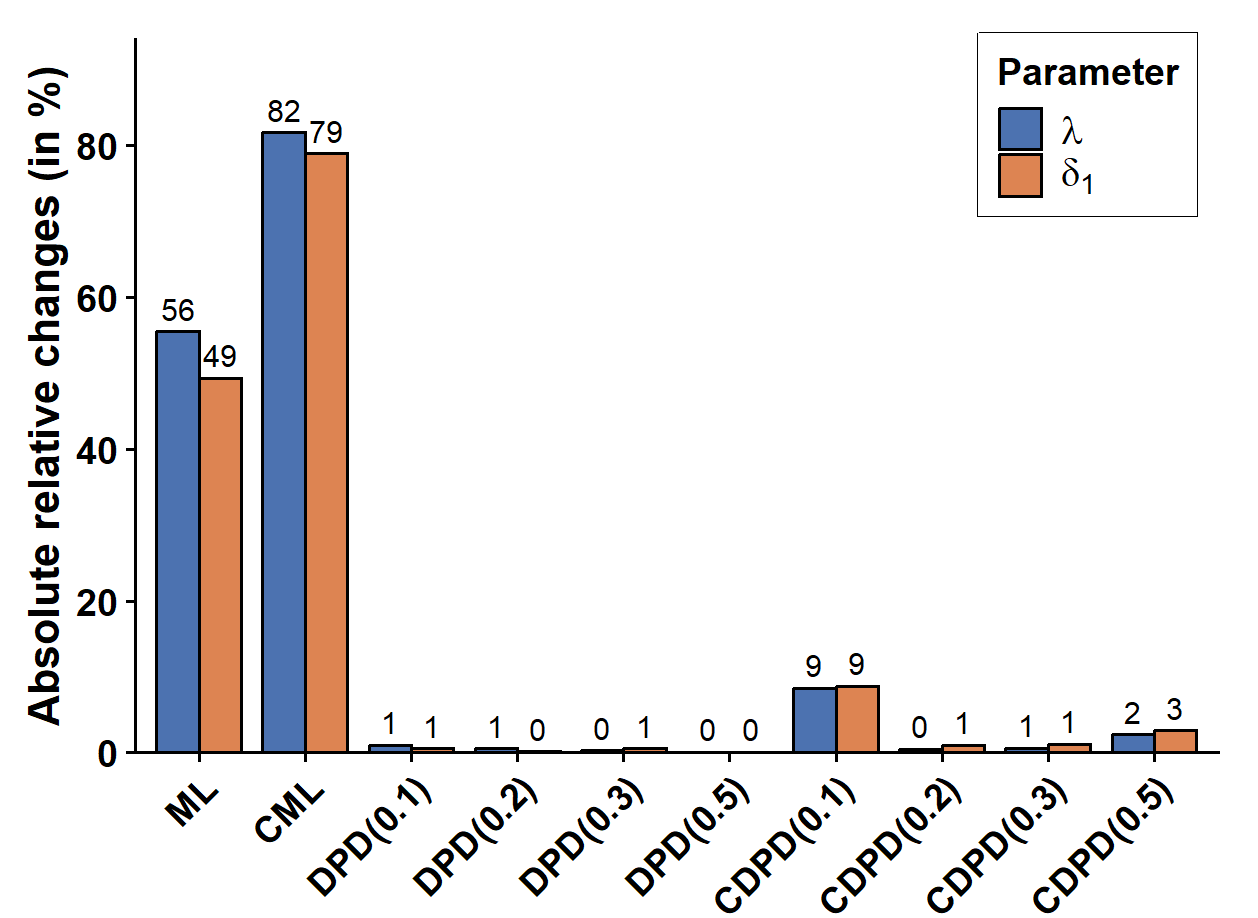}
	}
	\subfloat[Cellwise Contamination]{
		\includegraphics[width=0.4\linewidth]{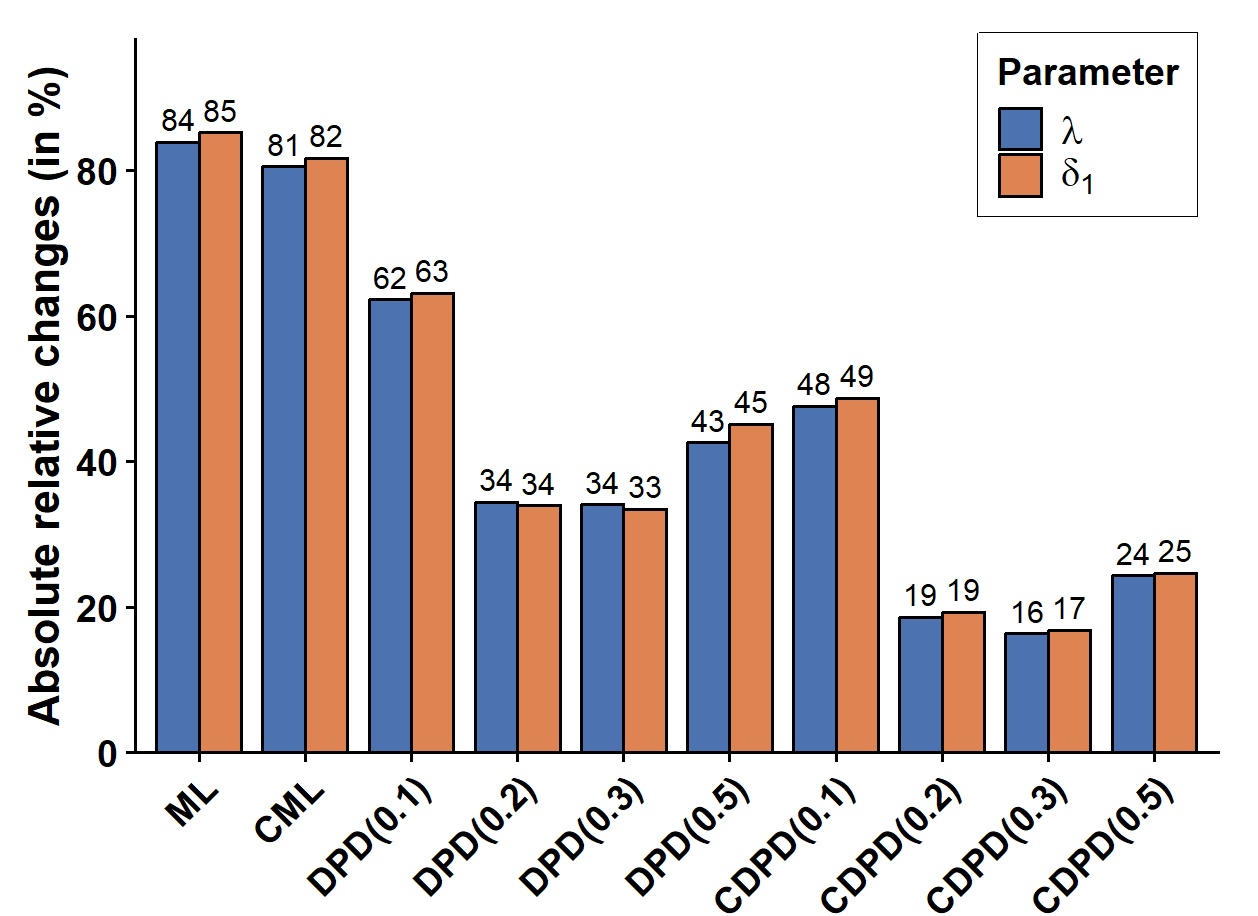}
	}
	\caption{Absolute relative changes (in \%) in the parameter estimates due to contamination for the ultramarathon data.}
	\label{FIG:marathon-RE}
\end{figure}

As in the simulation studies, the MCLE and the MLE exhibit substantial sensitivity to the contaminated entries, 
while the MCDPDEs remain considerably more stable (Table \ref{TAB:marathon-estimates}). 
Increasing $\beta$ progressively reduces the impact of the contaminated cells, with only a modest change in the uncontaminated estimates. 
Compared to the usual MDPDEs, the CMDPDEs are seen to produce significantly more stable estimates of important parameters under cellwise contamination
while being competitively similar under casewise contamination at suitable values of $\beta>0$. 
These results demonstrate that the proposed composite divergence methodology provides effective protection against cellwise contamination 
even for non-elliptical multivariate models where existing robust multivariate estimators fail or are not applicable.

\section{Discussion and Future Directions}
\label{SEC:discusions}

We introduced the composite DPD and the resulting minimum CDPD estimator as a general framework for robust inference 
in complex multivariate models where the full likelihood is unavailable or computationally infeasible. 
Our approach reveals a general principle for constructing divergence measures directly from lower-dimensional likelihood components. 
It combines the computational scalability of CL with the robustness mechanism of DPD, 
providing protection against both casewise and cellwise contamination. Importantly, this robustness comes without sacrificing asymptotic validity. 
Under regularity conditions on the component models alone, the MCDPDE is consistent and asymptotically normal, 
without requiring the full joint distribution to be tractable or even correctly specified.
For every $\beta>0$, the MCDPDE is both qualitatively robust and B-robust, having a bounded influence function. 
The estimator reduces to the MCLE at $\beta\downarrow 0$, and we pinpoint the explicit mechanism through 
which robustness is lost in this limit. The tuning parameter $\beta$ governs the usual robustness-efficiency tradeoff, 
downweighting observations in an increasingly data-dependent way as $\beta$ grows, thereby improving resistance to contamination. 
Our simulation studies and real-data analyses confirm that this yields substantial robustness gains 
at only moderate efficiency cost under the assumed model. A practical implementation for two representative examples, 
namely the multivariate Gaussian model and the non-elliptical multivariate ordered gamma model, 
is provided in the accompanying \texttt{R} package \texttt{mvdpd} on CRAN, 
only using computational ingredients similar to those required  for ordinary CL.

The proposed framework also broadens the scope of robust multivariate analysis beyond the elliptical setting. 
Most existing procedures designed for cellwise contamination rely on robust scatter estimation or distance-based constructions, 
which are naturally tied to Gaussian or elliptically symmetric models. 
In contrast, the CDPD operates directly on the component densities defining the dependence structure,
extending robustness naturally in models with non-Gaussian marginals, asymmetric distributions, discrete dependence structures, 
and other complex mechanisms whenever suitable CL components are available. This generality suggests several extensions beyond the examples considered in this paper. For clustered and longitudinal data, 
where pairwise or blockwise conditional distributions avoid high-dimensional random-effects integration, 
CDPDs could yield a robust alternative in the same spirit of the composite $\tau$-estimator of \cite{Agostinelli/Yohai:2016composite}. 
Similarly, for spatial and graphical models, CDPDs built from neighboring-pair conditionals could robustify pseudo-likelihood approaches 
such as those introduced by \cite{besag:1974spatial}, particularly against spatial outliers or localized measurement errors. 
More broadly, the CDPD provides a general foundation for robust CL inference in increasingly complex multivariate data structures.

The work also opens several theoretical and methodological directions for further investigation. 
First, although we develop the CDPD within the DPD family, the underlying construction is more general. 
The key idea is to replace the full joint density appearing in a divergence criterion by valid lower-dimensional component densities 
and aggregate the resulting component-wise divergences. Applying this principle to other divergence families, 
including general $\phi$-divergences \citep{Basu2011} and the broader class of $S$-divergences \citep{ghosh2017generalized}, 
may lead to a wider class of \textit{composite divergence estimators} with different robustness-efficiency properties. 
A systematic theory for such minimum composite divergence estimators remains to be developed.

Second, we assign equal weight to all CL components. In classical CL inference, 
optimal weighting via Godambe information can substantially improve efficiency \citep{lindsay:1988composite}. 
Extending this idea to the robust setting is nontrivial, since the DPD mechanism already introduces data-dependent weighting 
through $f_{\boldsymbol{\theta},k}^{\beta}$. Whether adaptive component weighting can further improve the MCDPDE's robustness, 
or whether the two weighting mechanisms are redundant, is an open question.

Third, robust model selection remains unresolved. Analogous to the CL-based information criterion of \cite{varin/etc:2005}, 
a CDPD-based criterion could replace the classical Godambe information adjustment with a divergence-based effective complexity term 
built from $\bm{J}_\beta$ and $\bm{K}_\beta$. Such a criterion would provide a principled way 
to compare marginal or conditional component choices in high-dimensional problems.

Fourth, the present work provides a natural basis for robust hypothesis testing under cellwise contamination, 
an area far less developed than casewise-robust testing even for elliptical models.
Most existing cellwise-robust procedures target estimation or covariance regularization,  
with limited attention to calibrated hypothesis testing. The MCDPDE's asymptotic covariance theory offers a direct route forward, 
since consistent estimation of $\bm{J}_\beta$ and $\bm{K}_\beta$ yields a sandwich covariance estimator (Proposition \ref{PROP:sandwich})
usable in Wald- or score-type tests. Such tests would combine asymptotic validity with B-robustness inherited from 
the bounded influence property. A complete theory of robust hypothesis testing based on composite divergence estimators, 
including optimal test statistics and finite-sample calibration under contamination, remains an important open problem.

Finally, our asymptotic theory treats the data dimension $d$, the number of components $K$, and the parameter dimension $p$ as fixed 
while the sample size $n\to\infty$, whereas modern applications increasingly involve settings where dimension grows with $n$. 
Establishing high-dimensional asymptotics for the MCDPDE, such as rates of convergence and conditions preserving robustness as 
$d$, $p$, or $K$ grow with $n$, is an important direction for future work.
This would further establish composite divergence methods as a general framework for robust inference in large-scale multivariate problems.


\newpage
\begin{center}
	{\Huge Supplementary Materials}
\end{center}
\setcounter{figure}{0}
\renewcommand{\thefigure}{S\arabic{figure}}
\setcounter{table}{0}
\renewcommand{\thetable}{S\arabic{table}}

\bigskip\bigskip

\appendix
\section{Technical Details and Derivations under Specific Models}
\label{APP:Computation}

\subsection{Gaussian Pairwise Composite Likelihood}
\label{APP:Algo_MVN}

Consider a multivariate continuous random vector  $\boldsymbol{X}=(X_1,\ldots,X_d)^\top$, 
where the composite likelihood is  constructed from all bivariate marginal distributions. 
Suppose that each  pairwise marginal distribution is modeled by a bivariate Gaussian density. 
The corresponding CL and CDPD objective function are obtained by combining the bivariate Gaussian component densities 
$f_{\btheta,jk}$, as provided in  Example~\ref{EX:MVN}.
We here present their estimating equations and asymptotic variance formulas with detailed derivations. 

Recall the parametrization $\btheta = (\btheta_1^\top,\ldots,\btheta_d^\top, \bm{r}^\top)^\top$ 
with $\btheta_j=(\mu_j,\sigma_j^2)^\top$ and $\bm{r}=(\rho_{12},\ldots,\rho_{d-1,d})^\top$, giving $p = 2d + \binom{d}{2}$ parameters. 
Let $\boldsymbol{x}_i=(x_{i1},\ldots,x_{id})^\top$ denote the $i$th observation 
and let $\boldsymbol{x}^{(i)}_{jk}=(x_{ij},x_{ik})^\top$ denote the corresponding pairwise observation,
which is modeled by the bivariate Gaussian density given by
\begin{equation*}
f_{\btheta,jk}(\bm{x}_{jk}) = \frac{1}{2\pi\sigma_j\sigma_k\sqrt{1-\rho_{jk}^2}}\exp\!\left(-\tfrac12\Delta(\bm{x}_{jk}\mid\bm{\mu}_{jk},\bm{\Sigma}_{jk})\right),
\end{equation*}
with $\bm{\mu}_{jk} = (\mu_j, \mu_k)^\top$ and $\bm{\Sigma}_{jk}=\begin{pmatrix}\sigma_j^2 & \rho_{jk}\sigma_j\sigma_k\\ \rho_{jk}\sigma_j\sigma_k & \sigma_k^2\end{pmatrix}$. 
Then, a simple integration gives 
\begin{equation}
	\int f_{\btheta,jk}(\bm{x})^{1+\beta}\,d\bm{x} = (2\pi)^{-\beta}(1+\beta)^{-1}\sigma_j^{-\beta}\sigma_k^{-\beta}(1-\rho_{jk}^2)^{-\beta/2} = \frac{\kappa_{jk}(\beta)}{(1+\beta)}, 
	\label{EQ:kappa}
\end{equation}
which has been used to construct the MCDPDE objective function in Example \ref{EX:MVN}.
\\

\subsubsection{Derivation of Estimating equations}
For a pair  $(j,k)$, $1\le j<k\le d$, differentiating $\log f_{\btheta,jk}$ gives the five nonzero entries of the pairwise score 
$\bm{u}_{\btheta,jk}(\bm{x})=\nabla_{\btheta}\log f_{\btheta,jk}(\bm{x})$ corresponding to the parameters 
$\mu_j$, $\mu_k$, $\sigma_j^2$, $\sigma_k^2$ and $\rho_{jk}$, while it has zero in every other coordinates.
For future reference, we will denote these five non-zero entries by $u_{\mu_j}$, $u_{\mu_k}$, $u_{\sigma_j^2}$, $u_{\sigma_k^2}$
and $u_{\rho_{jk}}$, which are given by
\begin{align*}
\frac{z_j-\rho_{jk} z_k}{\sigma_j (1-\rho_{jk}^2)}, \quad
& \frac{z_k-\rho_{jk} z_j}{\sigma_k (1-\rho_{jk}^2)}, \quad
\frac{1}{2\sigma_j^2}\left[\frac{z_j^2-\rho_{jk} z_jz_k}{(1-\rho_{jk}^2)}-1\right], \quad
\frac{1}{2\sigma_k^2}\left[\frac{z_k^2-\rho_{jk} z_jz_k}{(1-\rho_{jk}^2)}-1\right],
\end{align*}
and
\begin{equation}
\frac{\rho_{jk} (1-\rho_{jk}^2)+(1+\rho_{jk}^2)z_jz_k-\rho_{jk}(z_j^2+z_k^2)}{(1-\rho_{jk}^2)^2}, 
\label{EQ:ujk_MVN}
\end{equation}
respectively, with $z_j=(x_j-\mu_j)/\sigma_j$ for all $j=1, \ldots, d$. Then, the corresponding elements of the vector $\bm{m}_{jk}(\btheta)=\int \bm{u}_{\btheta,jk}(\bm{x})f_{\btheta,jk}(\bm{x})^{1+\beta}d\bm{x}$
turn out to be 
\begin{equation*}
0, \quad 0, \quad-\tfrac{\beta}{(1+\beta)^2}\tfrac{\kappa_{jk}(\beta)}{2\sigma_j^2}, 
\quad-\tfrac{\beta}{(1+\beta)^2}\tfrac{\kappa_{jk}(\beta)}{2\sigma_k^2},
\quad\tfrac{\beta}{(1+\beta)^2} \tfrac{\rho_{jk}\kappa_{jk}(\beta)}{(1-\rho_{jk}^2)},
\end{equation*} 
respectively (all other components of it being automatically zero).
Using all these, we get the following estimating equations for the (unweighted) MCDPDE at any $\beta\geq 0$.
\begin{align}
	\sum_{i=1}^n \sum_{k=1, k\neq j}^d & \frac{w_{jk,\beta}^{(i)}}{(1-\rho_{jk}^2)} \left[\frac{\left(x_{ij} - \mu_j\right)}{\sigma_j} 
	- \rho_{jk} \frac{\left(x_{ik} - \mu_k\right)}{\sigma_k}\right] = 0, \quad j = 1, \ldots, d,
	\label{EQ:EstEq_mu}\\
	\nonumber\\
	\sum_{i=1}^n \sum_{k=1, k\neq j}^d  & \frac{w_{jk,\beta}^{(i)}}{(1-\rho_{jk}^2)} \left[\frac{\left(x_{ij} - \mu_j\right)\left(x_{ik} - \mu_k\right)}{\sigma_j\sigma_k} \rho_{jk} 
	- \frac{\left(x_{ij} - \mu_j\right)^2}{\sigma_j^2} + (1-\rho_{jk}^2) \right]
	\nonumber\\
	& = \frac{n \beta}{(1+\beta)^{2}\sigma_j^{\beta}} \sum_{k=1, k\neq j}^d  \frac{1}{\sigma_k^\beta(1-\rho_{jk}^2)^{\frac{\beta}{2}}}, \quad j = 1, \ldots, d,
	\label{EQ:EstEq_sigma}\\
	\nonumber\\
	\sum_{i=1}^n \frac{w_{jk,\beta}^{(i)}} {(1-\rho_{jk}^2)} & \left[\frac{\left(x_{ij} - \mu_j\right)\left(x_{ik} - \mu_k\right)}{\sigma_j\sigma_k} \frac{(1+\rho_{jk}^2)}{(1-\rho_{jk}^2)} 
	- \left(\frac{\left(x_{ij} - \mu_j\right)^2}{\sigma_j^2} + \frac{\left(x_{ik} - \mu_k\right)^2}{\sigma_k^2}\right)\frac{\rho_{jk}}{(1-\rho_{jk}^2)} + \rho_{jk}\right]
	\nonumber\\
	& = \frac{n \beta\rho_{jk}}{(1+\beta)^{2}\sigma_j^\beta\sigma_k^\beta(1-\rho_{jk}^2)^{\frac{\beta}{2}+1}}, \quad 1\leq j <k \leq  d.	
	\label{EQ:EstEq_rho}
\end{align}
where the weights are given by  
\begin{equation*} 
w_{jk,\beta}^{(i)} = w_{jk,\beta}^{(i)}(\boldsymbol{\theta}_{j}, \boldsymbol{\theta}_{k},\rho_{jk})
= \frac{\exp\left({-\frac{\beta}{2}\Delta(\boldsymbol{x}_{jk}^{(i)}|\boldsymbol{\mu}_{jk}, \boldsymbol{\Sigma}_{jk})}\right)}{
	\sigma_j^\beta\sigma_k^\beta(1-\rho_{jk}^2)^{\beta/2}}, \quad 1\leq j<k \leq d, ~i=1, \ldots, n,
\end{equation*}
and $\Delta(\cdot)$ denotes the squared Mahalanobis distance as defined in the main paper.
These weights satisfy $w_{jk,0}^{(i)}=1$, and therefore the above equations reduce to the usual CL score equations when $\beta=0$. 
For $\beta>0$, observations  with small pairwise model density receive smaller weights, 
providing robustness against both casewise and cellwise contamination.

\subsubsection{A simple computational Algorithm}

The estimating equations can be solved using an alternating iterative procedure. At each iteration, 
the weights are updated using the current parameter values, and then the model parameters are updated sequentially
as per the following fixed-point iteration  algorithm.

\subsubsection{Algorithm A.1 (Computation of the MCDPDE for Gaussian Pairwise Marginals)}

Starting with an initial value $\boldsymbol{\theta}^{(0)}$ of the parameter vector $\btheta$, 
we update each component of the parameter in $(m+1)$-th iteration, denoted as $\boldsymbol{\theta}^{(m+1)}$, for $m=0,1, 2, \ldots$, 
as follows until convergence is achieved: 
\begin{itemize}
	\item Update the mean parameters $\mu_j$, for each $j=1, \ldots, d$, using the fixed-point equation 
	\begin{eqnarray}
		\mu_j = \frac{\sum_{i=1}^n \sum_{k=1, k\neq j}^d \frac{w_{jk,\beta}^{(i)}}{(1-\rho_{jk}^2)} \left[x_{ij}  
			- \rho_{jk}\sigma_j \frac{\left(x_{ik} - \mu_k\right)}{\sigma_k}\right]}{\sum_{i=1}^n \sum_{k=1, k\neq j}^d \frac{w_{jk,\beta}^{(i)}}{(1-\rho_{jk}^2)} },
		\label{EQ:Updtae_mu}
	\end{eqnarray}
	where the right-hand side is evaluated at the parameter value from the last iteration.
	
	\item Update the variance parameters $\sigma_j$ as the positive root of the quadratic equation 
	\begin{eqnarray}
		A_j(\boldsymbol{\theta}^{(m)}) \sigma^2 + B_j(\boldsymbol{\theta}^{(m)})\sigma - C_j(\boldsymbol{\theta}^{(m)}) = 0,
		\label{EQ:Updtae_sigma}
	\end{eqnarray}
	for each $j=1, \ldots, d$, 	where 
	\begin{eqnarray}
		A_j(\boldsymbol{\theta}) &=&  \sum_{i=1}^n \sum_{k=1, k\neq j}^d w_{jk,\beta}^{(i)} - \frac{n \beta}{(1+\beta)^{2}\sigma_j^{\beta}} \sum_{k=1, k\neq j}^d  \frac{1}{\sigma_k^\beta(1-\rho_{jk}^2)^{\frac{\beta}{2}}},
		\nonumber\\
		B_j(\boldsymbol{\theta}) &=& \sum_{i=1}^n \sum_{k=1, k\neq j}^d  \frac{w_{jk,\beta}^{(i)}}{(1-\rho_{jk}^2)} \frac{\left(x_{ij} - \mu_j\right)\left(x_{ik} - \mu_k\right)}{\sigma_k} \rho_{jk},
		\nonumber\\	
		C_j(\boldsymbol{\theta}) &=& \sum_{i=1}^n \sum_{k=1, k\neq j}^d  \frac{w_{jk,\beta}^{(i)}}{(1-\rho_{jk}^2)} \left(x_{ij} - \mu_j\right)^2.
		\nonumber
	\end{eqnarray}
	
	\item Update the correlation parameters $\rho_{jk}$, for each $1\leq j <k \leq  d$, using the fixed-point equation 
	\begin{eqnarray}
		\rho_{jk} = \frac{\sum_{i=1}^n w_{jk,\beta}^{(i)} \frac{\left(x_{ij} - \mu_j\right)\left(x_{ik} - \mu_k\right)}{\sigma_j\sigma_k} \frac{(1+\rho_{jk}^2)}{(1-\rho_{jk}^2)}}{
			\frac{n \beta}{(1+\beta)^{2}\sigma_j^\beta\sigma_k^\beta(1-\rho_{jk}^2)^{\frac{\beta}{2}}} + \sum_{i=1}^n w_{jk,\beta}^{(i)} \left[
			\left(\frac{\left(x_{ij} - \mu_j\right)^2}{\sigma_j^2} + \frac{\left(x_{ik} - \mu_k\right)^2}{\sigma_k^2}\right)\frac{1}{(1-\rho_{jk}^2)} -1\right] },
		\label{EQ:Updtae_rho}
	\end{eqnarray}
	where the right-hand side is evaluated at the parameter value from the last iteration. 	 
\end{itemize}
Iterations are terminated when the relative change in all parameter estimates falls below a specified tolerance.
Since the objective function is non-convex, we investigated the effect of initialization through the following three strategies: 
\begin{enumerate}
	\item[(i)] Oracle initialization: the algorithm is initialized at the true parameter value.  
	This evaluates the intrinsic statistical performance without initialization effects.
	
	\item[(ii)] Robust marginal initialization:	location parameters are initialized by component-wise medians.
	The covariance matrix is initialized using the marginal median absolute deviation (MAD) approximation, 
	following standard robust covariance initialization procedures \citep{maronna2019robust}, as given by
	\begin{equation*}
	\widehat{\Sigma}_{jk} = 	1.4826^2 \operatorname{MAD}(X_j) \operatorname{MAD}(X_k),
	\end{equation*}
	
	\item[(iii)] Robust filtering initialization: first, univariate DPD estimators with $\beta=0.99$ are computed for each coordinate. 
	Observations are standardized using these robust estimates, and cells with absolute standardized values exceeding $3$ 
	are treated as missing. The MCDPDE is then fitted using pairwise complete observations only, 
	exploiting the missing-data extension developed in Section~\ref{SEC:missing}.
\end{enumerate}

As expected, the first one (oracle choice) produces the best results among the three. However, in practice, 
we recommend to use the third option (robust filtering based initialization) as it performs significantly better than the second one 
and also quite close to the oracle initialization.  

\begin{remark}
For the scale update in Algorithm A.1, the quadratic equation has two roots. 
Under regularity conditions and for sufficiently large samples, one root is positive and one is negative, 
so the positive root provides an unambiguous update for $\sigma_j$. 
\qed
\end{remark}

\subsubsection{Derivation of the asymptotic covariance matrix of (unweighted) MCDPDEs}

Let us first compute the the sensitivity matrix $\bm{J}_\beta({\btheta}_0)$ which is given by the sum of the integral 
$\int \bm{u}_{{\btheta}_0,jk}(\bm{x})\bm{u}_{{\btheta}_0,jk}(\bm{x})^\top f_{{\btheta}_0,jk}(\bm{x})^{1+\beta}d\bm{x}$ 
over all pairs $(j,k)$ with $1\leq j<k \leq d$. Based on the form of $\bm{u}_{\btheta, jk}$ derived above,
the only non-zero entries of this integral (matrix) are
\begin{align}
\int u_{\mu_{j}}^2(\bm{x}) f_{{\btheta}_0,jk}^{1+\beta}(\bm{x})\,d\bm{x} &= \frac{\kappa_{jk}(\beta)}{\sigma_j^2(1 - \rho_{jk}^2)}(1+\beta)^{-2},  ~~1\leq j\leq d, \label{eq:Jmu1}\\
\int u_{\sigma_j^2}^2(\bm{x}) f_{{\btheta}_0,jk}^{1+\beta}(\bm{x})\,d\bm{x} &= \frac{\kappa_{jk}(\beta)}{4\sigma_j^4}\left[\frac{(3-2\rho_{jk}^2)}{(1-\rho_{jk}^2)} + \beta^2 -1 \right](1+\beta)^{-3},  ~~1\leq j\leq d,  
\label{eq:Jsjsj}\\
\int u_{\mu_{j}}(\bm{x})u_{\mu_{k}}(\bm{x}) f_{{\btheta}_0,jk}^{1+\beta}(\bm{x})\,d\bm{x} & =  - \frac{\rho_{jk}\;\kappa_{jk}(\beta)}{\sigma_j\sigma_k(1 - \rho_{jk}^2)}(1+\beta)^{-2}, ~~1\leq j, k\leq d,\label{eq:Jmu12}\\
\int u_{\sigma_j^2}(\bm{x})u_{\sigma_k^2}f_{{\btheta}_0,jk}^{1+\beta}(\bm{x})\,d\bm{x} &
=\frac{\kappa_{jk}(\beta)}{4\sigma_j^2\sigma_k^2}\left[\frac{(1-2\rho_{jk}^2)}{(1-\rho_{jk}^2)} + \beta^2 -1\right](1+\beta)^{-3}, ~~1\leq j, k\leq d, 
\label{eq:Jsjsk}\\
\int u_{\sigma_j^2}(\bm{x})u_{\rho_{jk}}f_{{\btheta}_0,jk}^{1+\beta}(\bm{x})\,d\bm{x} &
= -\frac{\rho_{jk}\;\kappa_{jk}(\beta)}{2\sigma_j^2(1-\rho_{jk}^2)} \frac{(1+\beta^2)}{(1+\beta)^{3}}, ~~1\leq j,  k\leq d, \label{eq:Jsjrho}\\
\int u_{\rho_{jk}}^2(\bm{x}) f_{{\btheta}_0,jk}^{1+\beta}(\bm{x})\,d\bm{x} &= \kappa_{jk}(\beta)\,\frac{1 + \rho_{jk}^2(1+\beta^2)}{(1-\rho_{jk}^2)^2} (1+\beta)^{-3}, ~~1\leq j, k\leq d. 
\label{eq:Jrhorho}
\end{align}
Then, the diagonal entries of $\bm{J}_\beta({\btheta}_0)$ corresponding to $\mu_l$ and $\sigma_l^2$ ($l=1, \ldots, d$) are given by 
\begin{equation*}
(1+\beta)^{-2}\sum_{k=1, k\neq l}^d \frac{\kappa_{lk}(\beta)}{\sigma_l^2(1 - \rho_{lk}^2)} 
~~\mbox{ and }~~
(1+\beta)^{-3}\sum_{k=1, k\neq l}^d \frac{\kappa_{lk}(\beta)}{4\sigma_l^4}\left[\frac{(3-2\rho_{lk}^2)}{(1-\rho_{lk}^2)} + \beta^2 -1 \right],
\end{equation*} 
respectively, and all remaining elements of $\bm{J}_\beta({\btheta}_0)$ are given by exactly one integral from \eqref{eq:Jmu12}--\eqref{eq:Jrhorho} 
if it is available, or zero otherwise. In particular, any pair of coordinates belonging to no common pair 
(e.g. $\rho_{jk}$ vs.\ $\rho_{lm}$ for $\{j,k\}\ne\{l,m\}$, or $\sigma_j^2$ vs.\ $\rho_{kl}$ with $j\notin\{k,l\}$), 
and any location-shape pair, becomes zero. 

Next, let us compute the variability matrix $\bm{K}_\beta({\btheta}_0)$, which can be expressed following \eqref{EQ:Kbeta-model} as 
\begin{eqnarray}
\bm{K}_\beta({\btheta}_0)=\sum_{S}\sum_{T} \big[\bm{C}_{S,T}({\btheta}_0) - \bm{m}_S({\btheta}_0)\bm{m}_T({\btheta}_0)^\top\big], 
\label{EQ:Kbeta-model-MVN}
\end{eqnarray}
with the double sum ranging over \emph{all ordered pairs} $(S,T)$ of composite components, 
i.e., $S, T \in\left\{ (j,k) : 1\leq j<k\leq d\right\}$, and
\begin{equation*}
\bm{C}_{S,T}({\btheta}_0) = \E_G\!\left[\bm{u}_{{\btheta}_0,S}(\bm{X})\bm{u}_{{\btheta}_0,T}(\bm{X})^\top f_{{\btheta}_0,S}(\bm{X}_S)^\beta f_{{\btheta}_0,T}(\bm{X}_T)^\beta\right].
\end{equation*}
Unlike $\bm{J}_\beta$, this requires the joint law of $\bm{X}$ restricted to $A=S\cup T$ whenever $S\ne T$, 
and the formulas differ dependent on whether $S,T$ share an index or are fully disjoint.
We will compute $\bm{C}_{S,T}$, and hence $\bm{K}_\beta({\btheta}_0)$, at the model assuming that these restricted distributions 
are Gaussian with suitable mean vector and covariance matrices, for which we will use the following lemma. 

\begin{lemma}\label{LEM:double-tilt}
Suppose $g_A(\cdot; \bm{\mu}_A, \bm{\Sigma}_A)$ denotes the density of $N(\bm{\mu}_A,\bm{\Sigma}_A)$ on $\mathbb{R}^{|A|}$. 
Let $S,T$ be two (possibly overlapping) index subsets of $A$, with component covariances $\bm{\Sigma}_S$, $\bm{\Sigma}_T$ 
(suitable sub-blocks of $\bm{\Sigma}_A$),  and let $\bm{P}_S, \bm{P}_T$ denote $\bm{\Sigma}_S^{-1},\bm{\Sigma}_T^{-1}$ 
embedded (zero-padded) into the $A$-space. Then for $\beta\ge0$,
\begin{equation}
	g_A(\bm{x}_A;  \bm{\mu}_A, \bm{\Sigma}_A)\,f_{\btheta,S}(\bm{x}_S)^\beta f_{\btheta,T}(\bm{x}_T)^\beta = \Lambda_{S,T}(\beta)\,g_A(\bm{x}_A; \bm{\mu}_A, \bm{M}^{-1}),
 \label{eq:Lambda}
\end{equation}
with $\bm{M} = \bm{\Sigma}_A^{-1}+\beta(\bm{P}_S+\bm{P}_T)$ and
\begin{equation*}
  \Lambda_{S,T}(\beta) = (2\pi)^{-\beta(|S|+|T|)/2}\,|\bm{\Sigma}_S|^{-\beta/2}|\bm{\Sigma}_T|^{-\beta/2}\,|\bm{\Sigma}_A|^{-1/2}|\bm{M}|^{-1/2}.
\end{equation*}  
Interesting special case is $T=\varnothing$, $S=A$ and $g_A = f_{\btheta,S}$, yielding
\begin{equation*}
  g_A(\bm{x}_A;  \bm{\mu}_A, \bm{\Sigma}_A)^{1+\beta} = \widetilde{\kappa}_A(\beta)\,g_A\!\left(\bm{x}_A;\bm{\mu}_A,\tfrac{1}{1+\beta}\bm{\Sigma}_A\right)
\end{equation*}
where
\begin{equation*}
\widetilde{\kappa}_A(\beta) = (2\pi)^{-\beta|A|/2}(1+\beta)^{-|A|/2}|\bm{\Sigma}_A|^{-\beta/2},
\end{equation*}
\end{lemma}
\begin{proof}
The result follow by matching the quadratic forms in the exponent of all the (Gaussian) densities involved, since  
\begin{equation*}
g_A(\bm{x}_A)f_{\btheta,S}^\beta f_{\btheta,T}^\beta \propto \exp\{-\tfrac12(\bm{x}_A-\bm{\mu}_A)^\top[\bm{\Sigma}_A^{-1}+\beta(\bm{P}_S+\bm{P}_T)](\bm{x}_A-\bm{\mu}_A)\},
\end{equation*} 
which is proportional to $g_A(\bm{x}_A;\bm{\mu}_A,\bm{M}^{-1})$.
Subsequently, the constant $\Lambda_{S,T}(\beta)$ is obtained by equating normalizing constants. 
\end{proof}

We can apply the above lemma to compute $\bm{C}_{S,T}({\btheta}_0)$ by noting that each non-zero components of the pairwise scores, 
$\bm{u}_{\btheta}(\bm{x})$, are normally distributed. This can be observed by noting that, with $\bm{Y}=(X_j-\mu_j,X_k-\mu_k)^\top$, we have
\begin{align}
	u_{\mu_j}(\bm{Y}) &= \bm{\ell}_{\mu_j}^\top \bm{Y}, & \bm{\ell}_{\mu_j} &= \frac{1}{(1-\rho_{jk}^2)}\left(\frac{1}{\sigma_j^2}, -\frac{\rho_{jk}}{\sigma_j\sigma_k}\right)^\top, \label{eq:ell}\\
	u_{\sigma_j^2}(\bm{Y}) &= \bm{Y}^\top \bm{B}_{\sigma_j^2}\bm{Y} + c_{\sigma_j^2}, &
	\bm{B}_{\sigma_j^2} &= \frac{1}{2\sigma_j^2(1-\rho_{jk}^2)}\begin{pmatrix} \frac{1}{\sigma_j^2} & -\frac{\rho_{jk}}{2\sigma_j\sigma_k}\\ -\frac{\rho_{jk}}{2\sigma_j\sigma_k} & 0\end{pmatrix}, \quad 
	c_{\sigma_j^2}=-\frac{1}{2\sigma_j^2}, \label{eq:Bsj}\\
	u_{\rho_{jk}}(\bm{Y}) &= \bm{Y}^\top \bm{B}_{\rho_{jk}}\bm{Y}+c_{\rho_{jk}}, &
	\bm{B}_{\rho_{jk}} &= \frac{1}{(1-\rho_{jk}^2)^2}\begin{pmatrix}-\frac{\rho_{jk}}{\sigma_j^2} & \frac{(1+\rho_{jk}^2)}{2\sigma_j\sigma_k}\\ \frac{(1+\rho_{jk}^2)}{2\sigma_j\sigma_k} & -\frac{\rho_{jk}}{\sigma_k^2}\end{pmatrix}, \quad c_{\rho_{jk}}=\frac{\rho_{jk}}{(1-\rho_{jk}^2)}, \label{eq:Brho}
\end{align}
Given two pairs $S=(j,k)$, $T=(j',k')$, let $A=S\cup T$, so that $|A|\in\{2,3,4\}$.
We can then form $\bm{\Sigma}_A=[\sigma_a\sigma_b\rho_{ab}]_{a,b\in A}$ from the components of $\btheta$ for every pair $\{a,b\}\subseteq A$, 
$\bm{P}_S,\bm{P}_T$ as embedded $\bm{\Sigma}_{jk}^{-1},\bm{\Sigma}_{j'k'}^{-1}$,  
and compute $\bm{\Omega}=\left[\bm{\Sigma}_A^{-1}+\beta(\bm{P}_S+\bm{P}_T)\right]^{-1}$ and  $\Lambda_{S,T}(\beta)$ as in Lemma \ref{LEM:double-tilt}. 
Then, for score components $a$ of $S$ and $b$ of $T$ (each embedded in $\mathbb{R}^{|A|}$ as in \eqref{eq:ell}--\eqref{eq:Brho}), yielding
\begin{align}
	\quad & \mbox{location - location:}\nonumber\\
  C_{S,T}[a,b] & = \Lambda_{S,T}(\beta)\,\bm{\ell}_a^\top\bm{\Omega}\,\bm{\ell}_b, \label{eq:Cloc}\\[4pt]
	\quad & \mbox{scale - scale / corr - corr /  scale - corr:} \nonumber\\
	C_{S,T}[a,b] & = \Lambda_{S,T}(\beta)\Big[2\,\tr(\bm{B}_a\bm{\Omega} \bm{B}_b\bm{\Omega})+\tr(\bm{B}_a\bm{\Omega})\tr(B_b\bm{\Omega}) \nonumber\\
	& + c_a\tr(\bm{B}_b\bm{\Omega})+c_b\tr(\bm{B}_a\bm{\Omega})+c_ac_b\Big], \label{eq:Cshape}\\[4pt]
	\quad & \mbox{loc - scale / loc - corr:} \nonumber\\
  C_{S,T}[a,b] &= 0 \quad \text{ (for every $\bm{\Omega}$ and dimension).}\label{eq:Czero}
\end{align}
Formulas \eqref{eq:Cloc}--\eqref{eq:Czero} are the standard raw-moment identities for linear/quadratic forms of a Gaussian vector $\bm{Y}_A\sim N(\bm{0},\bm{\Omega})$.
Identity \eqref{eq:Czero} holds because the third-order moments of any centered Gaussian vector vanish identically, regardless of $\bm{\Omega}$ or dimension.

Once we get $\bm{C}_{S,T}$ for all $(S, T)$, we can use them to compute $\bm{K}_\beta({\btheta}_0)$ via \eqref{EQ:Kbeta-model-MVN}, 
where $\bm{m}_S$ are derived previously. This requires summing over all $K=\binom{d}{2}$ composite components $S$ and $T$ (including $S=T$), 
embedding each $C_{S,T}[a,b]- m_S[a]m_T[b]$ contribution into the appropriate entries of the $p\times p$ matrix indexed by 
$(\mu_j,\mu_k,\sigma_j^2,\sigma_k^2,\rho_{jk})$ (from $S$) and $(\mu_{j'},\mu_{k'},\sigma_{j'}^2,\sigma_{k'}^2,\rho_{j'k'})$ (from $T$). 
Because a single coordinate (e.g.\ $\sigma_j^2$) belongs to every pair containing index $j$, 
a single diagonal entry such as $[\bm{K}_\beta]_{\sigma_j^2\sigma_j^2}$ receives contributions from $O(d^2)$ pairs of pairs, 
in contrast to the pair-local sparsity of $\bm{J}_\beta$.
Therefore, except for a few very special cases, this process does not lead to a nice close-form final expression for $\bm{K}_\beta({\btheta}_0)$.
However, they can be easily computed through any standard programming language; 
our supplementary \texttt{R}-package \texttt{mvdpd} provides a function to compute this variability matrix $\bm{K}_\beta({\btheta}_0)$, 
and subsequently the asymptotic variance of matrix of the MCDPDE combining it with $\bm{J}_\beta^{-1}({\btheta}_0)$. 

\begin{remark}[Computational cost]
Evaluating $\bm{J}_\beta({\btheta}_0)$ costs $O(K)=O(d^2)$ elementary operations ($K=\binom{d}{2}$ pair blocks, each a $2\times2$ solve). 
Evaluating $\bm{K}_\beta(\theta)$ exactly via \eqref{EQ:Kbeta-model-MVN} costs $O(K^2)=O(d^4)$, 
since every ordered pair of composite components requires forming and inverting a matrix of size $2$, $3$, or $4$. 
This is practical for smaller dimensions $d$ up to roughly $15$--$20$, but becomes prohibitive at larger dimensions. 
For such settings the empirical sandwich estimator from \eqref{EQ:empirical-sandwich} can be used, 
which requires no cross-pair integration at all and is the recommended estimator in practice.
\qed
\end{remark}

\begin{remark}[Comparison with usual MDPDEs under multivariate Gaussian model]
	For the usual MDPDEs of $\btheta$ obtained under multivariate Gaussian model distribution, 
	let us represent the elements of the $p$-dimensional score vector $\bm{u}_{\btheta}$ by $u_{\theta_a}$ 
	corresponding to the component $\theta_a$ of $\btheta$. Then, for the first $d$ components related to the mean parameters, 
	we get $u_{\mu_j}(\bm{x})=\bm{e}_j^\top\bm{\Sigma}^{-1}(\bm{x} - \bm{\mu})$  for $j=1, \ldots, d$,
	which are linear in $\bm{x}$. For all remaining (scale and correlation) components $\theta_a$ of $\btheta$ 
	affects only $\bm{\Sigma}$, and the corresponding score is given by 
	$u_{\theta_a}(\bm{x})=(\bm{x}-\bm{\mu})^\top \bm{B}_{\theta_a} (\bm{x}-\bm\mu) + c_{\theta_a}$, 
	with $\bm{B}_{\theta_a}=\tfrac12\bm{\Sigma}^{-1}\frac{\partial\bm{\Sigma}}{\partial\theta_a}\bm{\Sigma}^{-1}$, $c_{\theta_a}=-\tfrac12\tr\left(\bm{\Sigma}^{-1}\frac{\partial\bm{\Sigma}}{\partial\theta_a}\right)$. 
	In particular, for $\theta_a=\sigma_j^2$ ($j=1, \ldots, d$) and $\theta_a = \rho_{jk}$ ($1\leq j <k \leq d$), we have 
	\begin{align}
		\bm{B}_{\sigma_j^2} &= \tfrac12\bm{\Sigma}^{-1}\Big[\bm{e}_j\bm{e}_j^\top
		+\tfrac{1}{2\sigma_j^2}(\bm{e}_j\bar{\bm{c}}_j^\top +\bar{\bm{c}}_j\bm{e}_j^\top)\Big]\bm{\Sigma}^{-1}, 
		& c_{\sigma_j^2} &= -\tfrac12\tr\!\Big(\bm{\Sigma}^{-1}\tfrac{\partial\bm{\Sigma}}{\partial\sigma_j^2}\Big), 
		\label{eq:Bsigma}\\
		\bm{B}_{\rho_{jk}} &= \sigma_j\sigma_k\,\mathrm{sym}\big(\bm{\Sigma}^{-1}\bm{e}_j\bm{e}_k^\top\bm{\Sigma}^{-1}\big), 
		& c_{\rho_{jk}} &= -\sigma_j\sigma_k\,[\bm{\Sigma}^{-1}]_{jk}, 
		\label{eq:Brho2}
	\end{align}
where $\bar{\bm{c}}_j = \bm{\Sigma}\bm{e}_j$ with its $j$-th entry set to zero, and $\mathrm{sym}(\bm{M}) = \tfrac12(\bm{M}+\bm{M}^\top)$.
	Then, we can compute all quantities involved in the asymptotic covariance matrix of the MDPDEs through standard normal moments.
	For example, the elements of $\bm{m}_{\btheta}(\beta)=\int \bm{u}_{\btheta}(\bm{x})f_{\btheta}^{1+\beta}(\bm{x})d\bm{x}$ are given by 
	$m_{\mu_j}(\beta)=0$ and 
	$m_{\sigma_j^2}(\beta)=\kappa_\beta\left[\frac{1}{1+\beta}\tr(\bm{B}_{\sigma_j^2}\bm{\Sigma})+c_{\sigma_j^2}\right]$, 
	for $j=1, \ldots, d$, and 
	$m_{\rho_{jk}}(\beta)=\kappa_\beta\left[\frac{1}{1+\beta}\tr(\bm{B}_{\rho_{jk}}\bm{\Sigma})+c_{\rho_{jk}}\right]$
	for $j,k=1, \ldots, d$, with $\kappa_\beta=(2\pi)^{-\beta d/2}(1+\beta)^{-d/2}|\bm{\Sigma}|^{-\beta/2}$. 
	Similarly, all the required quantities can be computed in closed form.  
	However, unlike the case of MCDPDE requiring only $2\times 2$ solve of pairwise covariance matrices,  
	the computation of asymptotic variance of the MDPDEs required inverting/manipulating the full $d\times d$ matrix $\bm{\Sigma}$, 
	involving an $O(d^3)$ operation per trace and $O(p^2)=O(d^4)$ entries.
	\qed
\end{remark}

\subsubsection{Empirical sandwich estimator of the covariance of (unweighted) MCDPDEs}
The empirical sandwich estimator can be obtained from the formula \eqref{EQ:empirical-sandwich} of the main paper,
which involves only $\bm{\psi}_\beta$ and $\nabla_{\btheta}\bm{\psi}_\beta$.
The term $\bm{\psi}_\beta$ can be computed directly from the closed forms of $\bm{u}_{jk}$ and $\bm{m}_{jk}$ derived previously.
The term $\nabla_{\btheta}\bm{\psi}_\beta$ involves the gradient of these quantities, 
which also have nice and sparse closed forms with their (few) non-zero entries being 
(for all $1\leq j, k\leq d$)
\begin{align}
\frac{\partial u_{\mu_j}}{\partial\mu_j}  &= -\frac{1}{\sigma_j^2(1-\rho_{jk}^2)}, \qquad
\frac{\partial u_{\mu_j}}{\partial\mu_k} = \frac{\partial u_{\mu_k}}{\partial\mu_j} 
= \frac{\rho_{jk}}{\sigma_j\sigma_k(1-\rho_{jk}^2)},  
\notag\\
\frac{\partial u_{\mu_j}}{\partial\sigma_j^2} &= \frac{\partial u_{\sigma_j^2}}{\partial\mu_j} 
= \frac{\rho_{jk}\; z_k-2z_j}{2\sigma_j^3(1-\rho_{jk}^2)},  \qquad 
\frac{\partial u_{\mu_j}}{\partial\sigma_k^2}= \frac{\partial u_{\sigma_k^2}}{\partial\mu_j} 
= \frac{\rho_{jk}\; z_k}{2\sigma_j\sigma_k^2(1-\rho_{jk}^2)}
\notag\\
\frac{\partial u_{\sigma_j^2}}{\partial\sigma_j^2} &= \frac{2(1-\rho_{jk}^2)-4z_j^2+3\rho_{jk}\; z_jz_k}{4\sigma_j^4(1-\rho_{jk}^2)}, 		
\qquad
\frac{\partial u_{\sigma_j^2}}{\partial\sigma_k^2}  = \frac{\partial u_{\sigma_k^2}}{\partial\sigma_j^2} 
= \frac{\rho_{jk}\; z_jz_k}{4\sigma_j^2\sigma_k^2(1-\rho_{jk}^2)}, 
\notag\\
\frac{\partial u_{\mu_j}}{\partial\rho_{jk}} &=  \frac{\partial u_{\rho_{jk}}}{\partial\mu_j} 
= \frac{2\rho_{jk}\; z_j-z_k(1+\rho_{jk}^2)}{\sigma_j(1-\rho_{jk}^2)^2}, \qquad
\frac{\partial u_{\sigma_j^2}}{\partial\rho_{jk}}  = \frac{\partial u_{\rho_{jk}}}{\partial\sigma_j^2} 
= \frac{z_j\big[2\rho_{jk}\; z_j-z_k(1+\rho_{jk}^2)\big]}{2\sigma_j^2(1-\rho_{jk}^2)^2}, 
\notag\\
\frac{\partial u_{\rho_{jk}}}{\partial\rho_{jk}}  &= 
\frac{(1-\rho_{jk}^2)(1+\rho_{jk}^2)-(1+3\rho_{jk}^2)(z_j^2+z_k^2)+2\rho_{jk}(\rho_{jk}^2+3)z_jz_k}{(1-\rho_{jk}^2)^3}. 
\label{eq:hessian}
\end{align}
and
\begin{align}
\frac{\partial m_{\sigma_j^2}}{\partial\sigma_k^2}=\frac{\partial m_{\sigma_k^2}}{\partial\sigma_j^2} &= \frac{\beta^2}{(1+\beta)^2}\frac{\kappa_{jk}(\beta)}{4\sigma_j^2\sigma_k^2}, 
&	\frac{\partial m_{\sigma_j^2}}{\partial\sigma_j^2} &= \frac{\beta(2+\beta)}{(1+\beta)^2}\frac{\kappa_{jk}(\beta)}{4\sigma_j^4},
\notag\\
	\frac{\partial m_{\sigma_j^2}}{\partial\rho_{jk}}=\frac{\partial m_{\rho_{jk}}}{\partial\sigma_j^2} &= -\frac{\beta^2}{(1+\beta)^2}\frac{\rho_{jk}\;\kappa_{jk}(\beta)}{2\sigma_j^2(1-\rho_{jk}^2)}, &
\frac{\partial m_{\rho_{jk}}}{\partial\rho_{jk}} &= \frac{\beta(1+ (1+\beta)\rho_{jk}^2)}{(1+\beta)^2}
\frac{\kappa_{jk}(\beta)}{(1-\rho_{jk}^2)^2}, 
	\label{eq:gradm}
\end{align}
where $m_{\sigma_j^2}$, $m_{\sigma_k^2}$ and $m_{\rho_{jk}}$ denotes the three non-zero terms of $\bm{m}_{jk}$ 
corresponding to $\sigma_j^2$, $\sigma_k^2$ and $\rho_{jk}$, respectively.  
Note that, these satisfy the identity 
\begin{equation*}
\nabla_{\btheta} \bm{m}_{jk}(\btheta) = \int \nabla_{\btheta} \bm{u}_{jk}(\bm{x})f_{\btheta, jk}^{1+\beta}(\bm{x})\,d\bm{x} 
+ (1+\beta)\int \bm{u}_{jk}(\bm{x}) \bm{u}_{jk}^\top(\bm{x}) f_{\btheta, jk}^{1+\beta}(\bm{x})\,d\bm{x},
\end{equation*} 
which reduces at $\beta=0$ to $\nabla_{\btheta} \bm{m}_{jk}({\btheta}_0)=\bm{O}$, consistent with the theoretical expectations.

\subsection{Pairwise McKay Bivariate Gamma Composite Model}
\label{APP:Algo_GammaExp}

We next consider the pairwise composite likelihood constructed from McKay's bivariate gamma distribution 
as discussed in Example \ref{EX:GammaExp}. This model is particularly useful for illustrating that 
the proposed methodology applies beyond elliptical distributions and Gaussian-based dependence structures.

For a pairwise component \((j,k)\), recall that $f_{\boldsymbol{\theta},jk}(x_j,x_k)$ denote the corresponding McKay bivariate 
gamma density  defined in (\ref{EQ:McKay_2density}), with parameter $\btheta=(\delta_1,\ldots,\delta_d,\lambda)^\top\in\mathbb R_+^{d+1}$.
Interestingly, unlike the Gaussian model of Appendix \ref{APP:Algo_MVN}, a single pair's density depends on \emph{many} of the original $\delta_i$'s 
(all of $\delta_1,\ldots,\delta_k$) rather than just two of them. Standard integration using the Gamma function yields 
\begin{equation*}
\int f_{\btheta,jk}^{1+\beta}(\bm{x}_{jk})d\bm{x}_{jk}=\lambda^{2\beta}C_{\beta}\left(\delta_j^*, \delta_k^* - \delta_j^*\right), 
~~~~~~ 1\leq j < k  \leq d,
\end{equation*}
with $C_\beta(\cdot, \cdot)$ being as defined in (\ref{EQ:McKay_C}) and $\delta_j^* = \delta_1 + \cdots + \delta_j$ for all $j=1, \ldots, d$.
This has been used to construct the objective functions for the MCDPDEs of $\btheta$ in Example \ref{EX:GammaExp}.

\subsubsection{Derivation of Estimating equations}
For a pair  $(j,k)$, $1\le j<k\le d$, differentiating $\log f_{\btheta,jk}$ with respect to the components of $\btheta$ gives
\begin{eqnarray}
u_{\delta_a, jk}(\bm{x}_{jk}) = \frac{\partial \log f_{\btheta,jk}(\bm{x}_{jk}) }{\partial\delta_a} &=& \left\{
\begin{array}{ll}
\log\lambda-\psi\left(\delta_j^*\right)+\log x_j, 			& a = 1, \ldots, j, \\
\log\lambda-\psi(\delta_k^* - \delta_j^*)+\log(x_k-x_j),  	& a = (j+1), \ldots, k, \\
0 										& a > k,
\end{array}
\right.
\nonumber\\
u_{\lambda, jk}(\bm{x}_{jk}) = \frac{\partial \log f_{\btheta,jk}(\bm{x}_{jk}) }{\partial\lambda} &=&  \frac{\delta_k^*}{\lambda} - x_k,
\label{EQ:McKay_derivative}
\end{eqnarray}
where $\psi(\cdot)$ denotes the digamma function. 
Using these, we can express the pairwise score function $\bm{u}_{\btheta, jk}(\bm{x})$ compactly as 
\begin{eqnarray}
\bm{u}_{\btheta, jk}(\bm{x}_{jk}) = \left(u_{\delta_1, jk}(\bm{x}_{jk}), \ldots, u_{\delta_d, jk}(\bm{x}_{jk}), u_{\lambda, jk}(\bm{x}_{jk})\right)^\top
= \bm{L}_{jk}\;\widetilde{\bm{u}}_{\btheta, jk}(\bm{x}_{jk}),
\label{EQ:McKay_pairScore}
\end{eqnarray}
where $\widetilde{\bm{u}}_{\btheta, jk}(\bm{x}_{jk}) = \left(u_{\delta_j, jk}(\bm{x}_{jk}), u_{\delta_k, jk}(\bm{x}_{jk}), u_{\lambda, jk}(\bm{x}_{jk})\right)^\top$, and the fixed $0/1$ local-to-global spread matrix is given by 
\begin{equation*}
\bm{L}_{jk}=\begin{pmatrix} \bm{1}_j & \bm{0}_j & \bm{0}_j\\ 
							\bm{0}_{k-j} & \bm{1}_{k-j} & \bm{0}_{k-j}\\ 
							\bm{0}_{d-k} & \bm{0}_{d-k} & \bm{0}_{d-k}\\ 
							0 & 0 & 1 \end{pmatrix}\in\{0,1\}^{(d+1)\times 3}.
\end{equation*}
Since $\bm{L}_{jk}$ does not depend on $\btheta$, every local object corresponding to the three stochastic terms in \eqref{EQ:McKay_pairScore}
spreads into the full space via the same $\bm{L}_{jk}$ multiplications.
In particular, we have $\bm{m}_{jk}(\btheta) = \bm{L}_{jk} \widetilde{\bm{m}}_{jk}(\btheta)$ with 
\begin{equation*}
\widetilde{\bm{m}}_{jk}(\btheta) = \int \left(u_{\delta_j, jk}(\bm{x}_{jk}), u_{\delta_k, jk}(\bm{x}_{jk}), u_{\lambda, jk}(\bm{x}_{jk})\right)^\top f_{\btheta,jk}^{1+\beta}(\bm{x}_{jk})d\bm{x}_{jk}.
\end{equation*}
These three integrals then can be evaluated using the facts that, for $Z\sim\mathrm{Gamma}(\alpha,r)$, $\E[\log Z]=\psi(\alpha)-\log r$, $\Var(\log Z)=\psi_1(\alpha)$, $\Var(Z)=\alpha/r^2$, $\Cov(Z,\log Z)=1/r$,
which yields
\begin{equation}
	\widetilde{\bm{m}}_{jk}(\btheta)=\begin{pmatrix}
		\psi((1+\beta)\delta_j^* - \beta) - \psi(\delta_j^*) - \log(1+\beta)\\ 
		\psi((1+\beta)(\delta_k^*-\delta_j^*))-\psi(\delta_k^* - \delta_j^*)-\log(1+\beta)\\ 
		\frac{2\beta}{(1+\beta)\lambda}  \end{pmatrix}
	\lambda^{2\beta}C_{\beta}\left(\delta_j^*, \delta_k^* - \delta_j^*\right), 
\label{EQ:McKay_mjk}
\end{equation}
where $\psi(\cdot)$ and $\psi_1(\cdot)$ denotes the digamma and trigamma functions, respectively. 
Given all these simplifications, we can now write the minimum CDPD estimating equations for $\beta\geq 0$ as given by 
\begin{align}
 \sum_{i=1}^n & \sum_{1\le j<k\le d}\Big[m_{\delta_a,jk}(\btheta)\ -\ u_{\delta_a,jk}(\boldsymbol{x}_{jk}^{(i)})f_{\btheta,jk}^\beta(\boldsymbol{x}_{jk}^{(i)})\Big] = 0, 
	\qquad a=1,\ldots,d.
	\label{EQ:McKay_delta}	\\
\sum_{i=1}^n& \sum_{1\le j<k\le d}
	\left[\frac{2\beta\lambda^{2\beta-1}}{1+\beta} C_{\beta}(\delta_j^*, \delta_k^*)-  	 
	\left(\frac{\delta_k^*}{\lambda} - x_{ik}\right)f_{\boldsymbol{\theta},{jk}}^\beta(\boldsymbol{x}_{jk}^{(i)})\right]
	=0,
	\label{EQ:McKay_lambda}
\end{align}
where $m_{\delta_a,jk}(\btheta)$ denotes the $a$-th component of $\bm{m}_{jk}(\btheta)$ for $a=1, \ldots, d$.
At $\beta=0$, the weights $f_{\boldsymbol{\theta},jk}^{\beta}$ reduce to one and the above estimating equations coincide 
with those of the classical MCLE. 

\begin{remark}
The estimating equation \eqref{EQ:McKay_lambda} for $\lambda$ admits a direct fixed-point update. 
The equations \eqref{EQ:McKay_delta} for the shape parameters $\delta_j$, however, 
are nonlinear and do not generally possess closed-form solutions. 
In our empirical illustrations, we have computed the MLE, CMLE and MCDPDEs of $\lambda$ and $\bm{\delta}$ by simultaneously 
minimizing the corresponding objective functions using the limited-memory BFGS algorithm of \cite{byrd1995limited}
that allows the necessary box constraints restricting all parameters to be positive.
The implementation is available within our supplementary \texttt{R}-package \texttt{mvdpd} within CRAN. 
\qed
\end{remark}

\subsubsection{Derivation of the asymptotic covariance matrix of (unweighted) MCDPDEs}
Let us first derive the matrix $\bm{J}_\beta({\btheta}_0)$, indirectly to avoid computing expectations of the square of pairwise scores, 
through the following formulation (assuming all weights to be one)
\begin{eqnarray}
\bm{J}_\beta({\btheta}_0) &=&\sum_{1\le j<k\le d} \int \bm{u}_{{\btheta}_0,jk}(\bm{x}_{jk})\bm{u}_{{\btheta}_0,jk}^\top(\bm{x}_{jk}) f_{{\btheta}_0,jk}^{1+\beta}(\bm{x}_{jk})d\bm{x}_{jk} 
\nonumber\\
&=& \frac{1}{1+\beta}\left[\nabla_{\btheta} \bm{m}_{jk}(\btheta) - \int \nabla_{\btheta} \bm{u}_{{\btheta},jk}(\bm{x}_{jk}) f_{{\btheta},jk}^{1+\beta}(\bm{x}_{jk})d\bm{x}_{jk}  \right]_{\btheta={\btheta}_0}.
\nonumber
\end{eqnarray} 
This is because the derivatives of $\bm{m}_{jk}(\btheta)$ and $\nabla_{\btheta} \bm{u}_{{\btheta},jk}(\bm{x}_{jk})$ are much easier to compute for each pairs $1\leq j<k \leq d$.
In particular, $\nabla_{\btheta} \bm{u}_{{\btheta},jk}(\bm{x}_{jk})= \bm{L}_{jk}\;\widetilde{\bm{H}}_{jk}(\btheta)\bm{L}_{jk}^\top$
is a constant independent of $\bm{x}_{jk}$, since the local Hessian $\widetilde{\bm{H}}_{jk}(\btheta)$ is given by 
\begin{equation}
	\widetilde{\bm{H}}_{jk}(\btheta)=\begin{pmatrix}
		-\psi_1(\delta_j^*)&0&1/\lambda
		\\0&-\psi_1(\delta_k^*-\delta_j^*)&1/\lambda
		\\1/\lambda&1/\lambda&-\delta_k^*/\lambda^2
	\end{pmatrix}. 
\label{EQ:McKay_Htilde}
\end{equation}
Therefore, we simply get 
\begin{eqnarray*}
\int \nabla_{\btheta} \bm{u}_{{\btheta},jk}(\bm{x}_{jk}) f_{{\btheta}_0,jk}^{1+\beta}(\bm{x}_{jk})d\bm{x}_{jk} 
&=& \bm{L}_{jk}\widetilde{\bm{H}}_{jk}(\btheta)\bm{L}_{jk}^\top \int f_{{\btheta},jk}^{1+\beta}(\bm{x}_{jk})d\bm{x}_{jk} 
\\
&=& \bm{L}_{jk}\widetilde{\bm{H}}_{jk}(\btheta)\bm{L}_{jk}^\top \; \lambda^{2\beta}C_{\beta}\left(\delta_j^*, \delta_k^* - \delta_j^*\right).
\end{eqnarray*}
Further, differentiating $\bm{m}_{jk}(\btheta)$, we get 
\begin{equation*}
\nabla_{\btheta}\bm{m}_{jk}(\btheta) = \bm{L}_{jk} \nabla_{\btheta}\widetilde{\bm{m}}_{jk}(\btheta)\bm{L}_{jk}^\top
= \lambda^{2\beta}C_{\beta}\left(\delta_j^*, \delta_k^* - \delta_j^*\right) 
\bm{L}_{jk} \widetilde{\bm{M}}_{jk}^{(\beta)}(\btheta)\bm{L}_{jk}^\top
\end{equation*} 
where 
\begin{equation*}
\widetilde{\bm{M}}_{jk}^{(\beta)}(\btheta) = 
\begin{pmatrix}
	b_{\beta}(\delta_j^*) + (1+\beta)a_{\beta}^2(\delta_j^*) & (1+\beta)a_{\beta}(\delta_j^*)a_{\beta}(\delta_k^* - \delta_j^*) & \frac{2\beta}{\lambda}a_{\beta}(\delta_j^*)\\ 
	(1+\beta)a_{\beta}(\delta_j^*)a_{\beta}(\delta_k^* - \delta_j^*) & b_{\beta}(\delta_k^* - \delta_j^*)+ (1+\beta)a_{\beta}^2(\delta_k^* - \delta_j^*) & \frac{2\beta}{\lambda}a_{\beta}(\delta_k^* - \delta_j^*) \\
	\frac{2\beta}{\lambda}a_{\beta}(\delta_j^*)  & \frac{2\beta}{\lambda}a_{\beta}(\delta_k^* - \delta_j^*)   & \frac{2\beta(2\beta-1)}{(1+\beta)\lambda^2}
\end{pmatrix},
\end{equation*}
with $a_{\beta}(\delta) = \psi((1+\beta)\delta - \beta) - \psi(\delta) - \log(1+\beta)$ and  
$b_{\beta}(\delta) = (1+\beta)\psi_1((1+\beta)\delta-\beta) - \psi_1(\delta)$ for any $\delta>0$.  
Combining all these, we finally get 
\begin{equation}
\bm{J}_\beta({\btheta})=\sum_{1\le j<k\le d} \frac{\lambda^{2\beta}}{1+\beta} C_{\beta}\left(\delta_j^*, \delta_k^* - \delta_j^*\right)\; 
\bm{L}_{jk}\Big[\widetilde{\bm{M}}_{jk}^{(\beta)}(\btheta) - \widetilde{\bm{H}}_{jk}(\btheta) \Big]\bm{L}_{jk}^\top. 
\label{EQ:McKay_Jgamma}
\end{equation}

Next, we consider the variability matrix $\bm{K}_\beta(\btheta)$ as defined in \eqref{EQ:Kbeta-model},
which we will re-express as in \eqref{EQ:Kbeta-model-MVN} for the pairwise CL construction
with the double sum ranging over \emph{all ordered pairs} $(S,T)$ of composite components.
We can further simplify it using the matrix $\bm{L}_{jk}$ as 
\begin{equation}
\bm{K}_\beta(\btheta) = \sum_S\sum_T 
\bm{L}_S\Big[\widetilde{C}_{S,T}(\btheta)-\widetilde{\bm{m}}_S(\btheta)\widetilde{\bm{m}}_T^\top(\btheta)\Big]\bm{L}_T^\top, 
\label{EQ:Kbeta_mvGamma}
\end{equation}
where $\widetilde{C}_{S,T}(\btheta) = \E_G[\widetilde{\bm{u}}_{\btheta, S}\widetilde{\bm{u}}_{\btheta, T}^\top f_{\btheta, S}^\beta f_{\btheta, T}^\beta]$.
The own-pair blocks (with $S=T=(j,k)$) can again be computed in  closed form as 
\begin{equation*}
\widetilde{C}_{S,S}(\btheta) = \int \widetilde{\bm{u}}_{\btheta, S}\widetilde{\bm{u}}_{\btheta, S}^\top f_{\btheta, S}^{1+2\beta}
= \frac{\lambda^{4\beta}}{1+2\beta} C_{2\beta}\left(\delta_j^*, \delta_k^* - \delta_j^*\right)\; 
\Big[\widetilde{\bm{M}}_{jk}^{(2\beta)}(\btheta) - \widetilde{\bm{H}}_{jk}(\btheta) \Big]. 
\end{equation*}
For Cross-pair terms (with $S\ne T$), $\widetilde{C}_{S,T}$ requires the joint distribution of $\bm{X}$ restricted to $A=S\cup T$. 
Because $X_1<\cdots<X_d$ are cumulative sums of independent gamma increments $W_j = X_j-X_{j-1}$, 
the joint density of any subset of indices decomposes into independent gamma \emph{gaps}.
More precisely, sorting $A$'s breakpoints $0<b_1<\cdots<b_m$ ($m=|A|\le4$), 
the gaps $\Delta_1=X_{b_1},\Delta_2=X_{b_2}-X_{b_1},\ldots$ are mutually independent $\mathrm{Gamma}(\gamma_i,\lambda)$-distributed  
with $\gamma_i=\delta_{b_{i-1}+1}^*-$, and $U_S,V_S,U_T,V_T$ are each sums of a contiguous block of these gaps.
Unfortunately, these cross terms do not reduce to an elementary (Gamma-function) closed form; 
it remains a perfectly well-defined, low-dimensional integral given by 
\begin{equation}
	\widetilde{C}_{S,T}(\btheta)=\int \widetilde{\bm{u}}_{\btheta, S}(\bm{x}_S)\widetilde{\bm{u}}_{\btheta, T}^\top(\bm{x}_T) 
	f_{\btheta, S}^\beta(\bm{x}_S) f_{\btheta, T}^\beta(\bm{x}_T) \;g_A(\bm{x}_A)\,d\bm{x}_A. 
	\label{eq:crosspairintegral}
\end{equation}
These integrals may be evaluated by drawing $M$ Monte Carlo samples of the (at most $4$) independent gaps from their \emph{true}, 
untilted $\mathrm{Gamma}(\gamma_i,\lambda)$ distribution, forming $U_S,V_S,U_T,V_T$ as the corresponding gap-sums, 
and averaging $u_S\,u_T^\top\,f_S^\beta f_T^\beta$ over the draws. This provides an unbiased, 
arbitrarily-precise  estimate of \eqref{eq:crosspairintegral} as $M\to\infty$. 

Once we have computed $C_{S,T}$ for all $K=\binom d2$ composite components $S,T$ (including $S=T$), 
we can use them to compute $\bm{K}_\beta(\btheta)$  from \eqref{EQ:Kbeta_mvGamma}, and subsequently compute the asymptotic variance matrix of the (unweighted) MCDPDE for this specific model.

\subsubsection{The practical alternative: empirical sandwich estimator}

Because exact cross-pair evaluation costs $O(K^2)=O(d^4)$ Monte Carlo integrations, 
in practice we may use the empirical sandwich estimator following \eqref{EQ:empirical-sandwich} of the main paper.
This involves only $\bm{\psi}_\beta$ and $\nabla_{\btheta}\bm{\psi}_\beta$,
which can be computed efficiently from the previously derived formulas as follows: 
\begin{eqnarray*}
\boldsymbol{\psi}_\beta(\boldsymbol{x}|\boldsymbol{\theta}) 
&=& \sum_{1\le j<k\le d} \bm{L}_{jk}\Big[\widetilde{\bm{m}}_{jk}(\btheta) - 
\widetilde{\boldsymbol{u}}_{\boldsymbol{\theta},jk}(\boldsymbol{x}) f_{\boldsymbol{\theta},jk}^{\beta}(\boldsymbol{x}) \Big]. 
\\
\nabla_{\btheta}\boldsymbol{\psi}_\beta(\boldsymbol{x}|\boldsymbol{\theta}) 
&=& \sum_{1\le j<k\le d} \bm{L}_{jk}\Big[\nabla_{\btheta}\widetilde{\bm{m}}_{jk}(\btheta) - 
\widetilde{\bm{H}}_{jk}(\btheta) f_{\boldsymbol{\theta},jk}^{\beta}(\boldsymbol{x}) -
\beta \widetilde{\boldsymbol{u}}_{\boldsymbol{\theta},jk}(\boldsymbol{x}) \widetilde{\boldsymbol{u}}_{\btheta,jk}^\top(\boldsymbol{x}) 
f_{\boldsymbol{\theta},jk}^{\beta}(\boldsymbol{x}) \Big]\bm{L}_{jk}^\top. 
\end{eqnarray*}

\paragraph{Asymptotic variance of the MDPDEs under multivariate ordered gamma model}

For the standard MDPDEs of $\btheta$ computed directly under the multivariate ordered gamma model (\ref{EQ:density_GammaExp}), 
the associated asymptotic covariance matrix $\bm{J}_\beta^{-1}(\btheta)\bm{K}_\beta(\btheta)\bm{J}_\beta^{-1}(\btheta)$ 
can be computed in closed form as follows, whose implementation is available in our \texttt{R}-package \texttt{mvdpd} for comparison. 

The score function $\bm{u}_{\btheta}(\bm{x})$ associated with the multivariate model density \eqref{EQ:density_GammaExp} is given by
$u_{\delta_j}(\bm{x}) = \log\lambda-\psi(\delta_j)+\log(x_j  - x_{j-1})$ for $j=1, \ldots, d$, and 
$u_\lambda(\bm{x}) = \frac{1}{\lambda}\sum_{j=1}^d\delta_j -x_d$. 
Then, using the moment identity of Gamma distribution as before, we obtain 
\begin{eqnarray}
\bm{m}_\beta(\btheta) &=& \int \bm{u}_{\btheta}(\bm{x}) f_{\btheta}^{1+\beta}(\bm{x})d\bm{x}
=\kappa_\beta(\btheta)\widetilde{\bm{m}}_\beta(\btheta),
\nonumber
\\
\bm{J}_\beta(\btheta) &=& \int \bm{u}_{\btheta}(\bm{x})\bm{u}_{\btheta}^\top(\bm{x}) f_{\btheta}^{1+\beta}(\bm{x})d\bm{x} = 
\kappa_\beta(\btheta)\big[\widetilde{\bm{\Sigma}}_\beta(\btheta) 
+ \widetilde{\bm{m}}_\beta(\btheta)\widetilde{\bm{m}}_\beta^\top(\btheta)\big], 
\label{EQ:McKay_Jfull}
\\
\bm{K}_\beta(\btheta) &=&  \int \bm{u}_{\btheta}(\bm{x})\bm{u}_{\btheta}^\top(\bm{x}) f_{\btheta}^{1+2\beta}(\bm{x})d\bm{x}
-\bm{m}_\beta(\btheta)\bm{m}_\beta^\top(\btheta)
= \bm{J}_{2\beta}(\btheta) - \kappa_\beta^2(\btheta)\widetilde{\bm{m}}_\beta(\btheta)\widetilde{\bm{m}}_\beta^\top(\btheta),
\nonumber
\end{eqnarray}
where $\kappa_\beta(\btheta) = \lambda^{d\beta} \prod\limits_{j=1}^d  
\frac{\Gamma(\delta_j(1+\beta)-\beta)}{\Gamma(\delta_j)^{1+\beta}(1+\beta)^{\delta_j(1+\beta)-\beta}}$,
and 
\begin{equation}
	\widetilde{\bm{m}}_\beta(\btheta) = \begin{pmatrix}
		\psi(\delta_1(1+\beta)-\beta)-\psi(\delta_1)-\log(1+\beta)
		\\ \vdots\\ 
		\psi(\delta_d(1+\beta)-\beta)-\psi(\delta_d)-\log(1+\beta)
		\\ 
		\frac{d\beta}{(1+\beta)\lambda}
	\end{pmatrix}, 
\label{EQ:McKay_mfull}
\end{equation}
\begin{equation}
	\widetilde{\bm{\Sigma}}_\beta(\btheta) = \begin{pmatrix} 
		\psi_1(\delta_1(1+\beta)-\beta) & \hdots & 0 & -\frac{1}{\lambda(1+\beta)}\\ 
		& \ddots & & \vdots\\ 
		0 & \hdots & \psi_1(\delta_d(1+\beta)-\beta) & -\frac{1}{\lambda(1+\beta)}\\ 
		-\frac{1}{\lambda(1+\beta)} & \cdots & -\frac{1}{\lambda(1+\beta)} &  \frac{1}{\lambda^2(1+\beta)^2}\sum_{j=1}^d(\delta_j(1+\beta)-\beta) 
	\end{pmatrix}. \nonumber
\label{EQ:McKay_Sigfull}
\end{equation}
We can also derive explicit inverse of $\bm{J}_\beta$ via Sherman--Morrison--Woodbury theorem as given by 
\begin{equation}
	\bm{J}_\beta^{-1}(\btheta) = \frac{1}{\kappa_\beta(\btheta)}\left[\widetilde{\bm{\Sigma}}_\beta^{-1}(\btheta)  
	-\frac{\widetilde{\bm{\Sigma}}_\beta^{-1}(\btheta)\widetilde{\bm{m}}_\beta(\btheta)\widetilde{\bm{m}}_\beta^\top(\btheta)
	\widetilde{\bm{\Sigma}}_\beta^{-1}(\btheta)}{1+\widetilde{\bm{m}}_\beta^\top(\btheta)\widetilde{\bm{\Sigma}}_\beta^{-1}(\btheta)
	\widetilde{\bm{m}}_\beta(\btheta)}\right], 
\label{EQ:McKay_Jinv}
\end{equation}
where each element of the inverse of the arrow matrix $\widetilde{\bm{\Sigma}}_\beta$ can also be computed via its Schur complement
$c = \frac{1}{\lambda^2(1+\beta)^2}\left[\sum_{j=1}^d(\delta_j(1+\beta)-\beta) - \sum_{j=1}^d\frac{1}{\psi_{1,j}}\right]$ with $\psi_{1,j} = \psi_1(\delta_j(1+\beta)-\beta)$ for all $j\geq 1$, and is given by 
\begin{eqnarray}
	\left[\widetilde{\bm{\Sigma}}_\beta^{-1}(\btheta)\right]_{jk} =\left\{
	\begin{array}{ll}
\frac{1}{\psi_{1,j}}+\frac{1}{(1+\beta)^2\lambda^2c\,\psi_{1,j}^2} & \mbox{ if } k=j \in\{1, \ldots, d\},\\
\frac{1}{(1+\beta)^2\lambda^2 c\,\psi_{1,j}\psi_{1,k}} &  \mbox{ if } k, j\in\{1, \ldots, d\},\; k\neq j,\\
\frac{1}{(1+\beta) \lambda c\,\psi_{1,j}}, & \mbox{ if }j \in\{1, \ldots, d\},\; k=d+1,\\
\frac1c, & \mbox{ if }j = k=d+1.
	\end{array}
	\right.
\end{eqnarray}

At $\beta=0$, we get $\kappa_0=1$, $\widetilde{\bm{m}}_0=\bm{0}_{d+1}$,
and hence $\bm{J}_0=\bm{K}_0=\widetilde{\bm{\Sigma}}_0=diag(\psi_1(\delta_1),\ldots,\psi_1(\delta_d))$ bordered by $-1/\lambda$ and $\sum_j\delta_j/\lambda^2$, which is exactly the classical Fisher information of $d$ independent $\mathrm{Gamma}(\delta_j,\lambda)$ variables sharing a common rate $\lambda$, recovering the ordinary MLE.

\section{Proof of Results}
\label{APP:Proofs}

\subsection{Proof of Theorem \ref{THM:MCDPDE_asymp}}
\label{APP:Proof_Asymp}

The proof follows along the same line as the proof of asymptotic properties of ordinary MDPDEs available in \citet{Basu2011}.
We first prove that, with probability tending to 1, there exists a sequence of solutions (MCDPDE) to 
the estimating equation \eqref{EQ:MCDPDE_EstEqn} and this sequence of MCDPDs is indeed consistent for the MCDPDF 
$\boldsymbol{\theta}_\beta^\ast = \boldsymbol{T}_\beta(G)$ at the true distribution $G$ which may be outside the model family. 
For this purpose, let us fix $\beta\geq 0$ and denote the empirical and population objective functions as 
\begin{equation*}
H_{n,\beta} (\btheta) = \frac{1}{n} \sum_{i=1}^n V_{\beta}(\boldsymbol{x}_i|{\boldsymbol{\theta}}), 
\mbox{ and }
H_{\beta} (\btheta) = \_G\left[ V_{\beta}(\boldsymbol{x}_i|{\boldsymbol{\theta}})\right], 
\end{equation*}
respectively, and study the behavior of $H_{n,\beta}(\btheta)$ on a sphere $Q_a$ with the center at ${\btheta}_\beta^\ast$ and radius $a$.
Expanding $H_{n, \beta}(\btheta)$ using Taylor series around  ${\btheta}_\beta^\ast$, and scaling by $(1+\beta)$, we get 
\begin{align*}
	&\frac{H_{n, \beta}({\btheta}_\beta^\ast)-H_{n, \beta}(\btheta)}{1+\beta} 
	=-\frac{H_{n, \beta}(\theta) - H_{n, \beta}({\btheta}_\beta^\ast)}{1+\beta} \\
&= -\bm{A}^\top (\btheta - {\btheta}_\beta^\ast)   +\frac{1}{2} (\btheta - {\btheta}_\beta^\ast)^\top\bm{B}(\btheta - {\btheta}_\beta^\ast) 
\\
&\qquad +\frac{1}{6}\sum_j \sum_k \sum_l (\theta_j-{\theta}_{\beta,j}^\ast)(\theta_k-{\theta}_{\beta,k}^\ast)
(\theta_l-{\theta}_{\beta,l}^\ast)	\frac{1}{n} \sum_{i=1}^n \gamma_{jkl}(\bm{x}_i)M_{jkl}(\bm{x}_i)  \\
	&= S_1 + S_2 + S_3
\end{align*}
where $\bm{A} = \frac{1}{1+\beta} \nabla H_{n, \beta}(\btheta)|_{\btheta = {\btheta}_\beta^\ast}$, 
$\bm{B} = -\frac{1}{1+\beta} \nabla^2 H_{n, \beta}(\btheta)|_{\btheta = {\btheta}_\beta^\ast}$,
${\theta}_{\beta,j}^\ast$ denotes the $j$-th component of ${\btheta}_{\beta}^\ast$, 
$M_{jkl}$ is the $G$-integrable envelop functions dominating the indicated third order partial derivative of $V_\beta(\bm{x}\mid\btheta)$ 
from Assumption \ref{ASS:Asmp4}, and $0 \leq |\gamma_{jkl}(x)| \leq 1$  for all $j, k, l = 1, \ldots, p$. 

For the first term $S_1$, by the weak law of large numbers (WLLN), we get 
\begin{eqnarray}
	\bm{A} & = &  \sum_{k=1}^K w_k \int \bm{u}_{{\btheta}_\beta^\ast,k}(\bm{x}) f_{{\btheta}_\beta^\ast,k}^{1+\beta}(\bm{x})d\bm{x} 
- \sum_{k=1}^K w_k \frac{1}{n}\sum_{i=1}^n \boldsymbol{u}_{{\btheta}_\beta^\ast,k}(\boldsymbol{x}_i) 
f_{{\btheta}_\beta^\ast,k}^{\beta}(\boldsymbol{x}_i)
\nonumber\\
	& \xrightarrow{\mathcal{P}} & \sum_{k=1}^K w_k 
	\int \boldsymbol{u}_{{\btheta}_\beta^\ast,k}(\bm{x}) f_{{\btheta}_\beta^\ast,k}^{1+\beta}(\bm{x})d\bm{x} 
- \sum_{k=1}^K w_k \int  \boldsymbol{u}_{{\btheta}_\beta^\ast,k}(\boldsymbol{x}) f_{{\btheta}_\beta^\ast,k}^{\beta}(\boldsymbol{x})d\bm{x}
\nonumber \\
	& = & \frac{1}{1+\beta}  \nabla   H_\beta(\btheta)|_{\btheta={\btheta}_\beta^\ast}  = \bm{0}.\nonumber 
\end{eqnarray}
Thus, for any given $a>0$, each elements of $\bm{A}$ can be bounded in absolute value by $a^2$,
so that we get $|S_1| < p a^3$ with probability tending to 1.

By similar arguments using WLLN, we get $\bm{B} \xrightarrow{\mathcal{P}} -\bm{J}_{\beta}({\btheta}_\beta^\ast)$.
We then write  
\begin{equation*}
2 S_2 =
(\btheta - {\btheta}_\beta^\ast)^\top [-\bm{J}_{\beta}({\btheta}_\beta^\ast)](\btheta - {\btheta}_\beta^\ast) +
(\btheta - {\btheta}_\beta^\ast)^\top [\bm{B} + \bm{J}_{\beta}({\btheta}_\beta^\ast)] (\btheta - {\btheta}_\beta^\ast).
\end{equation*} 
The second term in the above is a sum of finitely many components each of which goes to zero in probability,
and so it the whole term. Thus, for any given $a>0$, its absolute value remains bounded by $p^2 a^3$ with probability tending to 1.  
The first term is, by Assumption \ref{ASS:Asmp5}, a negative (nonrandom) quadratic form in $(\btheta - {\btheta}_\beta^\ast)$.  
By an orthogonal transformation, this can be reduced to a diagonal form $\sum_{j=1}^p \lambda_i \xi^2_i\leq \lambda_1 a^2$,  
where $\lambda_p \leq \lambda_{p-1} \leq \ldots \leq \lambda_1 < 0$ are eigenvalues of $-\bm{J}_{\beta}({\btheta}_\beta^\ast)$
and $\sum_i \xi^2_i = a^2$ for all $\btheta$ on the surface of $Q_a$.
Combining these two bounds, there exist $c > 0$, $a_0 > 0$ such that $S_2 < - c a^2$ with probability tending to 1,  for all $a < a_0$.

Finally, with probability tending to 1, $|\frac{1}{n}\sum M_{jkl}(\bm{x}_i)| < 2 \E_G[M_{jkl}(\bm{X})]$ for all $i\geq 1$, 
and hence $|S_3| < ba^3$ on $Q_a$ where $b = \frac{1}{3} \sum\sum\sum \E_G[M_{jkl}(\bm{X})]<\infty$ by Assumption \ref{ASS:Asmp4}.  
Combining the bounds of the three terms, we finally get
\begin{equation*}
\max(S_1 + S_2 + S_3) < -ca^2 + (b+p)a^3 < 0, 
\end{equation*}
for any  $a < c/(b+p)$. Therefore,  $H_{n,\beta}(\theta) > H_{n,\beta}({\btheta}_\beta^\ast)$ for all points $\btheta$ 
on the surface of $Q_a$  with probability tending to one, for any sufficiently small $a>0$.
Then, $H_{n, \beta}(\theta)$ would have a local minimum in the interior of $Q_a$, satisfying \eqref{EQ:MCDPDE_EstEqn}.  
In other words, for all sufficiently small $a>0$, there exists a sequence of roots $\widehat{\theta}(a)$ to
the MCDSPDE estimating equation \eqref{EQ:MCDPDE_EstEqn}  within $Q_a$, with probability tending to 1 as $n \rightarrow \infty$. 
As these roots belong to $Q_a$, we already have 
$P(||\widehat{\btheta}(a) - {\btheta}_\beta^\ast||_2 < a) \rightarrow 1$ as $n\rightarrow\infty$.  

It now remains to show that we can determine such a sequence independently of $a$.
This can be done easily by considering the root $\widehat{\btheta}_{n,\beta}$ closest to ${\btheta}_\beta^\ast$,
which exists due to the continuity of $H_{n, \beta}(\btheta)$ as a function of $\btheta$.  
Then, clearly $P(||\widehat{\btheta}_{n,\beta} - {\btheta}_\beta^\ast||_2 < a) \rightarrow 1$ for all $a > 0$,
showing the existence of a sequence of consistent solutions $\widehat{\btheta}_{n,\beta}$ to 
the MCDSPDE estimating equation \eqref{EQ:MCDPDE_EstEqn} with probability tending to 1.

Finally, we now prove of the asymptotic normality of these MCDPDEs $\widehat{\btheta}_{n,\beta}$ via standard Taylor series arguments.  
Let $H_{n, \beta}^{j}$, $H_{n, \beta}^{jk}$ and $H_{n, \beta}^{jkl}$ denote the first second and third partial derivatives of 
$H_{n, \beta}$ with respect to the indicated indices of $\btheta$ (in the superscript). 
Expanding $H_{n, \beta}^j(\btheta)$ about ${\btheta}_\beta^\ast$ and evaluating at $\btheta=\widehat{\btheta}_{n,\beta}$, we get 
\begin{equation*}
0 = H_{n, \beta}^j(\widehat{\btheta}_{n,\beta}) = H_{n, \beta}^j({\btheta}_\beta^\ast) 
+ \sum_k (\theta_k - {\theta}_{\beta, k}^\ast) H_{n, \beta}^{jk}({\btheta}_\beta^\ast) 
+ \frac{1}{2} \sum_k \sum_l (\theta_k - {\theta}_{\beta, k}^\ast) (\theta_l - {\theta}_{\beta, l}^\ast) H_{n, \beta}^{jkl}({\btheta}_*),
\end{equation*} 
where  ${\btheta}_*$ is a point on the line segment connecting $\widehat{\btheta}_{n,\beta}$ and ${\btheta}_\beta^\ast$. 
Rearranging and scaling suitably, we get the vector form
\begin{equation}	
\bm{J}_{n} \left[\sqrt{n}(\widehat{\btheta}_{n,\beta} - {\btheta}_{\beta}^\ast)\right] 
= -\sqrt{n} \nabla H_{n,\beta}({\btheta}_\beta^\ast), 
\label{EQ:9Lehmann_1}
\end{equation}
with the $(j,k)$-th elements of the matrix $\bm{J}_n$ being 
\begin{align*}
	J_{jk,n} &= \left[H^{jk}_{n,\beta}({\btheta}_\beta^\ast) 
	+ \frac{1}{2}\sum_{l=1}^p(\hat{\theta}_l -{\theta}_{\beta, l}^\ast)H^{jkl}_{n,\beta}(\theta^*)\right],
	~~ j, k = 1, \ldots, p.
\end{align*}

Now, from Assumption \ref{ASS:Asmp4}, each $H^{jkl}_n(\theta^*)$ is bounded with probability tending to 1, 
so that the consistency of $\widehat{\btheta}_{n,\beta}$ implies that the second term of $J_{jk,n}$ converges to zero in probability.  
Further, following our earlier arguments with WLLN, we see that $H^{jk}_n({\btheta}_\beta^\ast)$ converges in probability to
$(1+\beta)$ times the $(j,k)$-th entry of the matrix $\bm{J}_{\beta}({\btheta}_\beta^\ast)$.
Thus, we get $\bm{J}_n \xrightarrow{\mathcal{P}}  (1 + \beta) \bm{J}_{\beta}({\btheta}_\beta^\ast)$. Next, we note that 
\begin{equation}
-\sqrt{n} \nabla H_{n,\beta}({\btheta}_\beta^\ast) 
= - \sqrt{n} \frac{(1+\beta)}{n} \sum_{i=1}^n \bm{\psi}_{\beta}(\bm{x}_i|{\btheta}_\beta^\ast).
	\label{EQ:9T_jn}
\end{equation} 
But, by the definition of ${\btheta}_\beta^\ast$, we have  $\E_G[\bm{\psi}_{\beta}(\bm{x}_i|{\btheta}_\beta^\ast)] = \bm{0}$,
and therefore $\Cov_G[\bm{\psi}_{\beta}(\bm{x}_i|{\btheta}_\beta^\ast)] =  \bm{K}_{\beta}({\btheta}_\beta^\ast)$
for each $i = 1, \ldots, n$. 
It then follows by an application of the central limit theorem, together with Equations~(\ref{EQ:9T_jn}), 
that the limiting distribution of $-\sqrt{n} \nabla H_{n,\beta}({\btheta}_\beta^\ast) $ is multivariate normal with mean vector zero
and covariance matrix $(1+\beta)^2 \bm{K}_{\beta}({\btheta}_\beta^\ast)$.

Therefore, by Lemma 4.1 of \citet[][page 432--433]{lehmann1983theory}, the solutions, 
$\sqrt{n}(\widehat{\btheta}_{n,\beta} - {\btheta}_{\beta}^\ast)$,  of the linear equations in (\ref{EQ:9Lehmann_1}) 
has the limiting distribution same as the distribution of the solution $\bm{Y}$ of the linear equations 
$(1+\beta)\bm{J}_\beta({\btheta}_{\beta}^\ast)\bm{Y} = \bm{Z}$ where $\bm{Z}\sim N_p(\bm{0}_p, (1+\beta)^2\bm{K}_{\beta}({\btheta}_\beta^\ast))$. Under Assumption \ref{ASS:Asmp5},
this solution $\bm{Y}$ is explicitly given by $\bm{Y} = (1+\beta)^{-1}\bm{J}_{\beta}^{-1}({\btheta}_\beta^\ast)\bm{Z}$.
Therefore, its distribution, and hence the asymptotic distribution of $\sqrt{n}(\widehat{\btheta}_{n,\beta} - {\btheta}_{\beta}^\ast)$, 
is multivariate normal with mean vector $\bm{0}_p$ and covariance matrix 
\begin{align}
\left[(1+\beta)^{-1}\bm{J}_{\beta}^{-1}({\btheta}_\beta^\ast)\right] \cdot\left[(1 + \beta)^2 \bm{K}_{\beta}({\btheta}_\beta^\ast)\right] 
\cdot\left[(1+\beta)^{-1}\bm{J}_{\beta}^{-1}({\btheta}_\beta^\ast)\right] 
= \bm{J}_{\beta}^{-1}({\btheta}_\beta^\ast)\bm{K}_{\beta}({\btheta}_\beta^\ast) \bm{J}_{\beta}^{-1}({\btheta}_\beta^\ast).
\nonumber
\label{EQ:9asymp_variance}
\end{align}
This completes the proof of Theorem \ref{THM:MCDPDE_asymp}. 

\subsection{Proof of Proposition \ref{PROP:sandwich}}
\label{APP:Proof-Varest}

By Theorem \ref{THM:MCDPDE_asymp}, we have that $\widehat{\btheta}_{n,\beta}\xrightarrow{\mathcal{P}}{\btheta}_\beta^*$.

Let us first consider the matrix 
\begin{equation*}
\widehat{\boldsymbol J}_\beta = \frac1n \sum_{i=1}^n \nabla  \bm\psi_\beta (\bm X_i|\widehat{\btheta}_{n,\beta}).
\end{equation*}
By Assumption \ref{ASS:Asmp6}, the derivatives of the estimating function are dominated by an integrable envelope. 
Therefore, the uniform law of large numbers gives
\begin{equation*}
\sup_{\btheta\in\mathcal N} \left\| \frac1n \sum_{i=1}^n \nabla \bm\psi_\beta(\bm X_i|\btheta) -
\E_G[\nabla \bm\psi_\beta(\bm X|\btheta)] \right\| \xrightarrow{\mathcal{P}} \bm{0}.
\end{equation*}
Now, evaluating at the consistent estimator $\widehat{\btheta}_{n,\beta}$ and applying the continuous mapping theorem,
we get 
\begin{equation*}
\widehat{\boldsymbol J}_\beta \xrightarrow{\mathcal{P}} 
\E_G[ \nabla \bm\psi_\beta(\bm X|{\btheta}_\beta^*)] = \boldsymbol J_\beta({\btheta}_\beta^*).
\end{equation*}

Similarly, we note that 
$\widehat{\boldsymbol K}_\beta = \frac1n \sum_{i=1}^n \bm\psi_\beta(\bm X_i|\widehat{\btheta}_{n,\beta}) 
\bm\psi_\beta^\top (\bm X_i|\widehat{\btheta}_{n,\beta})$.
Then, the envelope condition in Assumption \ref{ASS:Asmp6} ensures that the corresponding class is Glivenko--Cantelli,
and hence 
\begin{equation*}
\widehat{\boldsymbol K}_\beta \xrightarrow{\mathcal{P}} 
\E_G[ \bm\psi_\beta(\bm X|{\btheta}_\beta^*) \bm\psi_\beta^\top(\bm X|{\btheta}_\beta^*)]
= \boldsymbol K_\beta({\btheta}_\beta^*).
\end{equation*}

Since $\boldsymbol J_\beta({\btheta}_\beta^*)$ is nonsingular by Assumption \ref{ASS:Asmp5}, 
the result follows from another application of the continuous mapping theorem.

\section{Performances under benchmark Gaussian datasets}
\label{APP:normal_data_examples}

We compare the proposed MCDPDE with the state-of-the-art cellMCD estimator using three benchmark datasets from 
the \texttt{R} package \texttt{robustbase} \citep{robustbase}. These datasets were also used by the developers of 
the \texttt{cellWise} package to illustrate the performance of the cellMCD estimator \\
\noindent (\texttt{cran.r-project.org/web/packages/cellWise/vignettes/cellMCD\_examples.html}). \\
All of the following three datasets are known to contain outlying observations.

\begin{itemize}
	\item \texttt{alcohol}: A dataset consisting of $d=7$ physicochemical characteristics measured on $n=44$ aliphatic alcohols, 
	collected to study alcohol solubility using molecular descriptors.
	
	\item \texttt{milk}: A dataset containing measurements on $d=8$ compositional variables for $n=86$ milk samples.
	
	\item \texttt{bushfire}: A dataset containing satellite measurements in $d=5$ spectral frequency bands for $n=38$ image pixels, 
 	recorded to identify bushfire scars.
\end{itemize}

To facilitate comparison with the published cellMCD analyses, each dataset is first transformed toward multivariate Gaussian
using the same \texttt{tranfo} function in the \texttt{cellWise} package. This procedure applies robustly estimated 
Yeo-Johnson transformations based on reweighted maximum likelihood estimation \citep{raymaekers2024transforming}. 
We then fit the cellMCD estimator, the Gaussian MLE/CMLE, and the proposed MCDPDE based on the pairwise Gaussian likelihood 
with tuning parameters $\beta\in\{0.1,0.3,0.5\}$.

\begin{figure}[!ht]
	\centering
		\includegraphics[width=0.49\linewidth]{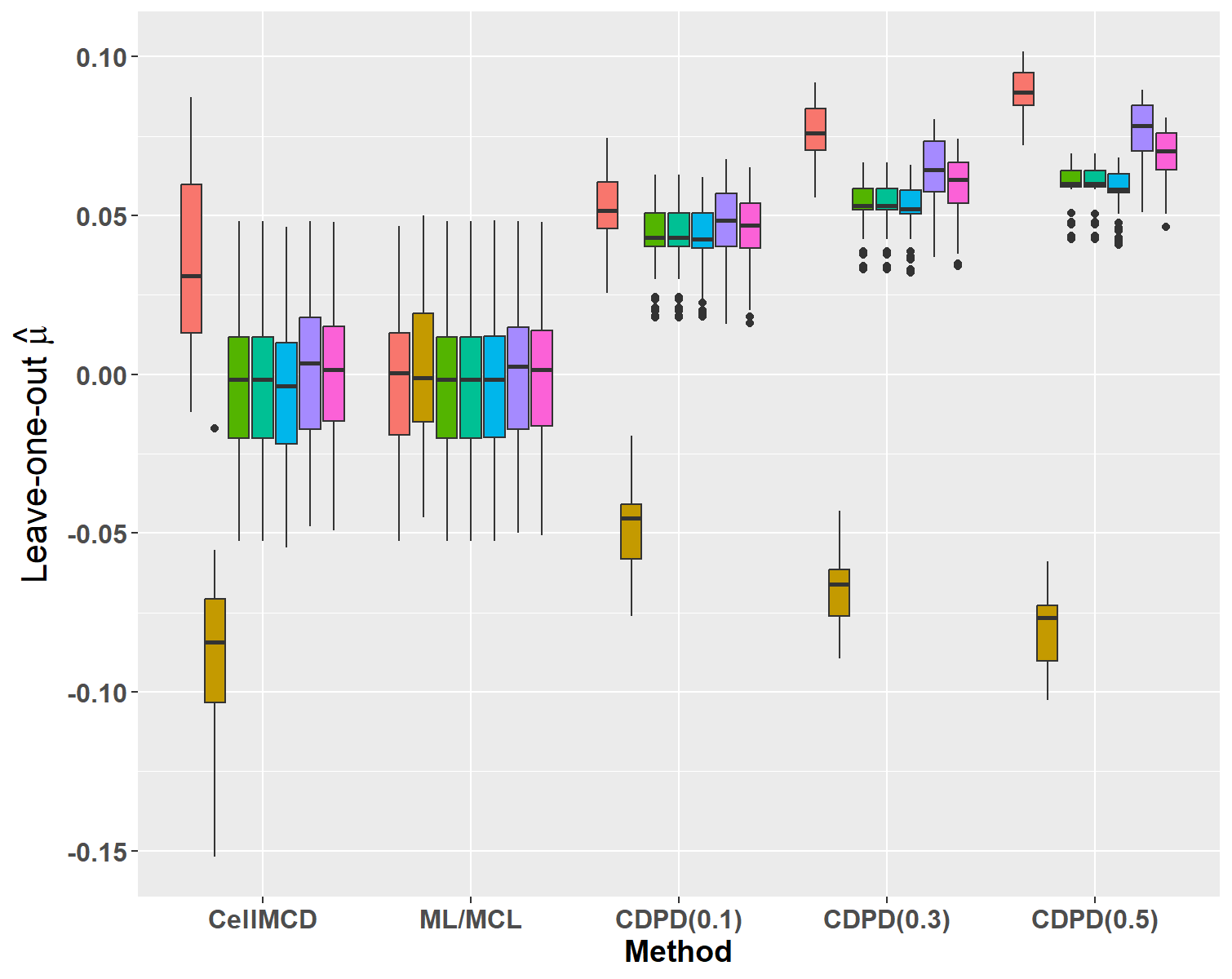}
		\includegraphics[width=0.49\linewidth]{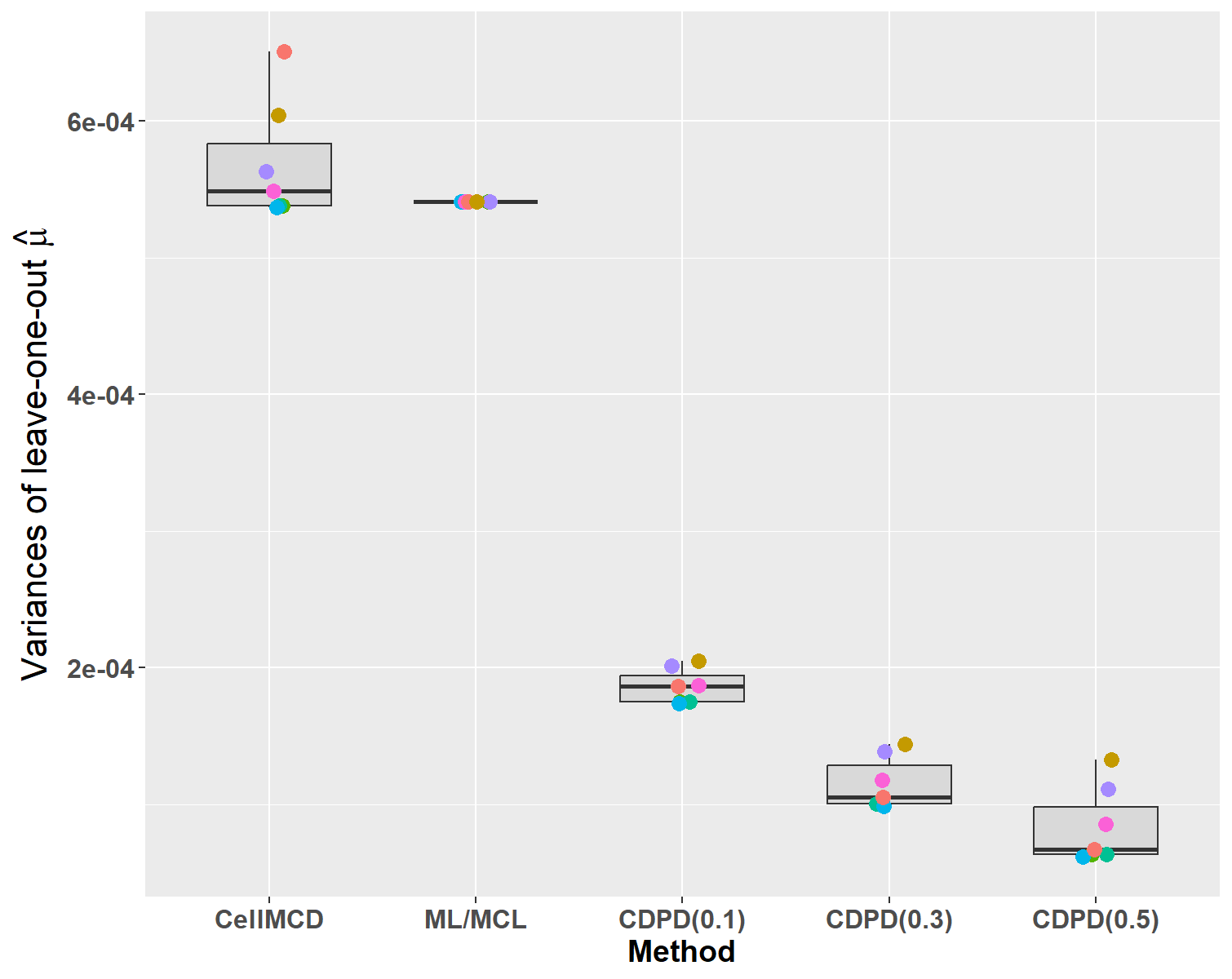}
		\includegraphics[width=0.49\linewidth]{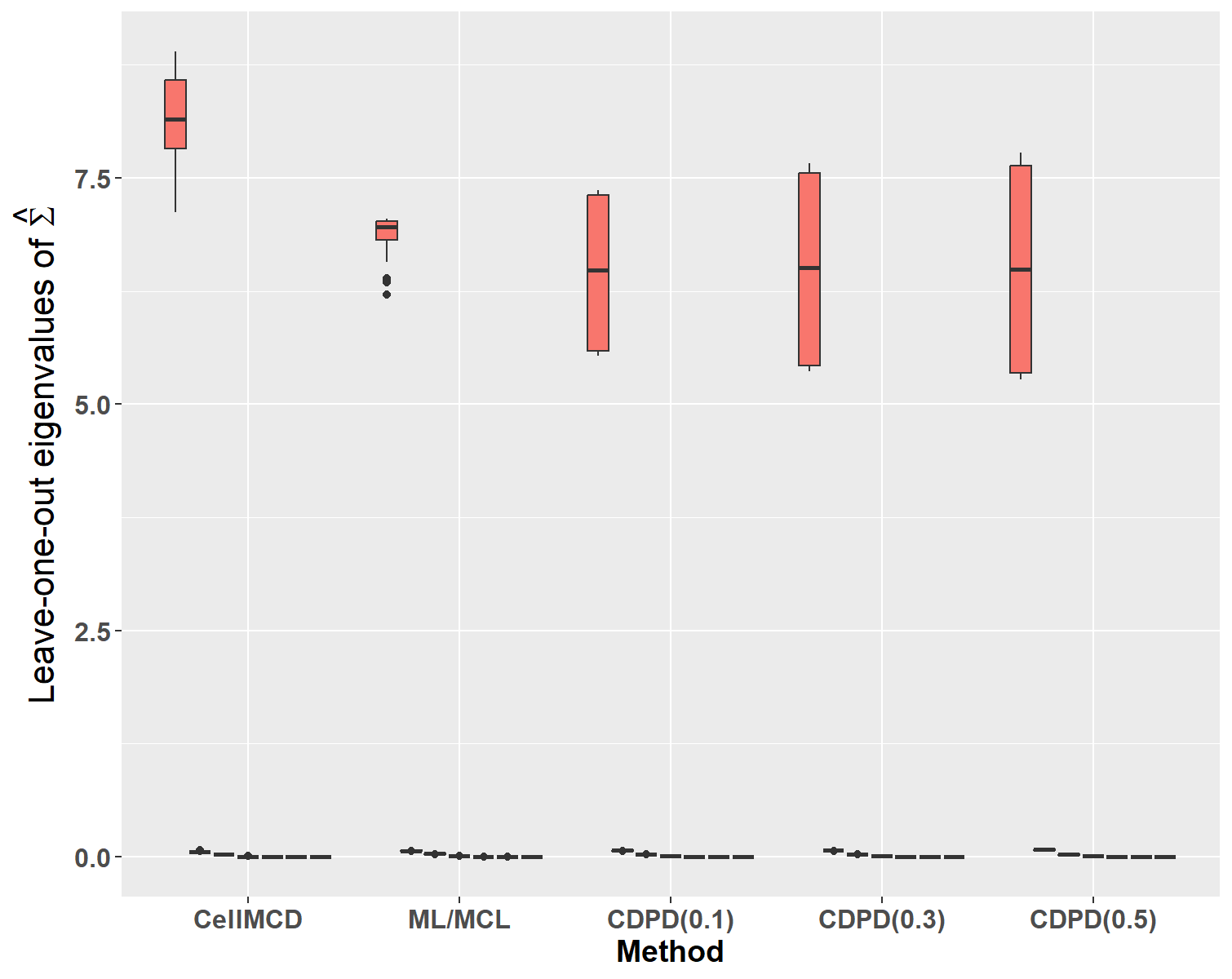}
		\includegraphics[width=0.49\linewidth]{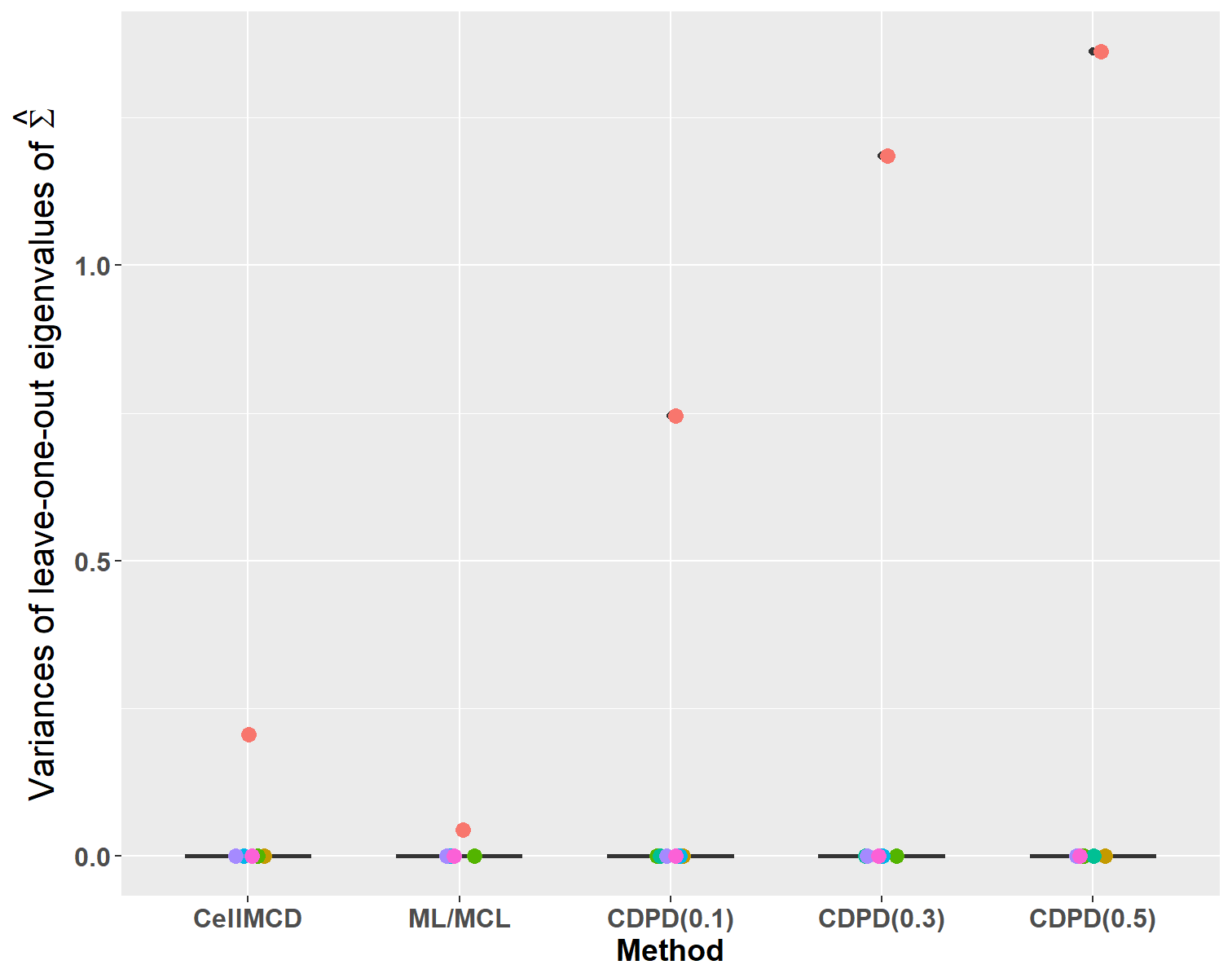}
	 \caption{Leave-one-out estimates of the location components and the eigenvalues of the scatter matrix for the \texttt{alcohol} dataset. Different colors correspond to different variables or eigenvalues.}
	\label{FIG:alcohol}
\end{figure}

\begin{figure}[!ht]
	\centering
	\includegraphics[width=0.49\linewidth]{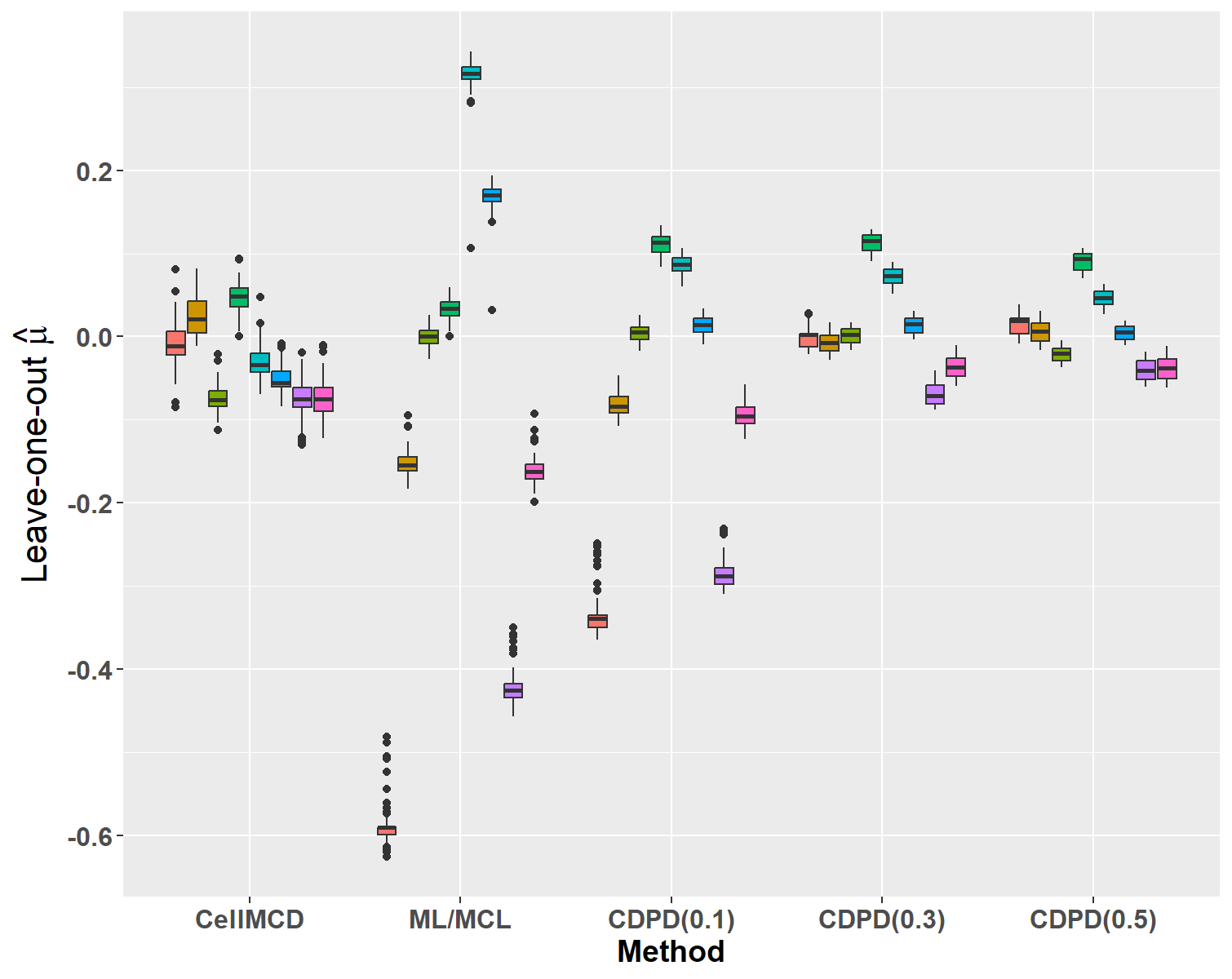}
	\includegraphics[width=0.49\linewidth]{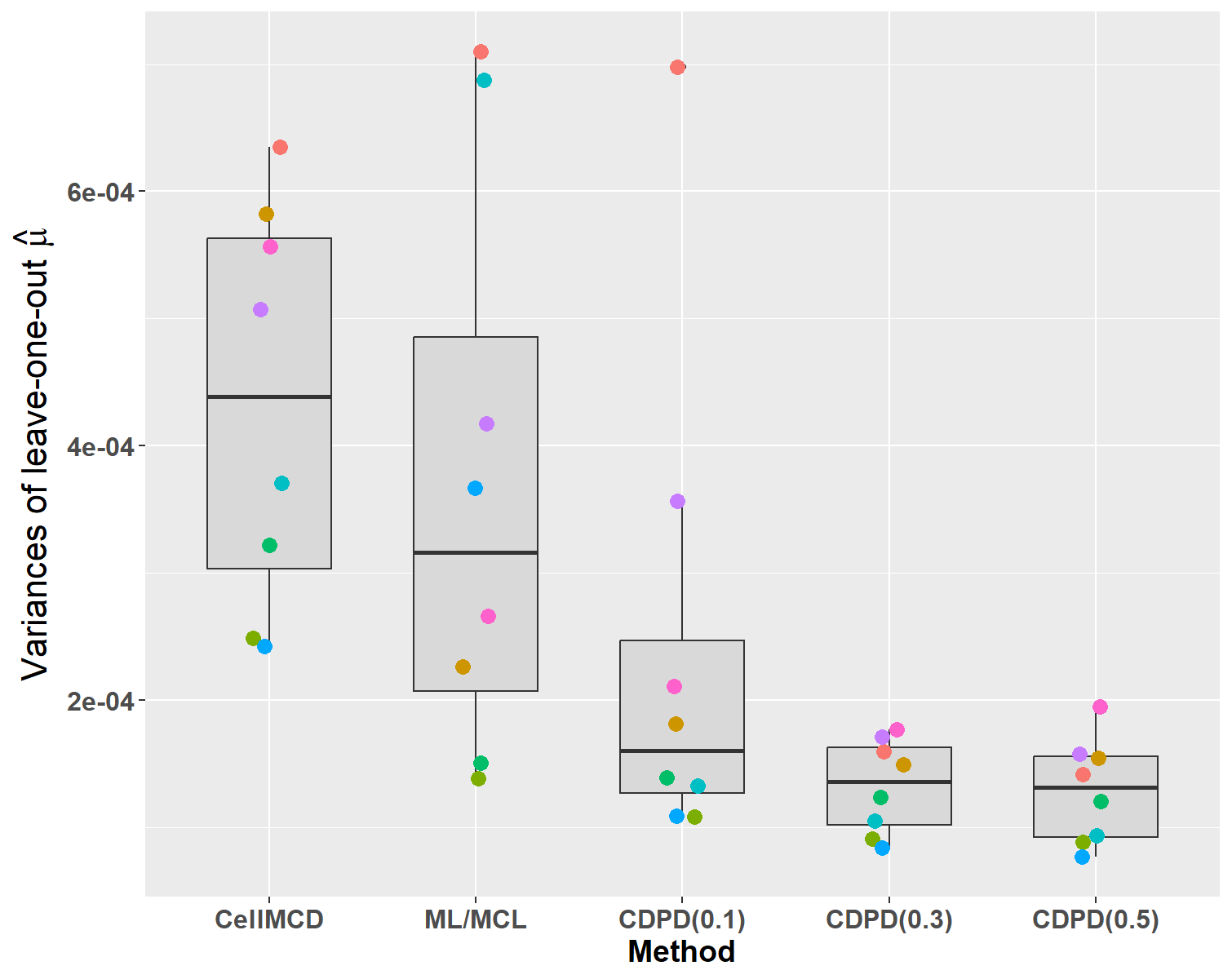}
	\includegraphics[width=0.49\linewidth]{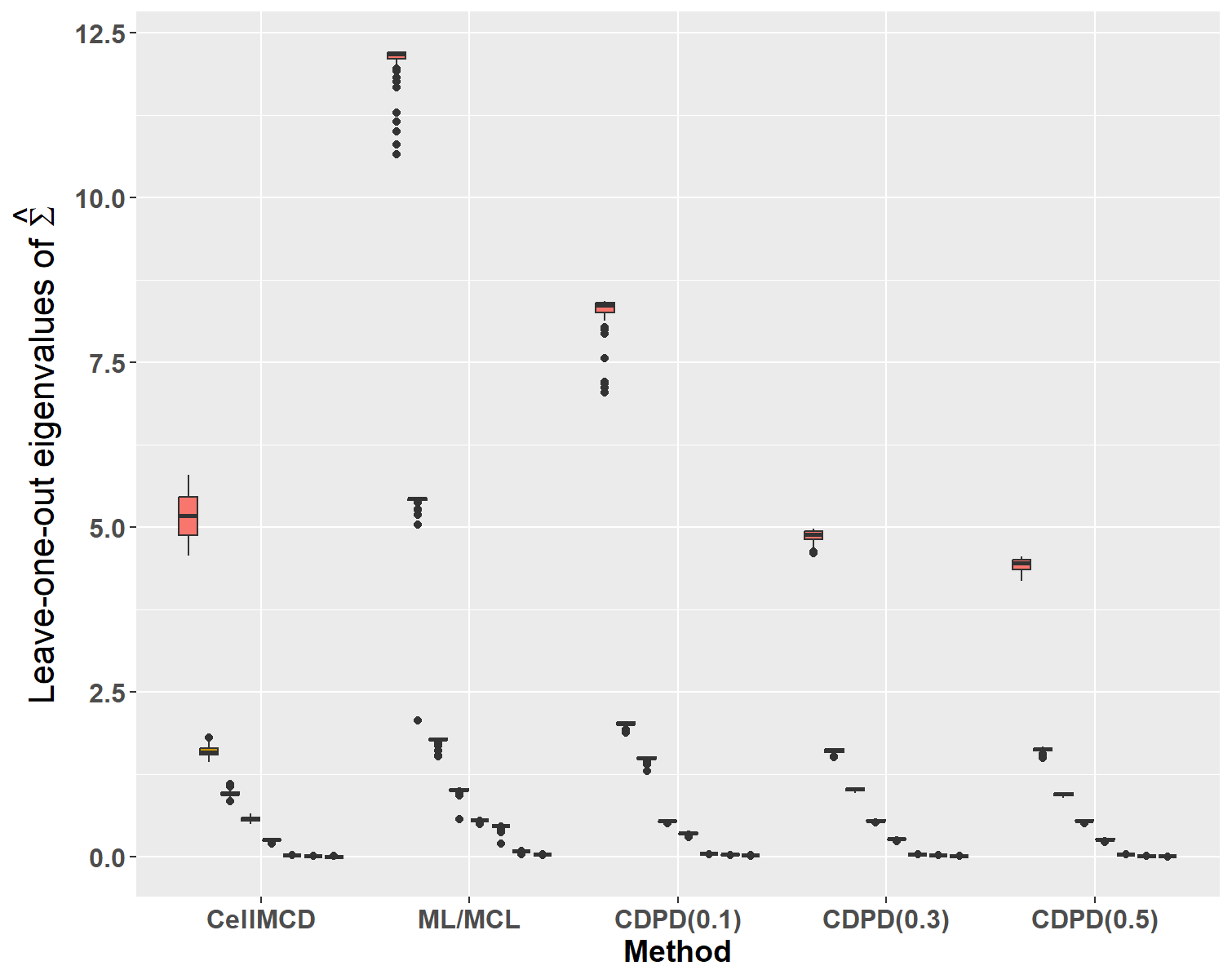}
	\includegraphics[width=0.49\linewidth]{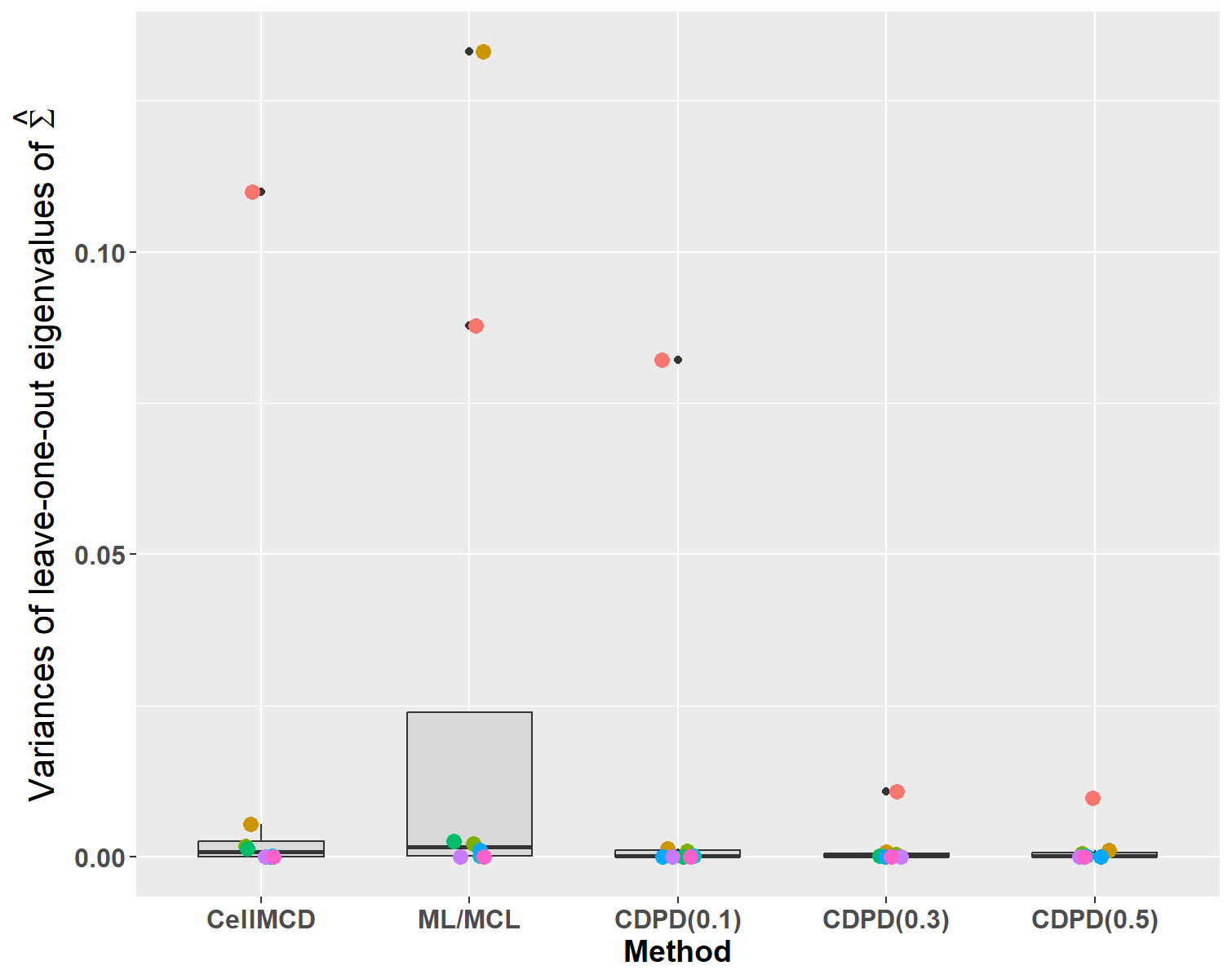}
	\caption{Leave-one-out estimates of the location components and the eigenvalues of the scatter matrix for the \texttt{milk} dataset. Different colors correspond to different variables or eigenvalues.}
	\label{FIG:milk}
\end{figure}

For each method, we examine the estimated location vector $\widehat{\bm{\mu}}$ 
and the eigenvalues of the estimated scatter matrix $\widehat{\bm{\Sigma}}$. To assess the stability of these estimators, 
we perform a leave-one-out (LOO) analysis by recomputing the estimates after deleting each observation in turn, 
yielding $n$ LOO estimates of $\widehat{\bm{\mu}}$ and $\widehat{\bm{\Sigma}}$ for each dataset. Figures~\ref{FIG:alcohol}--\ref{FIG:bushfire} display the resulting LOO estimates 
together with their empirical variances across the $n$ deletion samples.
Lower LOO variability provides empirical evidence of reduced local sensitivity, complementing the IF analysis in Section \ref{SEC:IF}.

\begin{figure}[!ht]
	\centering
	\includegraphics[width=0.49\linewidth]{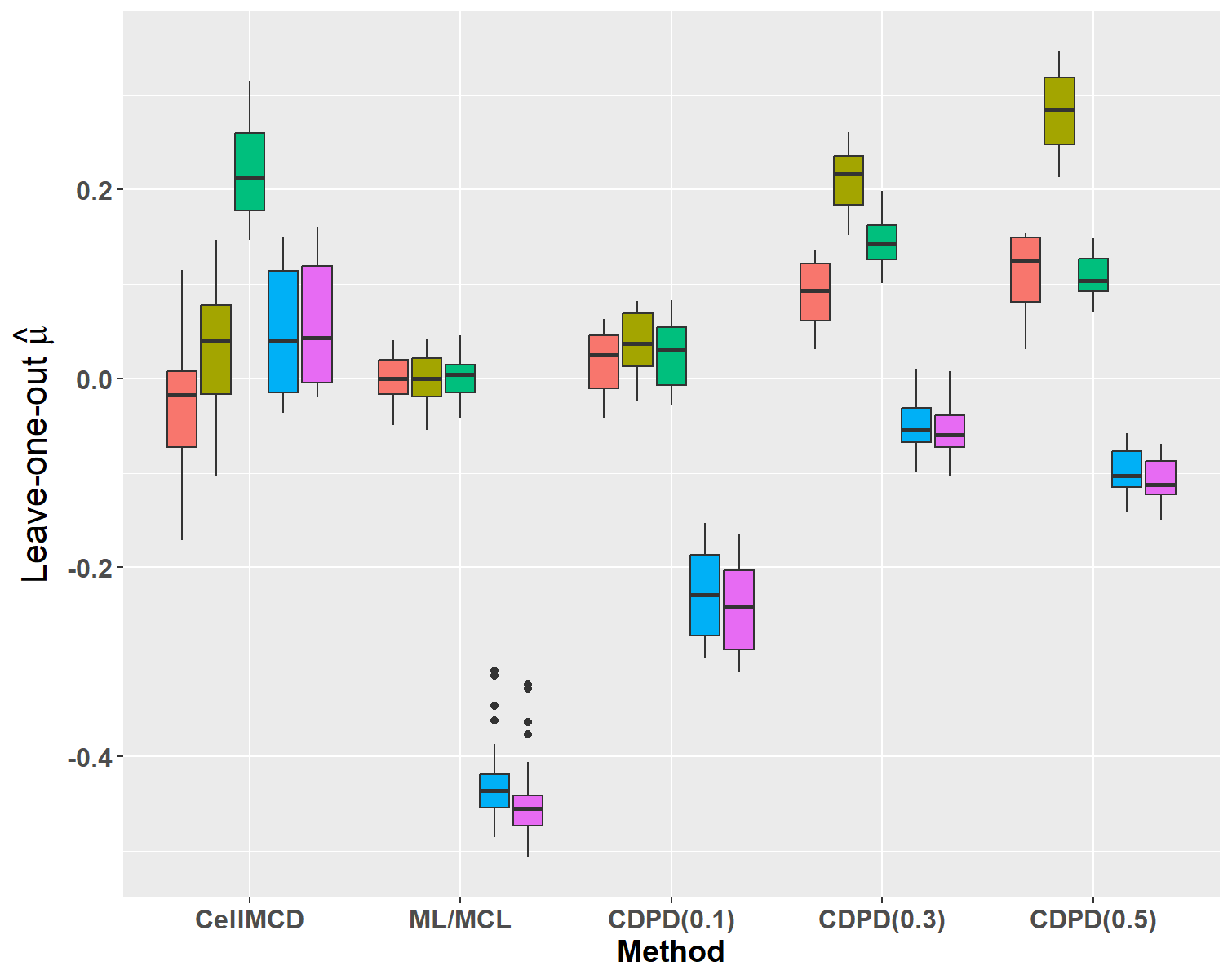}
	\includegraphics[width=0.49\linewidth]{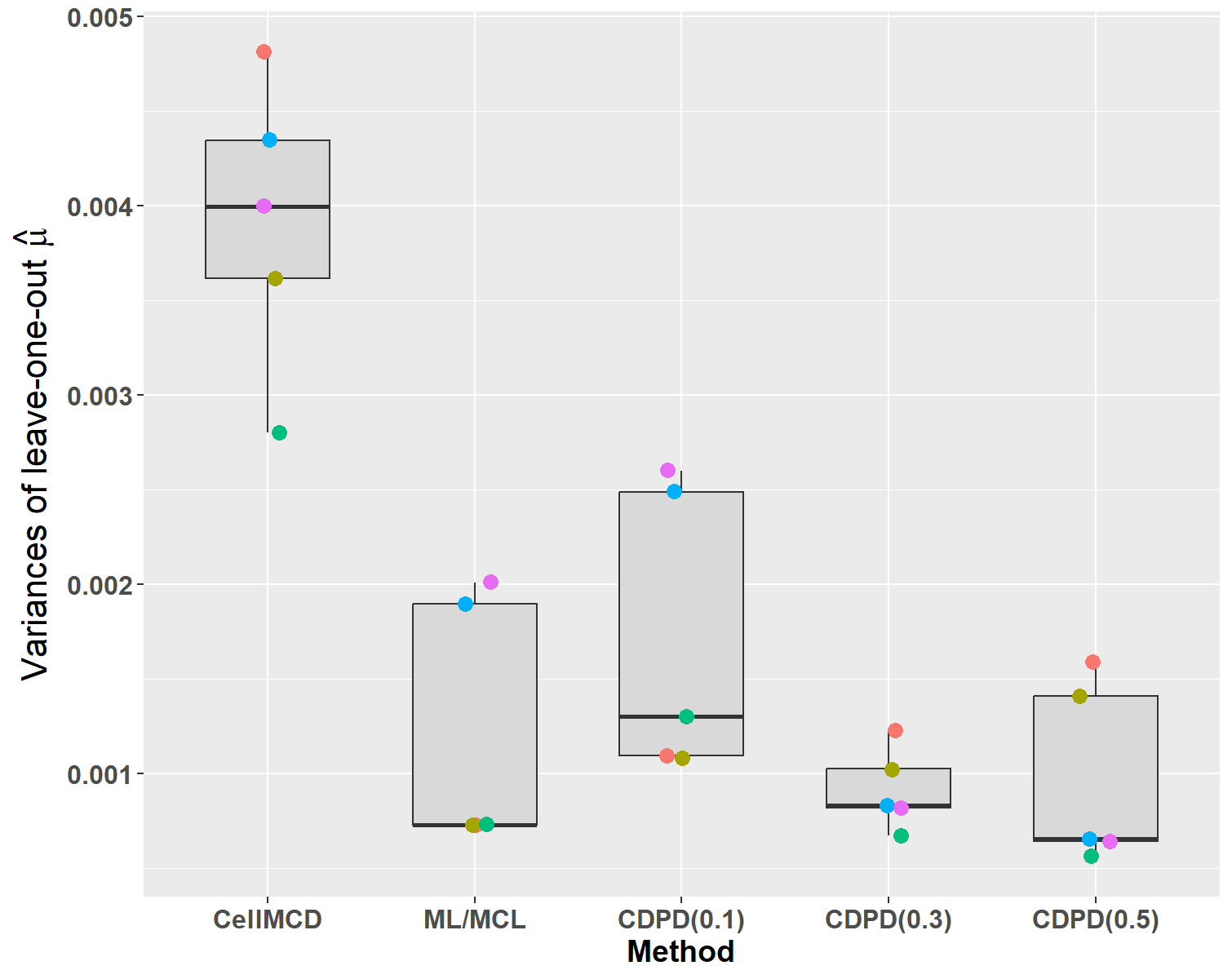}
	\includegraphics[width=0.49\linewidth]{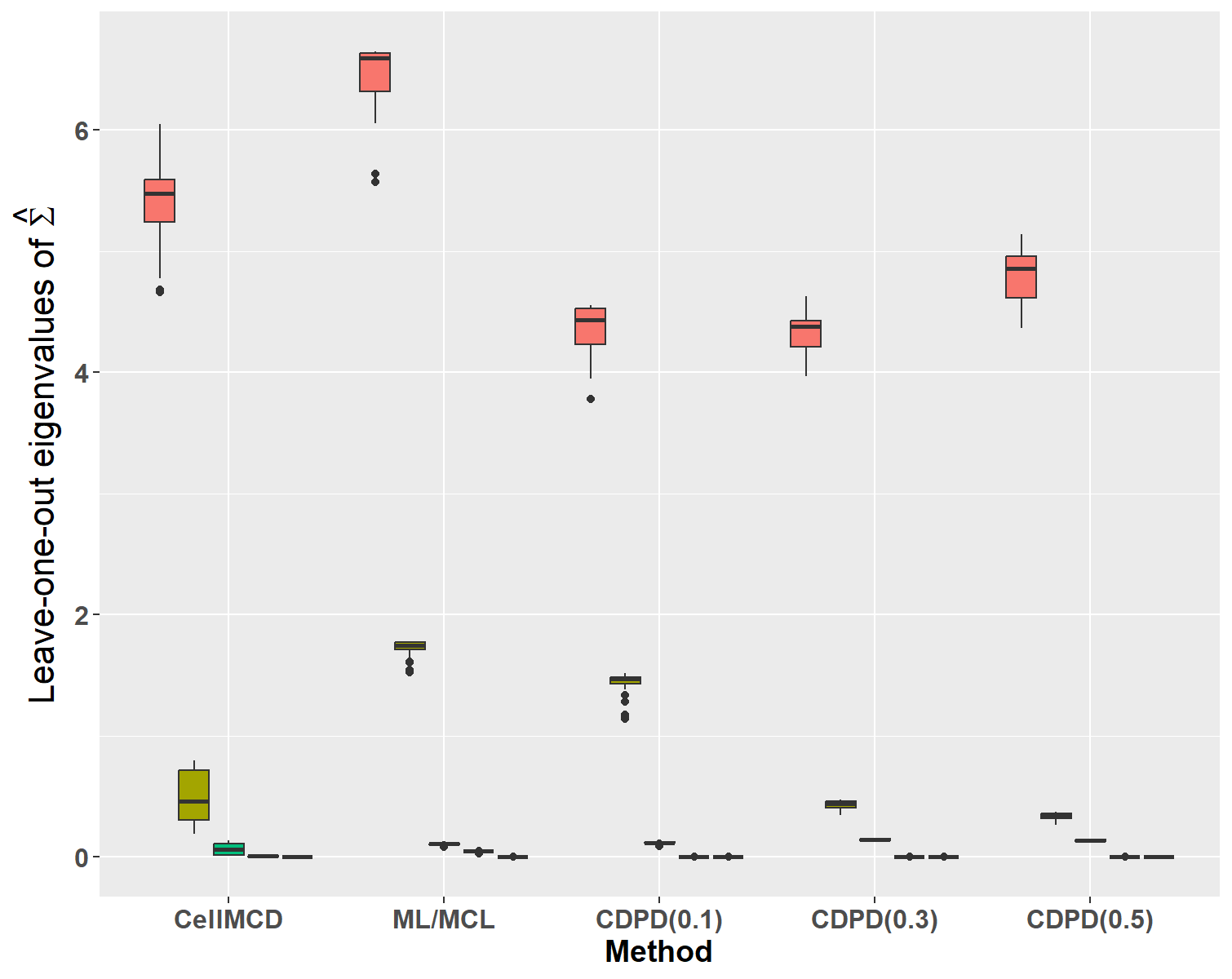}
	\includegraphics[width=0.49\linewidth]{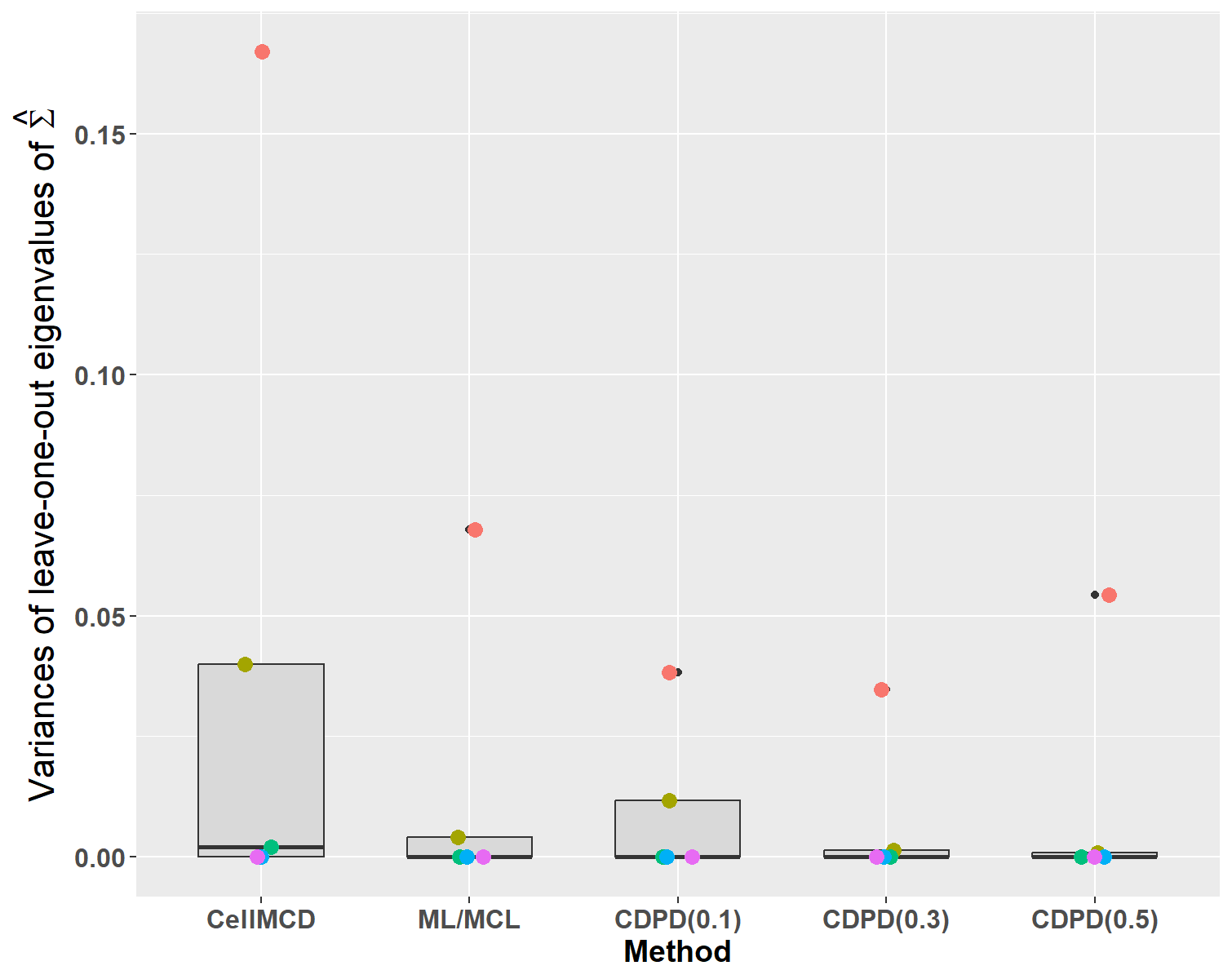}
	\caption{Leave-one-out estimates of the location components and the eigenvalues of the scatter matrix for the \texttt{bushfire} dataset. Different colors correspond to different variables or eigenvalues.}
	\label{FIG:bushfire}
\end{figure}

Figures~\ref{FIG:alcohol}--\ref{FIG:bushfire} show that the LOO variability of the estimated location vector $\bm{\mu}$ 
is substantially smaller for the proposed MCDPDE than for both the cellMCD estimator and the Gaussian MLE/CMLE across all three datasets. 
Moreover, the empirical variability decreases monotonically as the tuning parameter $\beta$ increases, 
consistent with the increasing robustness of the MCDPDE.
A similar pattern is observed for the estimated eigenvalues of the scatter matrix in the \texttt{milk} and \texttt{bushfire} datasets, 
where the proposed estimator exhibits noticeably greater stability than its competitors. 
For the \texttt{alcohol} dataset, the estimated scatter matrix is nearly singular, 
with one dominant eigenvalue and the remaining eigenvalues close to zero, a phenomenon previously reported for the cellMCD estimator. 
The proposed MCDPDE exhibits the same qualitative behavior while producing slightly larger variability for the leading eigenvalue. 
Nevertheless, the variability remains comparable to that of the cellMCD estimator and  the Gaussian MLE.

We additionally consider the \texttt{toxicity} dataset from the \texttt{robustbase} package, 
which contains $n=38$ observations on $d=10$ molecular descriptors used to predict the toxicity of carboxylic acids. 
This dataset is particularly challenging because, even after transformation toward central normality using 
the robust preprocessing described above, the cellMCD algorithm fails to produce robust location and scatter estimates 
due to the presence of a large number of marginal outliers. In contrast, the proposed MCDPDE can still be computed without difficulty, 
both with and without the robust transformation, yielding the location and scatter estimates reported in Tables 
\ref{TAB:toxicity-est} and \ref{TAB:toxicity-estWO}, respectively. 
This example illustrates that, beyond its improved robustness and stability, the proposed MCDPDE possesses a broader computational domain 
of applicability, producing reliable estimates even for datasets where the cellMCD algorithm is unable to return a valid solution.

\begin{table}[!ht]
	\centering
	\caption{MCDPDs obtained under the Gaussian model for \texttt{toxicity} data for different $\beta$ 
		(The case $\beta=0$ gives the MLE/MCLE) after robust transformation.}
	\begin{tabular}{lrrrr|lrrrr}
		\hline
		$\beta\rightarrow$ & 0 & 0.1 & 0.3 & 0.5 & 	$\beta\rightarrow$ & 0 & 0.1 & 0.3 & 0.5\\ 
		\hline
		\multicolumn{5}{c|}{\textbf{Location components}} & \multicolumn{5}{c}{\textbf{Variance components}}\\\hline
$\mu_1$    & $-$0.08 & $-$0.08 	& $-$0.09 	& $-$0.07 	& $\sigma_{1}^2$  & 0.89 & 1.19 & 1.20 & 1.17\\
$\mu_2$    & $-$0.00 & 0.02 	& 0.06 		& 0.10 		&  $\sigma_{2}^2$ & 0.74 & 1.01 & 1.02 & 1.01\\
$\mu_3$    & 0.00 	 & 0.02 	& 0.07 		& 0.12 		& $\sigma_{3}^2$  & 0.74 & 1.04 & 1.15 & 1.22\\
$\mu_4$    & $-$0.00 & 0.02 	& 0.08 		& 0.13 		& $\sigma_{4}^2$  & 0.74 & 1.04 & 1.10 & 1.15\\
$\mu_5$    & $-$0.07 & $-$0.06 	& $-$0.04 	& $-$0.02	&  $\sigma_{5}^2$ & 0.86 & 1.13 & 1.09 & 1.01\\
$\mu_6$    & $-$0.00 & $-$0.00 	& $-$0.01 	& $-$0.01	& $\sigma_{6}^2$  & 0.74 & 1.03 & 1.11 & 1.17\\ 
$\mu_7$    & $-$0.00 & 0.01 	& 0.02 		& 0.02 		& $\sigma_{7}^2$  & 0.74 & 1.00 & 1.02 & 1.02\\
$\mu_8$    & 0.00 	 & 0.01 	& 0.01 		& 0.01 		& $\sigma_{8}^2$  & 0.74 & 1.03 & 1.11 & 1.19\\
$\mu_9$    & $-$0.00 & $-$0.01 	& $-$0.03 	& $-$0.05 	& $\sigma_{9}^2$  & 0.74 & 1.02 & 1.09 & 1.15\\
$\mu_{10}$ & 0.00 	 & 0.02 	& 0.06 		& 0.11 		& $\sigma_{10}^2$ & 0.74 & 1.06 & 1.21 & 1.36\\
		\hline
		\multicolumn{10}{c}{\textbf{Correlation parameters}}\\
		\hline
		$\rho_{12}$   &  0.83 &  0.86 &  0.92 &  0.93 & 		$\rho_{13}$   &  0.11 &  0.12 &  0.14 &  0.16\\
		$\rho_{14}$   &  0.28 &  0.28 &  0.28 &  0.26 & 		$\rho_{15}$   &  0.30 &  0.32 &  0.39 &  0.44\\
		$\rho_{16}$   &  0.49 &  0.50 &  0.51 &  0.51 & 		$\rho_{17}$   &  0.72 &  0.74 &  0.77 &  0.80\\
		$\rho_{18}$   &  0.48 &  0.50 &  0.53 &  0.54 & 		$\rho_{19}$   &  0.76 &  0.76 &  0.74 &  0.72\\
		$\rho_{1,10}$ &  0.61 &  0.60 &  0.58 &  0.54 & 		$\rho_{23}$   &  0.59 &  0.58 &  0.54 &  0.51\\
		$\rho_{24}$   &  0.73 &  0.73 &  0.73 &  0.73 & 		$\rho_{25}$   &  0.86 &  0.86 &  0.84 &  0.82\\
		$\rho_{26}$   &  0.08 &  0.08 &  0.06 &  0.04 & 		$\rho_{27}$   &  0.23 &  0.22 &  0.21 &  0.21\\
		$\rho_{28}$   &  0.62 &  0.63 &  0.65 &  0.67 & 		$\rho_{29}$   &  0.56 &  0.59 &  0.63 &  0.68\\
		$\rho_{2,10}$ &  0.83 &  0.83 &  0.83 &  0.86 & 		$\rho_{34}$   &  0.17 &  0.17 &  0.14 &  0.11\\
		$\rho_{35}$   &  0.25 &  0.24 &  0.22 &  0.20 & 		$\rho_{36}$   &  0.82 &  0.82 &  0.82 &  0.81\\
		$\rho_{37}$   &  0.83 &  0.85 &  0.86 &  0.86 & 		$\rho_{38}$   & $-$0.13 & $-$0.12 & $-$0.10 & $-$0.08\\
		$\rho_{39}$   & $-$0.04 & $-$0.03 & $-$0.01 &  0.03 & 	$\rho_{3,10}$ & $-$0.60 & $-$0.62 & $-$0.64 & $-$0.66\\
		$\rho_{45}$   & $-$0.58 & $-$0.62 & $-$0.68 & $-$0.73 &	$\rho_{46}$   & $-$0.01 & $-$0.01 & $-$0.02 & $-$0.03\\
		$\rho_{47}$   &  0.19 &  0.21 &  0.24 &  0.27 & 		$\rho_{48}$   &  0.42 &  0.43 &  0.43 &  0.43\\
		$\rho_{49}$   & $-$0.32 & $-$0.30 & $-$0.28 & $-$0.25 &	$\rho_{4,10}$ & $-$0.29 & $-$0.28 & $-$0.23 & $-$0.17\\
		$\rho_{56}$   & $-$0.63 & $-$0.63 & $-$0.64 & $-$0.65 &	$\rho_{57}$   & $-$0.65 & $-$0.68 & $-$0.76 & $-$0.81\\
	    $\rho_{58}$    & $-$0.29 & $-$0.27 & $-$0.23 & $-$0.18 & $\rho_{59}$    &  0.06 &  0.08 &  0.12 &  0.19\\
		$\rho_{5,10}$  &  0.14 &  0.17 &  0.21 &  0.24 &		$\rho_{67}$    &  0.89 &  0.90 &  0.90 &  0.90\\
		$\rho_{68}$    &  0.23 &  0.22 &  0.22 &  0.21 &		$\rho_{69}$    &  0.36 &  0.36 &  0.34 &  0.32\\
		$\rho_{6,10}$  &  0.86 &  0.87 &  0.88 &  0.88 &		$\rho_{78}$    &  0.80 &  0.81 &  0.82 &  0.84\\
		$\rho_{79}$    &  0.50 &  0.50 &  0.50 &  0.50 &		$\rho_{7,10}$  &  0.07 &  0.06 &  0.05 &  0.03\\
		$\rho_{89}$    &  0.10 &  0.10 &  0.10 &  0.09 &		$\rho_{8,10}$  & $-$0.68 & $-$0.71 & $-$0.73 & $-$0.75\\
		$\rho_{9,10}$  & $-$0.68 & $-$0.69 & $-$0.69 & $-$0.70\\
		\hline
	\end{tabular}
	\label{TAB:toxicity-est}
\end{table}

\begin{table}[!ht]
	\centering
	\caption{MCDPDs obtained under the Gaussian model for \texttt{toxicity} data for different $\beta$ 
		(The case $\beta=0$ gives the MLE/MCLE) without transformation.}
\begin{tabular}{lrrrr|lrrrr}
	\hline
	$\beta\rightarrow$ & 0 & 0.1 & 0.3 & 0.5 & 	$\beta\rightarrow$ & 0 & 0.1 & 0.3 & 0.5\\ 
	\hline
	\multicolumn{5}{c|}{\textbf{Location components}} & \multicolumn{5}{c}{\textbf{Variance components}}\\\hline
	$\mu_1$   & $-$0.16 & $-$0.18 & $-$0.21 & $-$0.24 & 	$\sigma_{1}^2$      &   0.12 &   0.16 &   0.16 &   0.17\\
	$\mu_2$   &  1.67 &  1.63 &  1.66 &  1.61 & 	$\sigma_{2}^2$      &   1.34 &   1.85 &   1.98 &   2.20\\
	$\mu_3$   &  0.65 &  0.64 &  0.90 &  0.90 &		$\sigma_{3}^2$      &   0.15 &   0.21 &   0.01 &   0.01\\
	$\mu_4$   &  4.34 &  4.39 &  4.65 &  4.63 &		$\sigma_{4}^2$      &   0.53 &   0.61 &   0.06 &   0.07\\
	$\mu_5$   & 17.19 & 17.18 & 17.23 & 17.22 &		$\sigma_{5}^2$      &   0.34 &   0.45 &   0.42 &   0.44\\
	$\mu_6$   &  3.12 &  2.98 &  2.84 &  2.67 &		$\sigma_{6}^2$      &   6.97 &   9.23 &  9.35 &  9.59\\
	$\mu_7$   & 34.27 & 33.85 & 33.46 & 33.11 &		$\sigma_{7}^2$      & 114.75 & 151.44 & 148.81 & 155.45\\
	$\mu_8$   &  1.45 &  1.45 &  1.45 &  1.45 &		$\sigma_{8}^2$      &   0.00 &   0.00 &   0.00 &   0.00\\
	$\mu_9$   & 38.14 & 38.25 & 32.99 & 33.25 &		$\sigma_{9}^2$      &  69.62 &  87.83 &   2.08 &   2.52\\
	$\mu_{10}$&  6.81 &  1.46 &  1.46 &  1.46 &		$\sigma_{10}^2$   &  31.97 &   0.00 &   0.00 &   0.00\\
	\hline
	\multicolumn{10}{c}{\textbf{Correlation parameters}}\\
	 \hline
	 $\rho_{12}$  &  0.82 &  0.84 &  0.92 &  0.94 &  $\rho_{13}$  & $-$0.10 & $-$0.10 &  0.51 &  0.52\\
	 $\rho_{14}$  &  0.15 &  0.16 &  0.42 &  0.42 &	 $\rho_{15}$  & $-$0.00 &  0.06 &  0.57 &  0.61\\
	 $\rho_{16}$  &  0.37 &  0.36 &  0.48 &  0.49 &	 $\rho_{17}$  &  0.57 &  0.59 &  0.86 &  0.86\\
	 $\rho_{18}$  &  0.43 &  0.45 &  0.52 &  0.60 &	 $\rho_{19}$  &  0.75 &  0.75 &  0.73 &  0.73\\
	 $\rho_{1,10}$&  0.56 &  0.57 &  0.38 &  0.40 &	 $\rho_{23}$  &  0.64 &  0.62 &  0.32 &  0.31\\
	 $\rho_{24}$  &  0.72 &  0.74 &  0.74 &  0.75 &	 $\rho_{25}$  &  0.87 &  0.86 &  0.85 &  0.83\\
	 $\rho_{26}$  &  0.11 &  0.09 &  0.06 &  0.02 &	 $\rho_{27}$  &  0.14 &  0.13 &  0.09 &  0.09\\
	 $\rho_{28}$  &  0.59 &  0.59 &  0.59 &  0.60 &	 $\rho_{29}$  &  0.57 &  0.60 &  0.64 &  0.69\\
	 $\rho_{2,10}$&  0.83 &  0.83 &  0.83 &  0.83 &	 $\rho_{34}$  &  0.21 &  0.20 &  0.00 &  $-$0.00\\
	 $\rho_{35}$  &  0.25 &  0.22 & $-$0.03 & $-$0.01 &	 $\rho_{36}$  &  0.76 &  0.77 &  0.77 &  0.78\\
	 $\rho_{37}$  &  0.89 &  0.89 &  0.89 &  0.89 &	 $\rho_{38}$  & $-$0.04 & $-$0.07 & $-$0.13 & $-$0.23\\
	 $\rho_{39}$  & $-$0.00 & $-$0.00 &  0.01 & $-$0.05 &	 $\rho_{3,10}$  & $-$0.37 & $-$0.37 & $-$0.66 & $-$0.62\\
	 $\rho_{45}$    & $-$0.34 & $-$0.43 & $-$0.76 & $-$0.79 &	 $\rho_{46}$    & $-$0.01 & $-$0.01 & $-$0.02 & $-$0.06\\
	 $\rho_{47}$    &  0.25 &  0.21 &  0.21 &  0.12 & 	 $\rho_{48}$    &  0.43 &  0.40 &  0.38 &  0.31\\
	 $\rho_{49}$    & $-$0.46 & $-$0.45 &  0.13 &  0.01 	&  $\rho_{4,10}$  & $-$0.50 & $-$0.48 &  0.12 &  0.04\\
	 $\rho_{56}$    & $-$0.20 & $-$0.21 & $-$0.70 & $-$0.72 & 	 $\rho_{57}$    & $-$0.50 & $-$0.57 & $-$0.47 & $-$0.54\\
	 $\rho_{58}$    & $-$0.39 & $-$0.34 & $-$0.27 & $-$0.29 &	 $\rho_{59}$    & $-$0.21 & $-$0.16 &  0.38 &  0.30\\ 
	 $\rho_{5,10}$  & $-$0.06 & $-$0.00 &  0.25 &  0.20 	&	 $\rho_{67}$    &  0.74 &  0.74 &  0.64 &  0.64\\
	 $\rho_{68}$    &  0.40 & $-$0.24 & $-$0.23 & $-$0.21 	&	 $\rho_{69}$    &  0.51 &  0.03 &  0.00 &  0.02\\
	 $\rho_{6,10}$  &  0.60 &  0.15 & $-$0.28 & $-$0.24 	&	 $\rho_{78}$    &  0.49 & $-$0.09 & $-$0.61 & $-$0.59\\
	 $\rho_{79}$    &  0.56 &  0.33 &  0.29 &  0.30 		&	 $\rho_{7,10}$  &  0.28 &  0.12 &  0.07 &  0.06\\
	 $\rho_{89}$    &  0.19 &  0.45 &  0.40 &  0.37 		&	$\rho_{8,10}$  & $-$0.56 &  0.83 &  0.78 &  0.80\\
	 $\rho_{9,10}$  & $-$0.51 &  0.62 &  0.03 &  0.08\\
	\hline
\end{tabular}
\label{TAB:toxicity-estWO}
\end{table}

\section{Additional Figures and Tables}
\label{APP:empResults}

\begin{figure}[!ht]
	\centering
	\subfloat[Casewise contamination]{
		\includegraphics[width=0.49\linewidth]{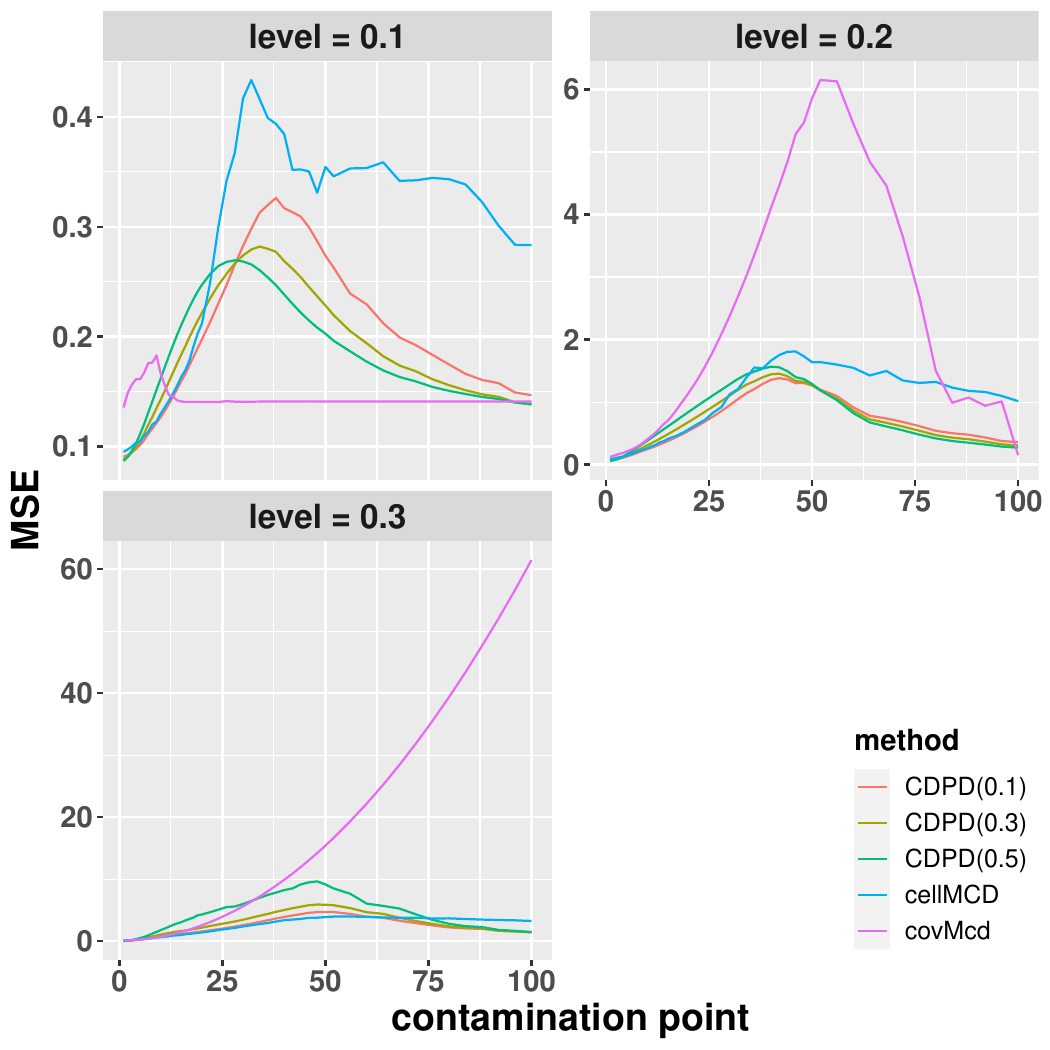}
		\includegraphics[width=0.49\linewidth]{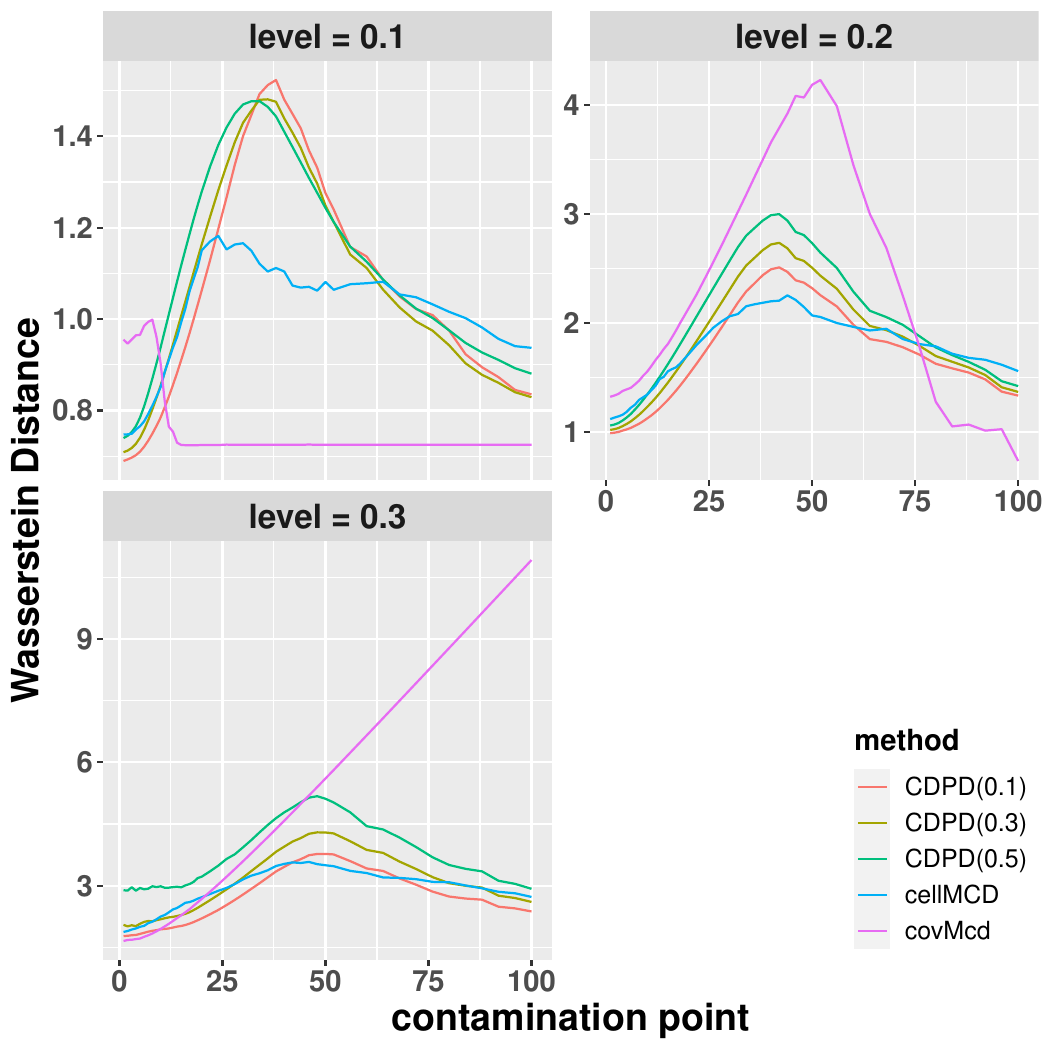}
		\label{FIG:normal-case-10}
	}
\\
	\subfloat[Cellwise contamination]{
		\includegraphics[width=0.49\linewidth]{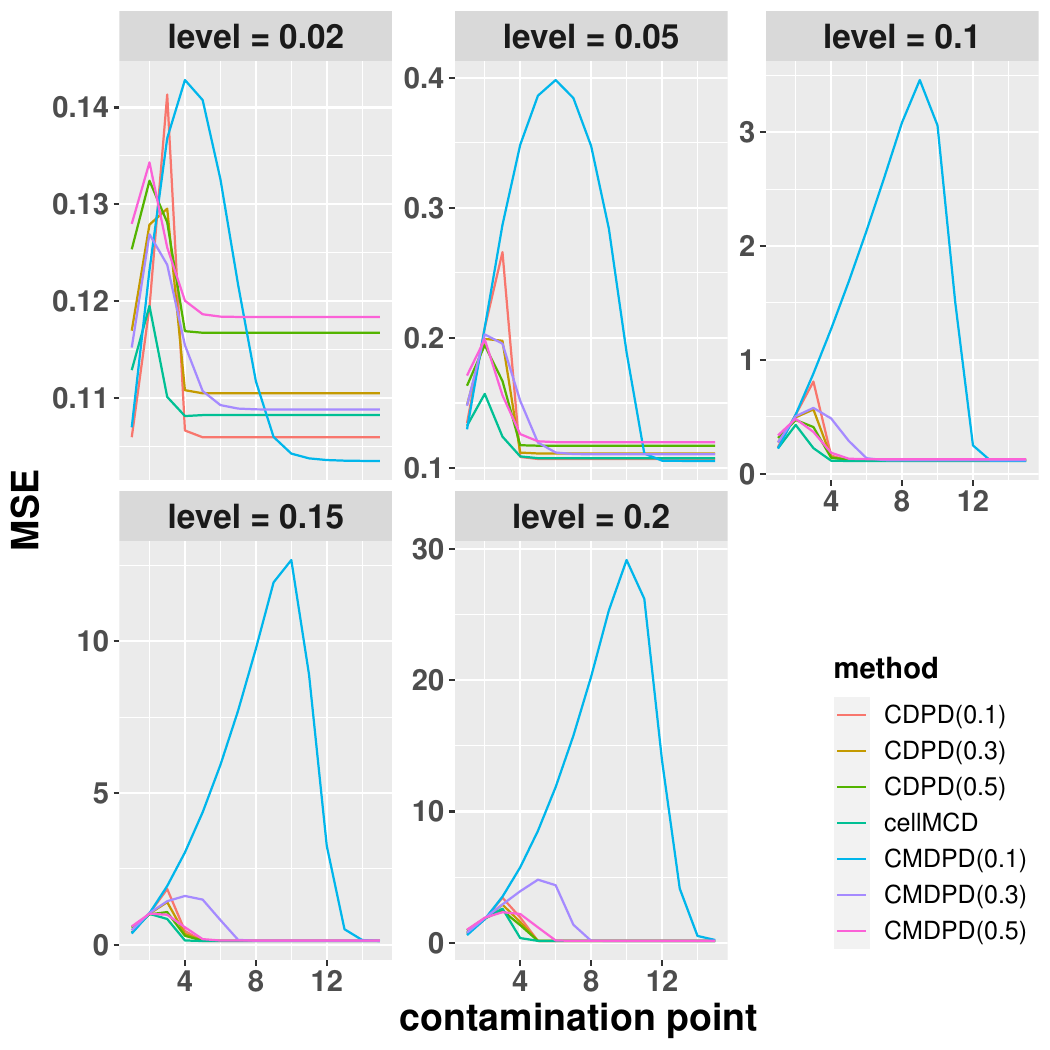}
		\includegraphics[width=0.49\linewidth]{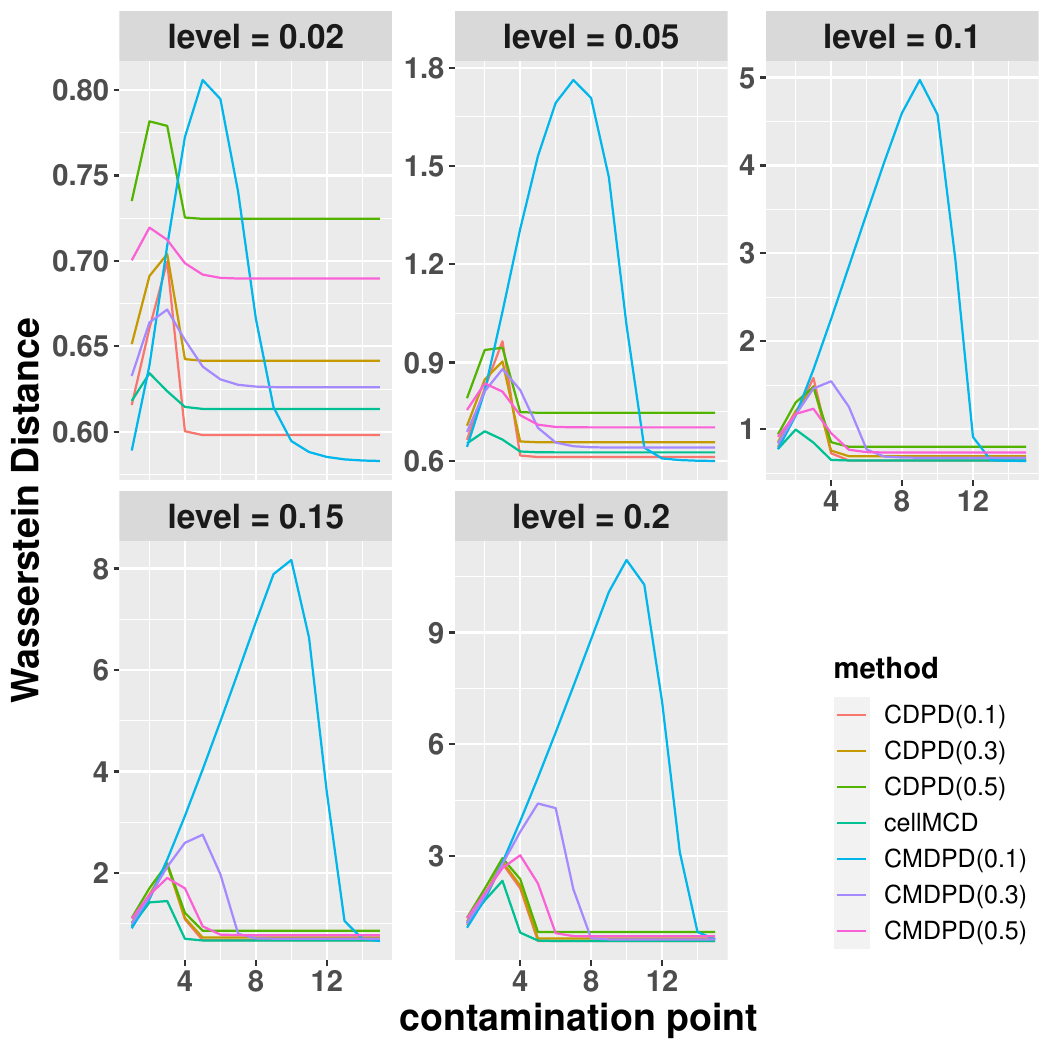}
		\label{FIG:normal-cell-10}
	}
	\caption{Plots of empirical MSE and Wasserstein-2 distance of different competing estimators over the contamination point $y$
		for data dimension $d=10$.}
	\label{FIG:normal_cont10}
\end{figure}

\begin{figure}
	\centering
	\subfloat[Casewise contamination]{
		\includegraphics[width=0.49\linewidth]{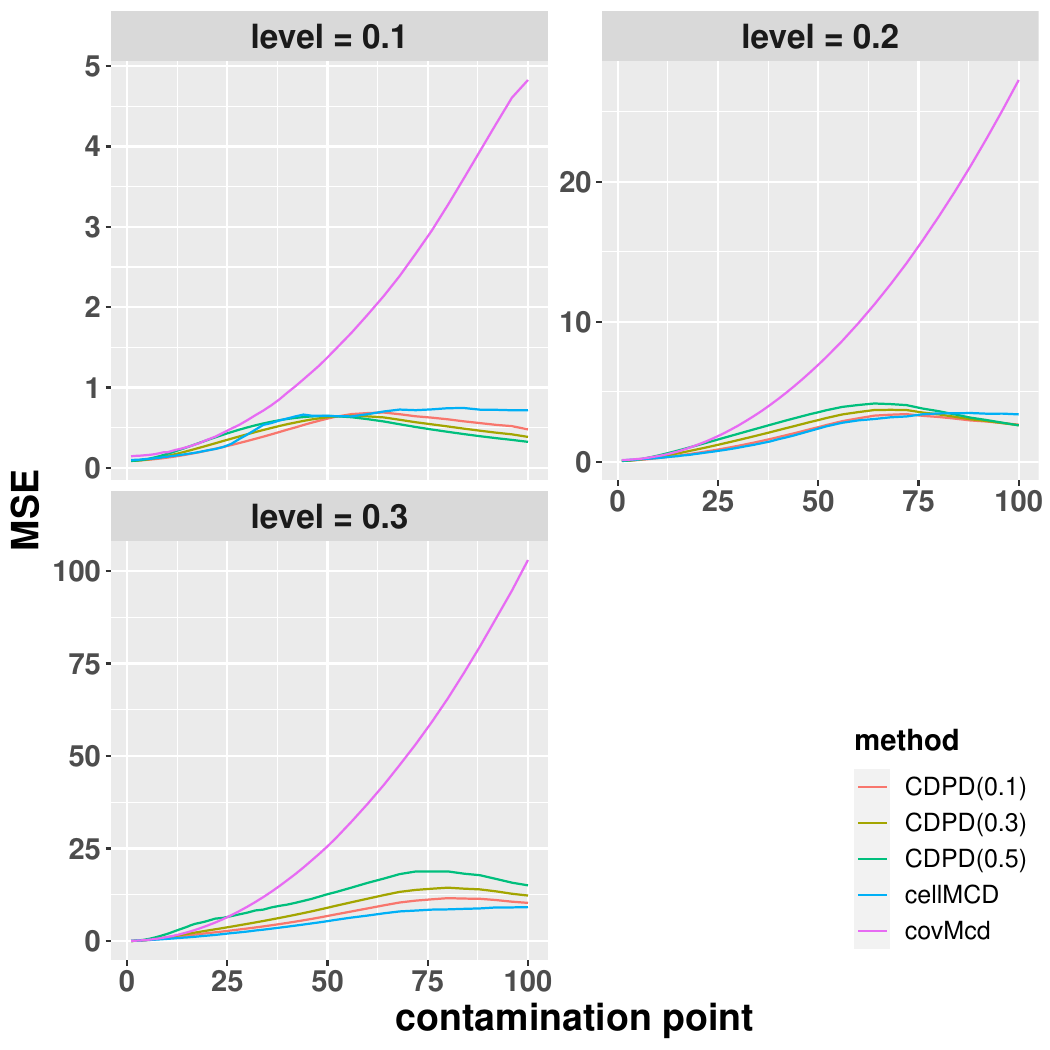}
		\includegraphics[width=0.49\linewidth]{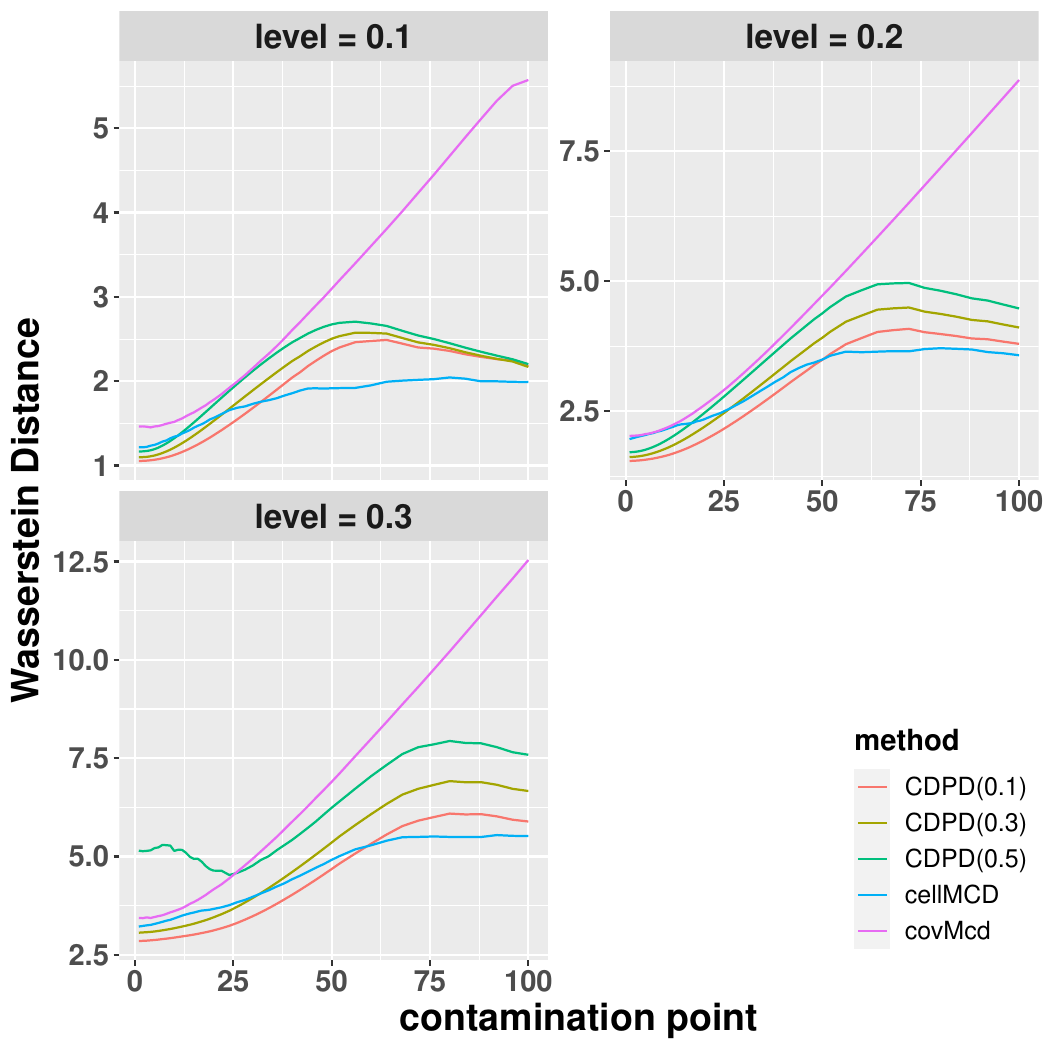}
		\label{FIG:normal-case-30}
	}
\\
	\subfloat[Cellwise contamination]{
		\includegraphics[width=0.49\linewidth]{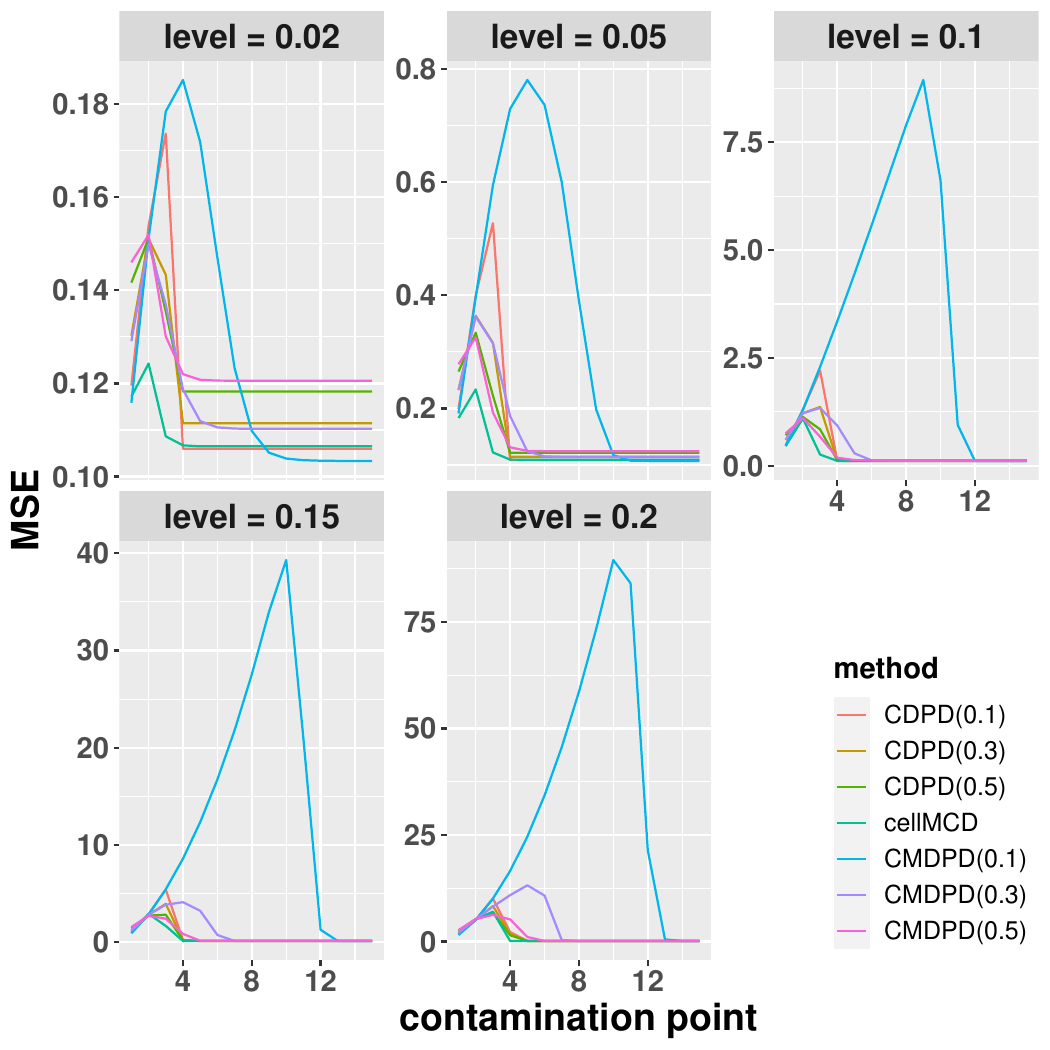}
		\includegraphics[width=0.49\linewidth]{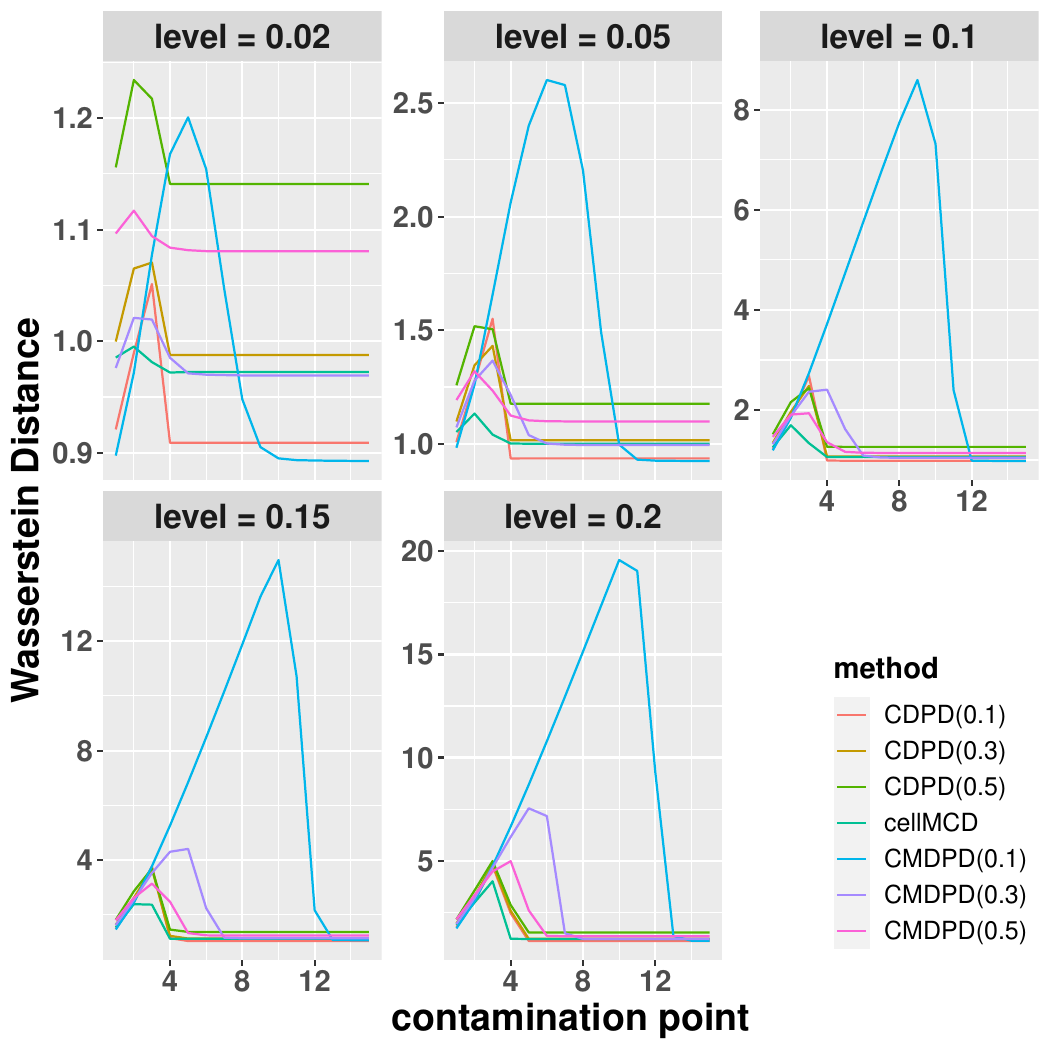}
		\label{FIG:normal-cell-30}
	}
	\caption{Plots of empirical MSE and Wasserstein-2 distance of different competing estimators over the contamination point $y$
		for data dimension $d=30$.}
	\label{FIG:normal_cont30}
\end{figure}

\begin{figure}
	\centering
	\subfloat[Casewise contamination]{
		\includegraphics[width=0.49\linewidth]{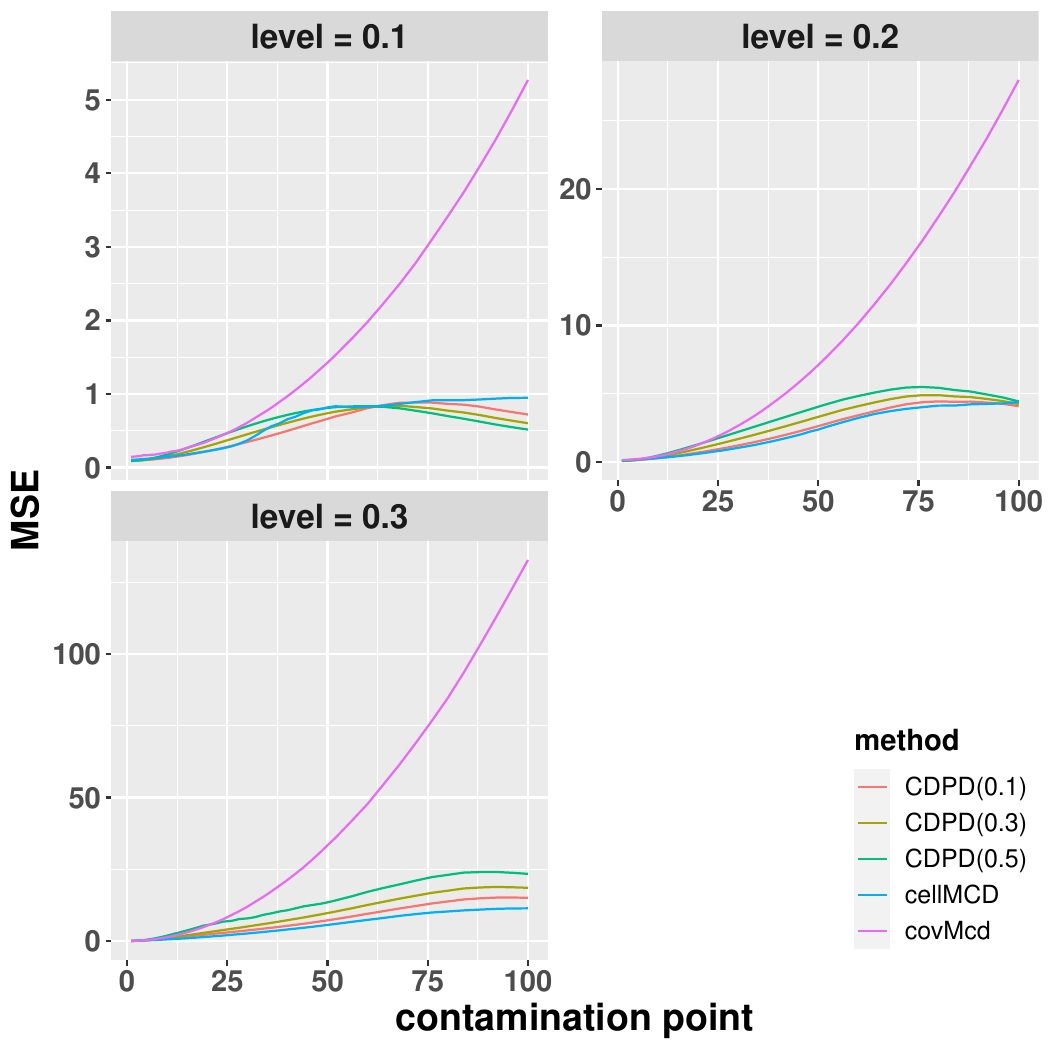}
		\includegraphics[width=0.49\linewidth]{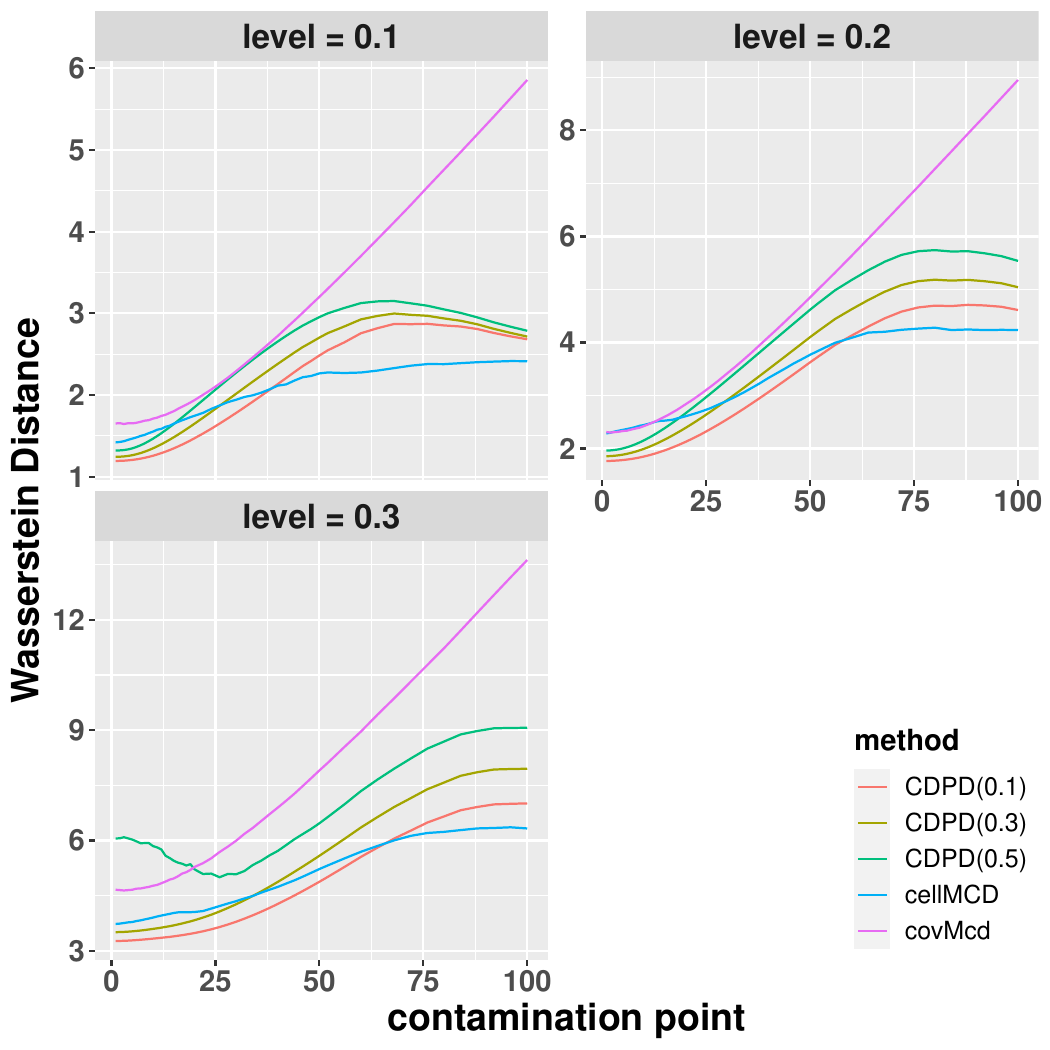}
		\label{FIG:normal-case-40}
	}
\\
	\subfloat[Cellwise contamination]{
		\includegraphics[width=0.49\linewidth]{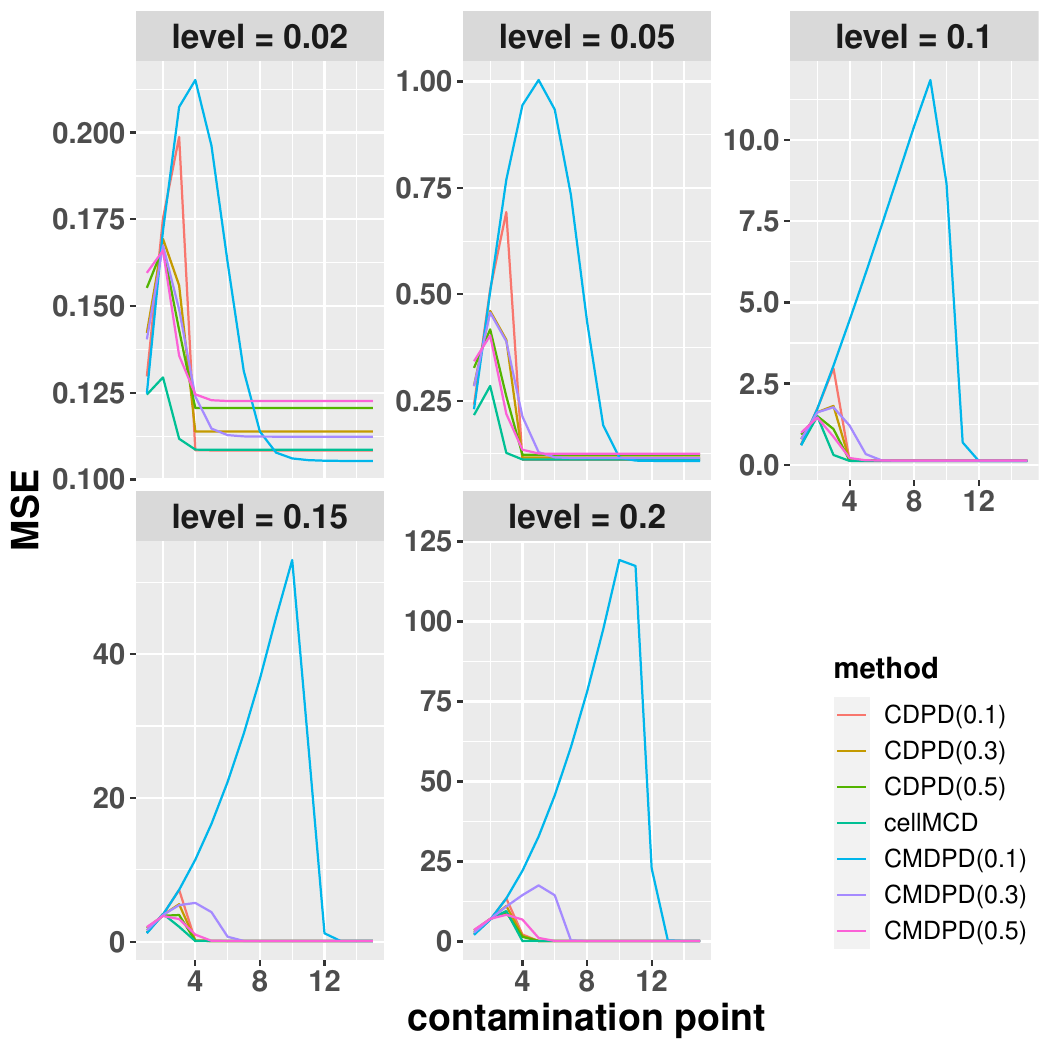}
		\includegraphics[width=0.49\linewidth]{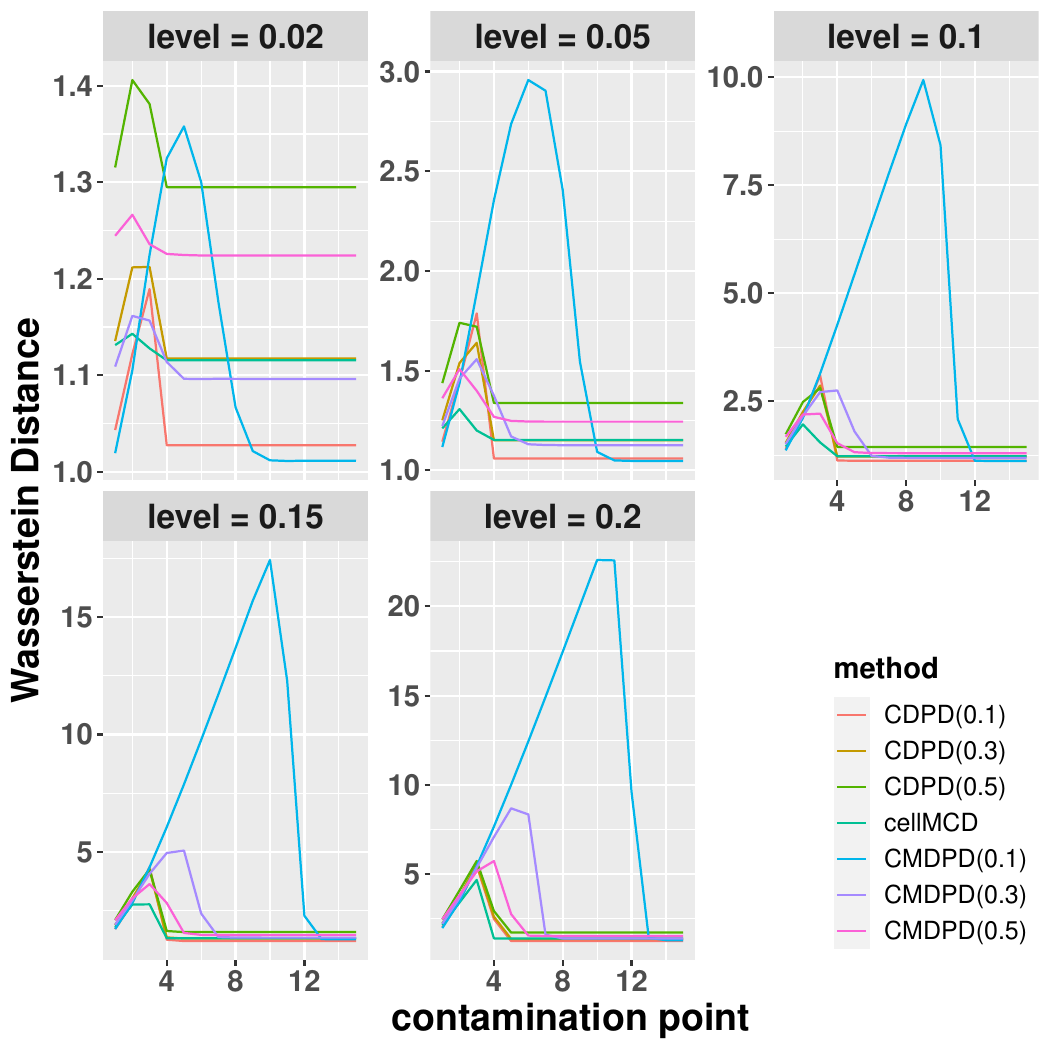}
		\label{FIG:normal-cell-40}
	}
	\caption{Plots of empirical MSE and Wasserstein-2 distance  of different competing estimators over the contamination point $y$
		for data dimension $d=40$.}
	\label{FIG:normal_cont40}
\end{figure}

\begin{figure}
	\centering
	\subfloat[Casewise contamination]{
		\includegraphics[width=0.49\linewidth]{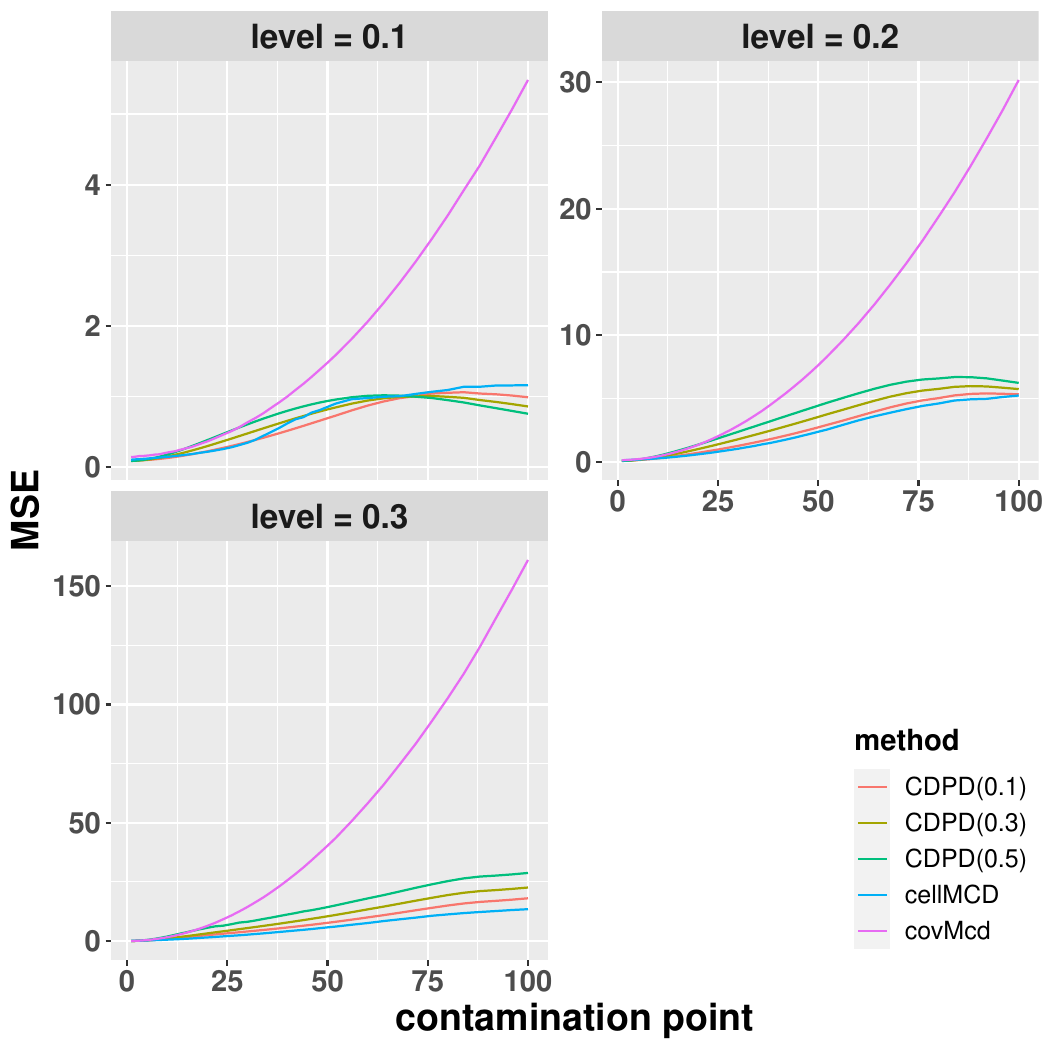}
		\includegraphics[width=0.49\linewidth]{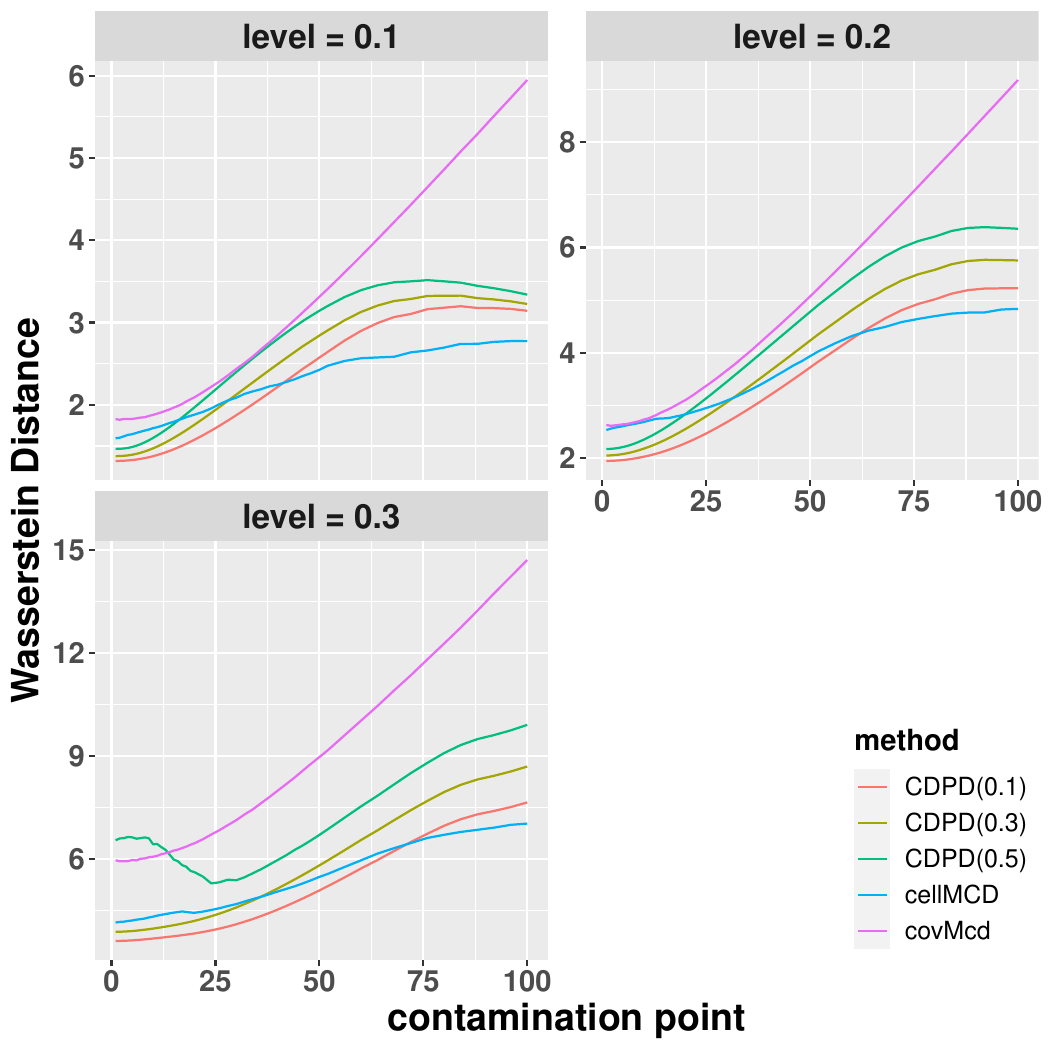}
		\label{FIG:normal-case-50}
	}
\\
	\subfloat[Cellwise contamination]{
		\includegraphics[width=0.49\linewidth]{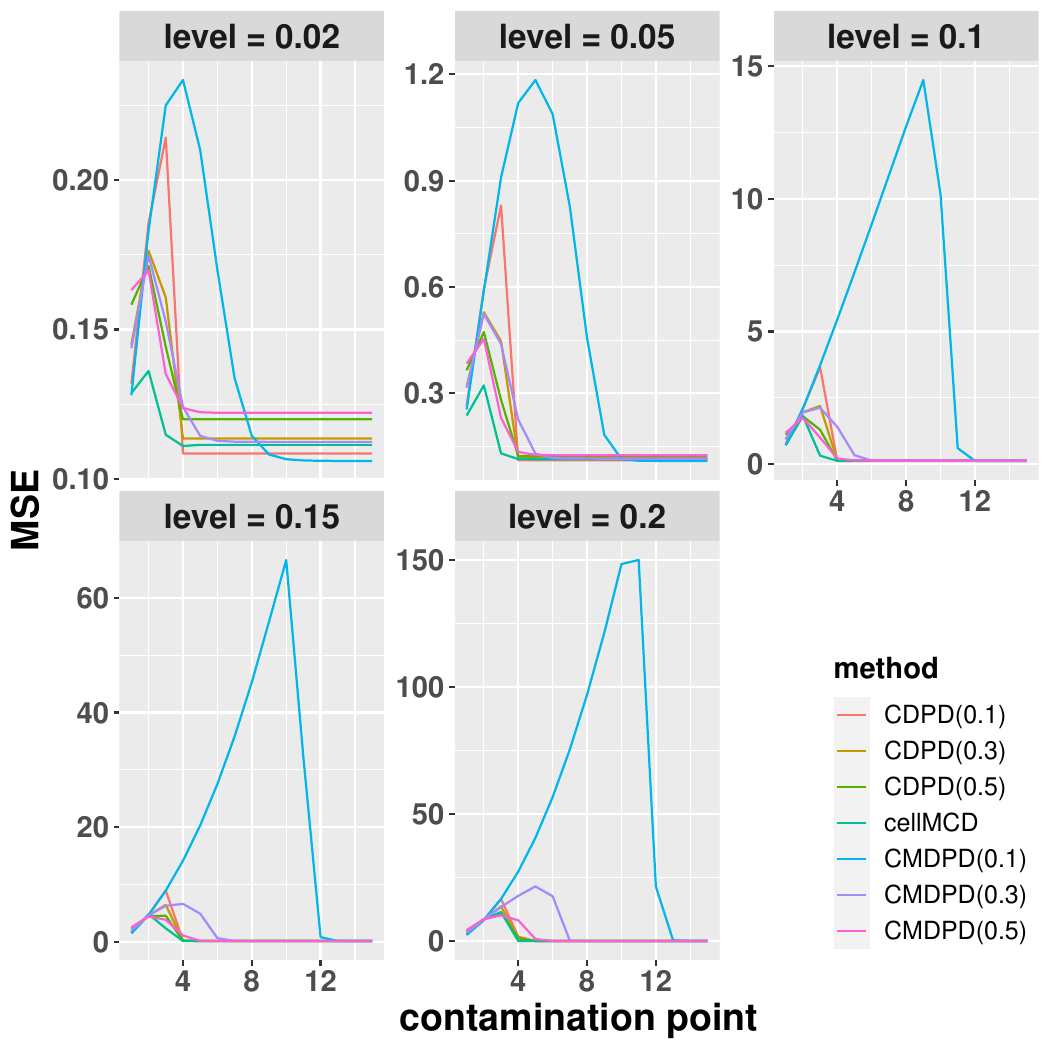}
		\includegraphics[width=0.49\linewidth]{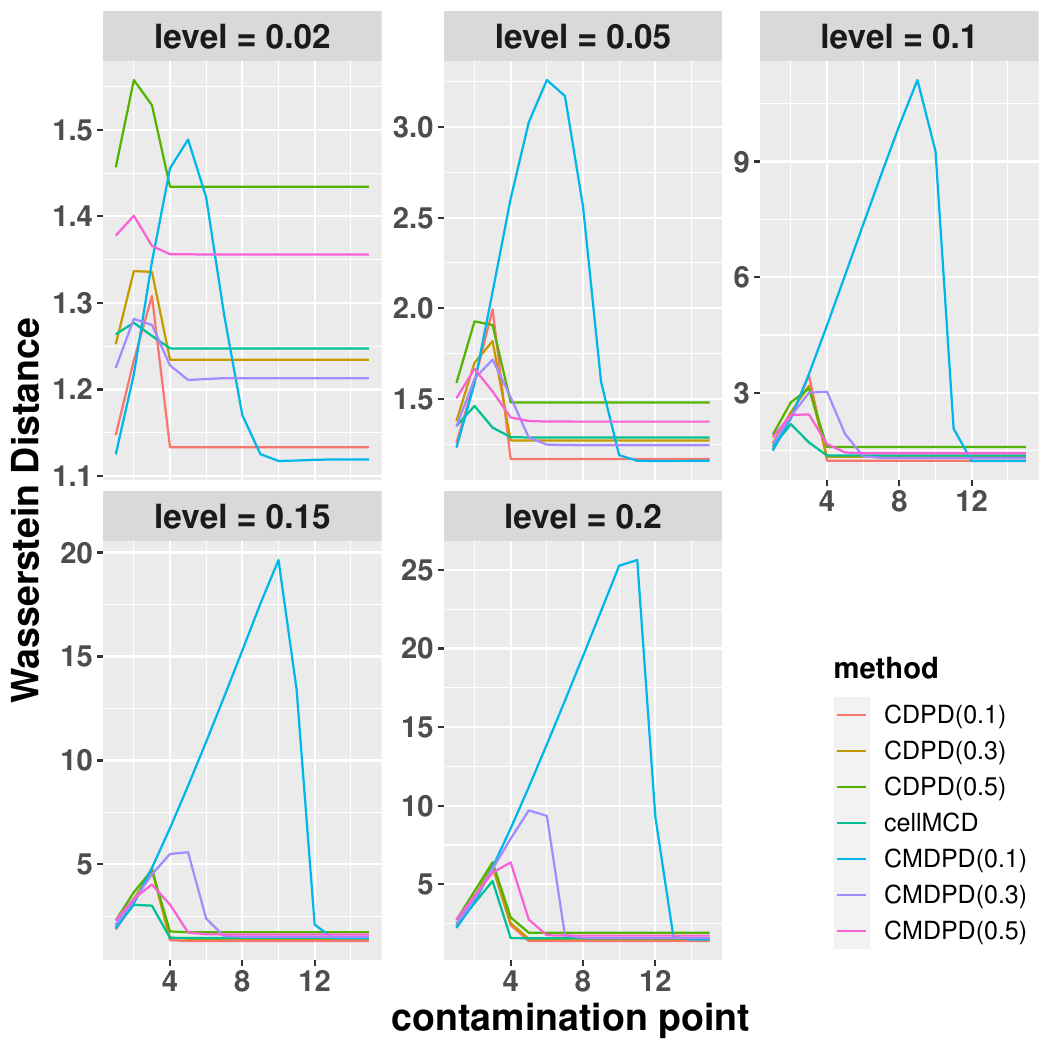}
		\label{FIG:normal-cell-50}
	}
	\caption{Plots of empirical MSE and Wasserstein-2 distance of different competing estimators over the contamination point $y$
		for data dimension $d=50$.}
	\label{FIG:normal_cont50}
\end{figure}

\begin{figure}[!ht]
	\centering
	\includegraphics[width=0.48\linewidth]{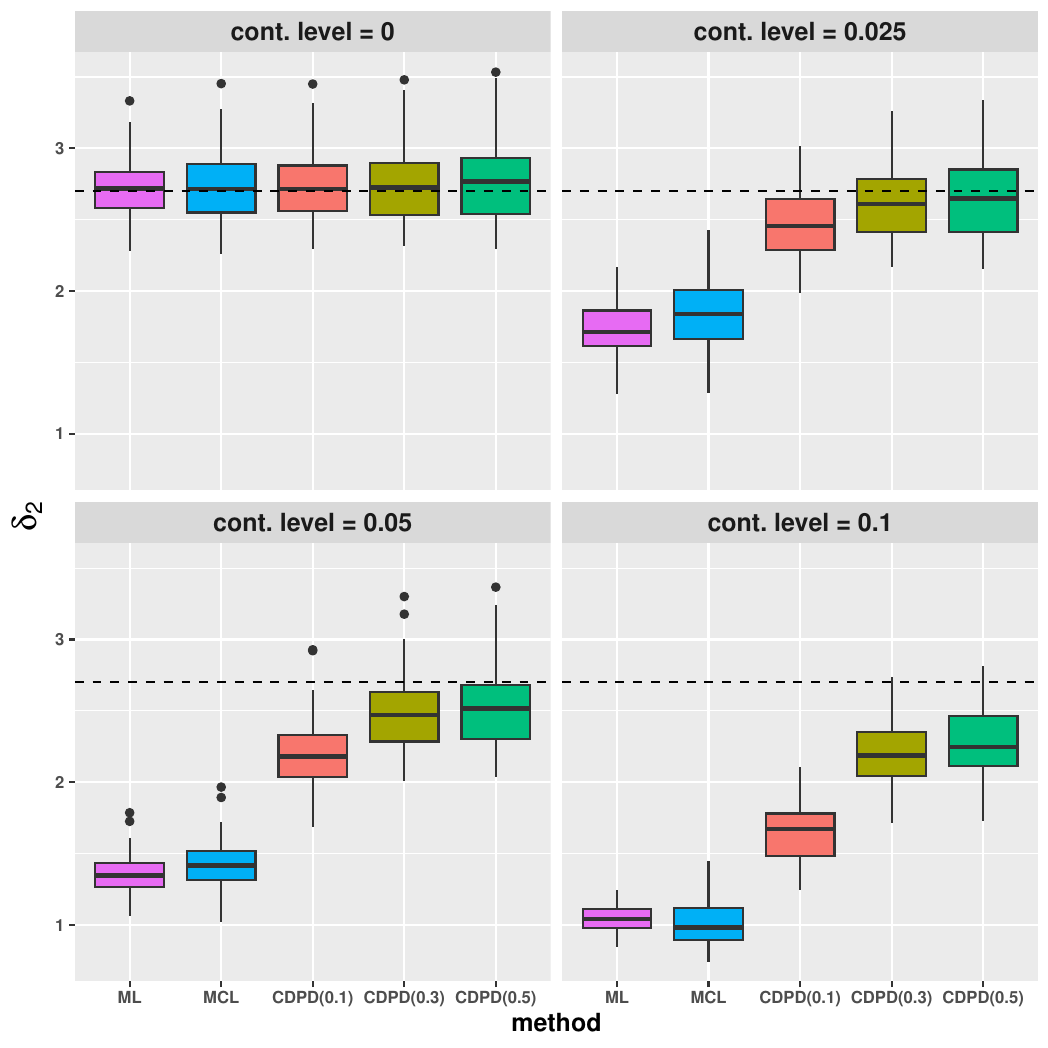}
	\includegraphics[width=0.48\linewidth]{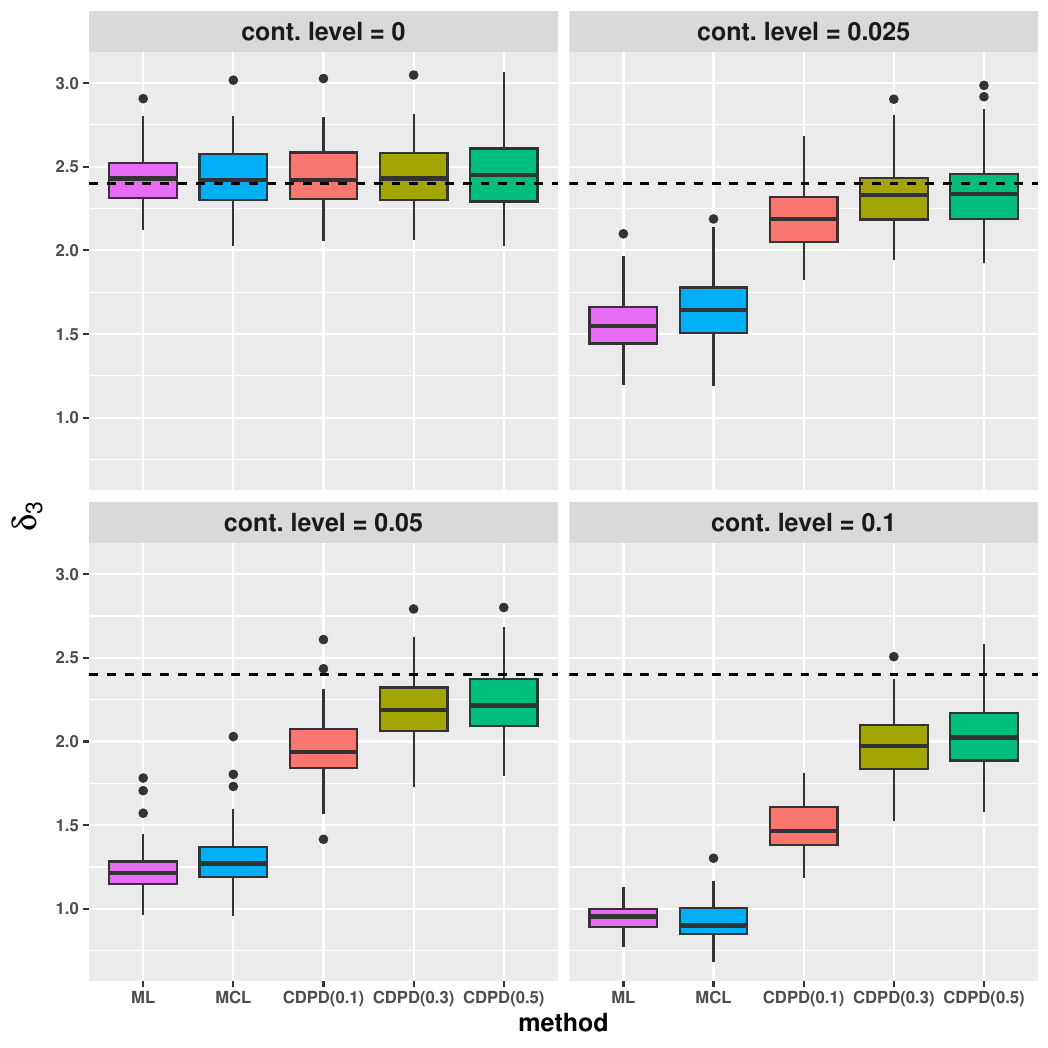}
	\includegraphics[width=0.48\linewidth]{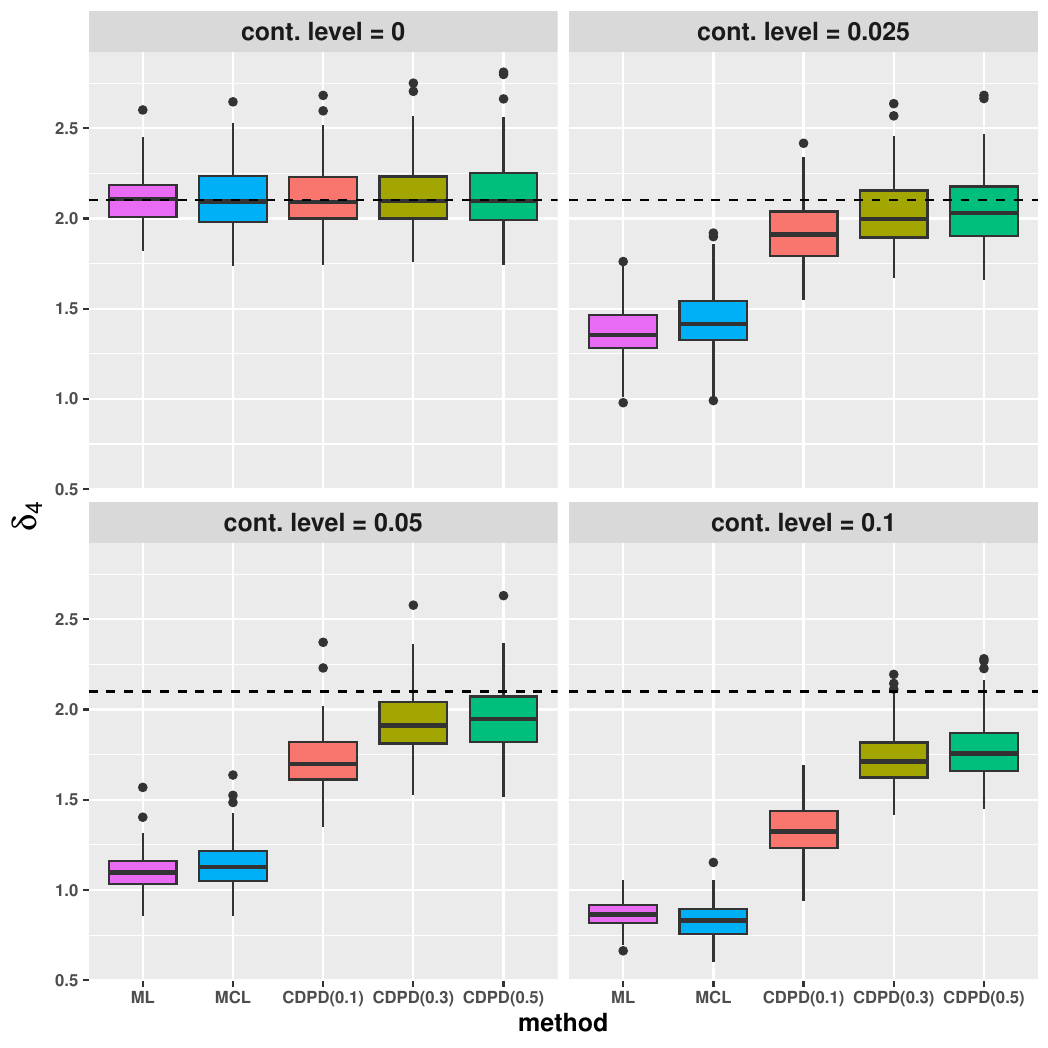}
	\includegraphics[width=0.48\linewidth]{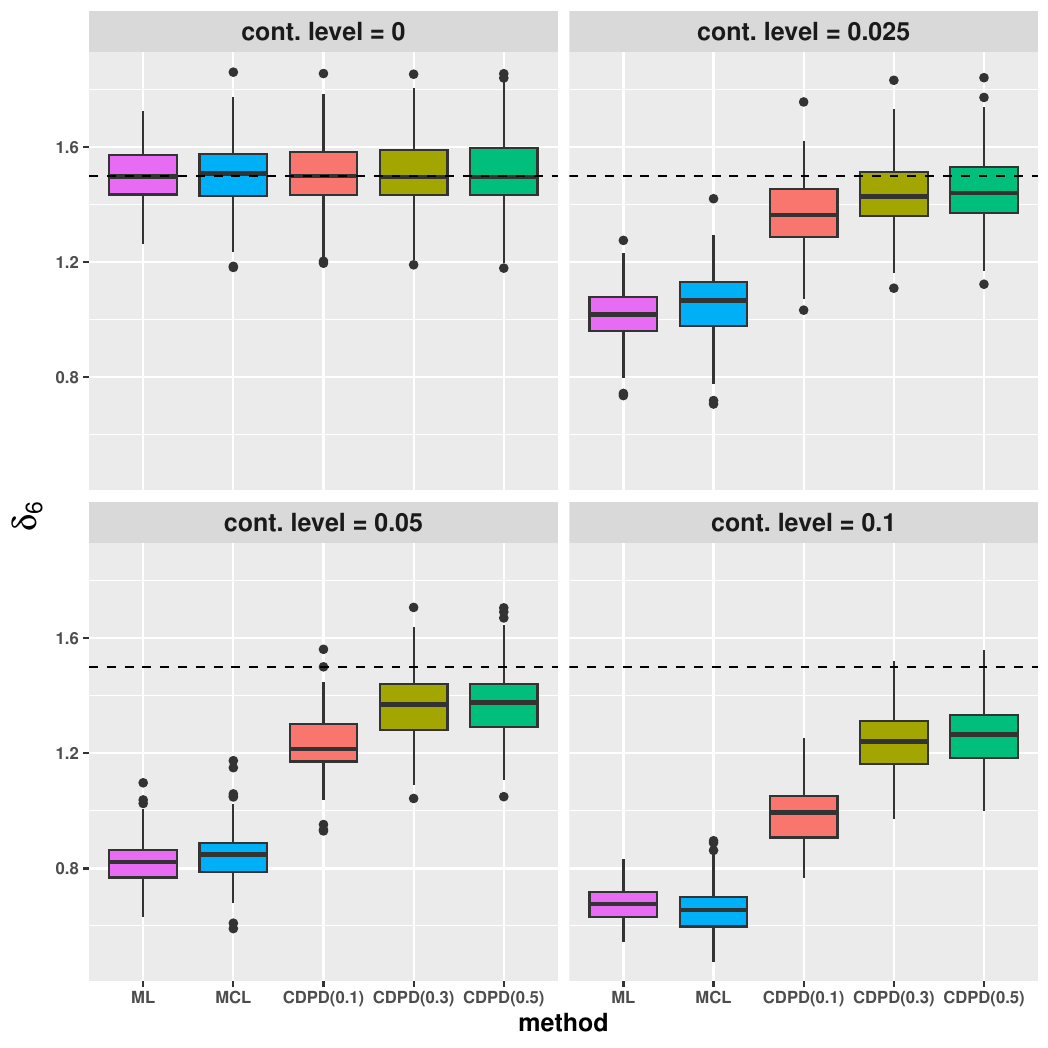}
	\includegraphics[width=0.48\linewidth]{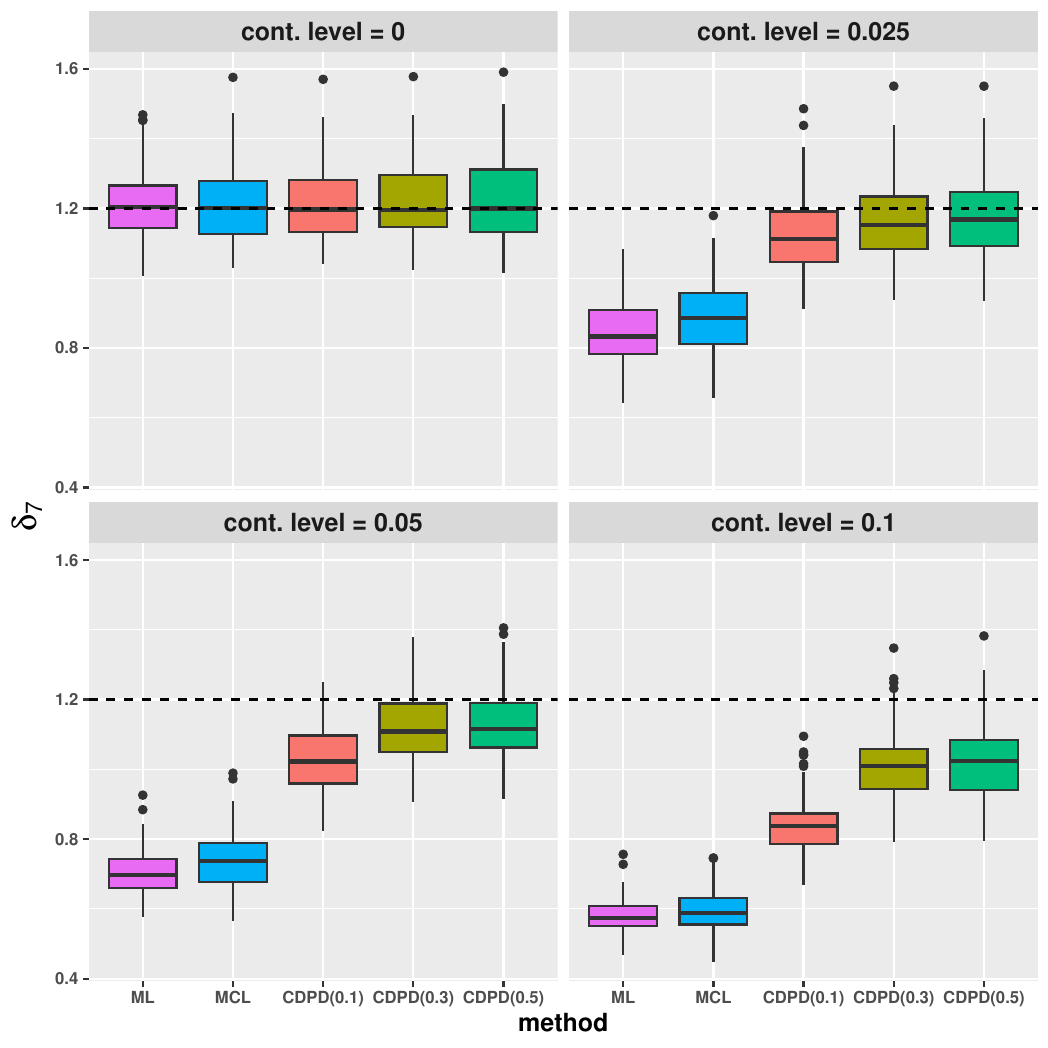}
	\includegraphics[width=0.48\linewidth]{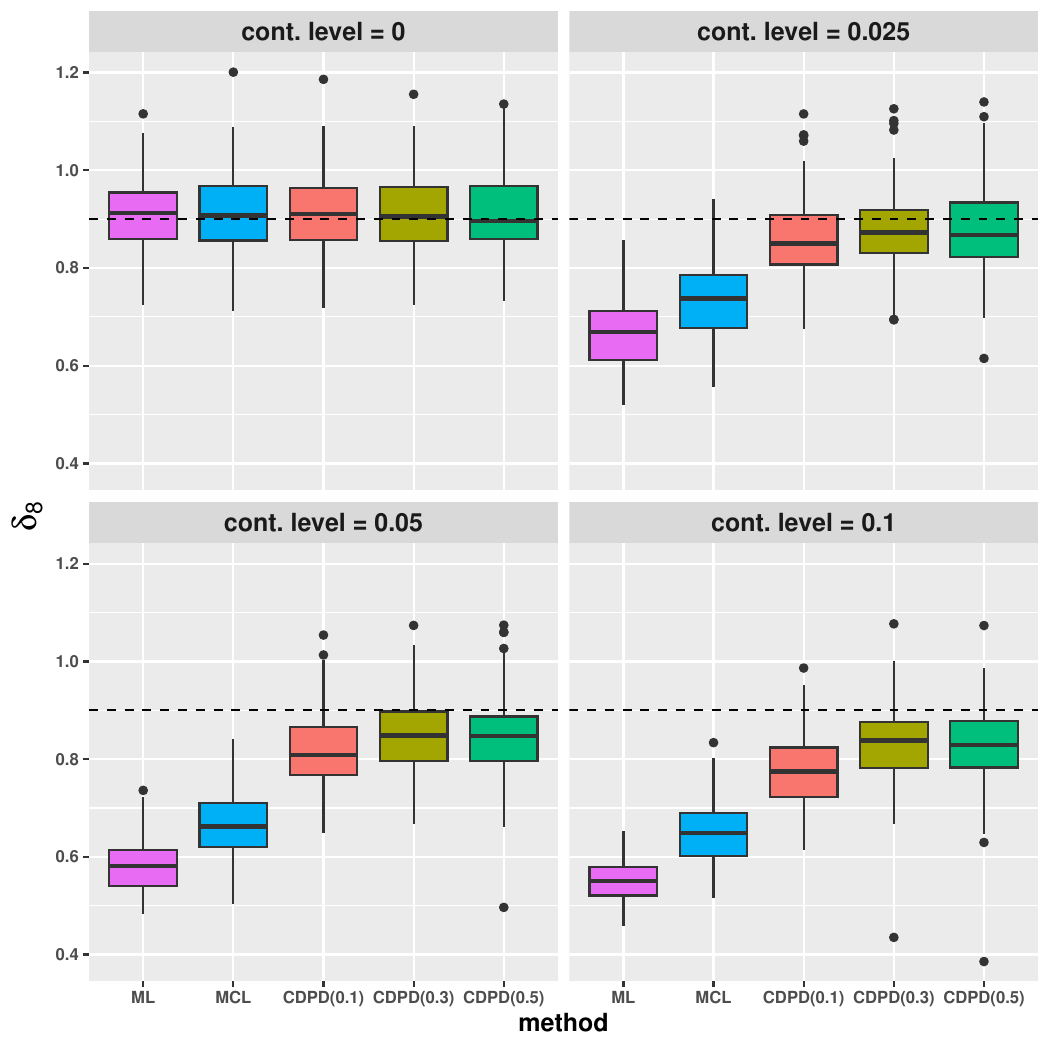}
	\caption{Boxplot of estimated parameters ($\delta_2, \ldots, \delta_4$, $\delta_6, \ldots, \delta_8$) 
		obtained across Monte Carlo replications under the Multivariate Ordered Gamma Model.}
	\label{FIG:marathon-est2}
\end{figure}

\begin{table}[ht]
	\centering
	\caption{Parameter estimates obtained under clean and contaminated ultramarathon data.\\}
	\begin{tabular}{l|rr|rrrr|rrrr}
		\hline
		& MLE & MCLE &\multicolumn{4}{c|}{DPD($\beta$)} & \multicolumn{4}{c}{CDPD($\beta$)}\\ 
$\beta \rightarrow$		&	&	& 0.1	&	0.2	&	0.3	&	0.5	& 0.1	&	0.2	&	0.3	&	0.5\\	\hline\hline
		&\multicolumn{10}{l}{\underline{Best 200 races (potentially clean data)}}\\
$\lambda$	&	11.53	&	11.94	&	15.55	&	18.37	&	19.14	&	27.33	&	14.01	&	16.12	&	17.90	&	20.24	\\
$\delta_1$	&	35.59	&	37.27	&	47.56	&	55.48	&	57.17	&	85.22	&	43.76	&	50.33	&	55.88	&	63.13	\\
$\delta_2$	&	1.37	&	1.52	&	1.31	&	1.14	&	0.93	&	0.59	&	1.65	&	1.78	&	1.87	&	1.96	\\
$\delta_3$	&	1.06	&	1.11	&	1.12	&	1.06	&	0.94	&	0.59	&	1.25	&	1.38	&	1.48	&	1.59	\\
$\delta_4$	&	0.96	&	0.95	&	0.98	&	0.97	&	0.92	&	1.67	&	1.05	&	1.13	&	1.20	&	1.28	\\
$\delta_5$	&	0.88	&	0.83	&	0.96	&	0.98	&	0.96	&	1.46	&	0.92	&	1.00	&	1.07	&	1.15	\\
$\delta_6$	&	0.89	&	0.81	&	0.95	&	0.97	&	0.94	&	0.84	&	0.90	&	0.99	&	1.07	&	1.15	\\
$\delta_7$	&	0.82	&	0.75	&	0.86	&	0.89	&	0.90	&	2.45	&	0.83	&	0.91	&	0.98	&	1.06	\\
$\delta_8$	&	0.76	&	0.67	&	0.83	&	0.85	&	0.80	&	0.59	&	0.75	&	0.84	&	0.90	&	0.98	\\
$\delta_9$	&	0.77	&	0.67	&	0.81	&	0.81	&	0.77	&	0.59	&	0.75	&	0.82	&	0.88	&	0.96	\\
$\delta_10$	&	0.76	&	0.70	&	0.76	&	0.76	&	0.75	&	0.74	&	0.76	&	0.82	&	0.87	&	0.93	\\
		\hline\hline		
		&\multicolumn{10}{l}{\underline{10\% Casewise contamination with worst races}}\\
$\lambda$	&	5.12	&	2.18	&	15.71	&	18.48	&	19.09	&	27.37	&	12.82	&	16.04	&	17.79	&	19.76	\\
$\delta_1$	&	18.02	&	7.83	&	47.81	&	55.55	&	56.84	&	85.32	&	39.91	&	49.81	&	55.22	&	61.27	\\
$\delta_2$	&	0.77	&	0.37	&	1.26	&	1.10	&	0.91	&	0.59	&	1.43	&	1.64	&	1.72	&	1.78	\\
$\delta_3$	&	0.67	&	0.30	&	1.09	&	1.04	&	0.92	&	0.59	&	1.12	&	1.33	&	1.42	&	1.50	\\
$\delta_4$	&	0.60	&	0.26	&	0.96	&	0.96	&	0.91	&	1.68	&	0.94	&	1.09	&	1.16	&	1.23	\\
$\delta_5$	&	0.58	&	0.24	&	0.96	&	0.99	&	0.96	&	1.46	&	0.84	&	0.98	&	1.05	&	1.12	\\
$\delta_6$	&	0.58	&	0.23	&	0.96	&	0.98	&	0.95	&	0.84	&	0.83	&	0.98	&	1.06	&	1.14	\\
$\delta_7$	&	0.53	&	0.22	&	0.85	&	0.88	&	0.90	&	2.46	&	0.75	&	0.88	&	0.94	&	1.01	\\
$\delta_8$	&	0.54	&	0.22	&	0.85	&	0.85	&	0.79	&	0.59	&	0.72	&	0.85	&	0.92	&	0.99	\\
$\delta_9$	&	0.54	&	0.23	&	0.81	&	0.81	&	0.77	&	0.59	&	0.70	&	0.82	&	0.89	&	0.96	\\
$\delta_10$	&	0.53	&	0.26	&	0.77	&	0.77	&	0.75	&	0.74	&	0.72	&	0.83	&	0.88	&	0.93	\\
		\hline\hline
		&\multicolumn{10}{l}{\underline{Artificial cellwise contamination within 20 best races}}\\
$\lambda$	&	1.85	&	2.32	&	5.87	&	12.05	&	12.61	&	15.67	&	7.34	&	13.11	&	14.97	&	15.31	\\
$\delta_1$	&	5.26	&	6.80	&	17.52	&	36.61	&	38.02	&	46.71	&	22.44	&	40.58	&	46.45	&	47.55	\\
$\delta_2$	&	0.61	&	0.66	&	0.95	&	1.20	&	1.03	&	0.76	&	1.25	&	1.78	&	1.90	&	1.82	\\
$\delta_3$	&	0.47	&	0.39	&	0.77	&	1.10	&	1.07	&	1.25	&	0.80	&	1.19	&	1.29	&	1.26	\\
$\delta_4$	&	0.42	&	0.31	&	0.63	&	0.83	&	0.78	&	0.62	&	0.66	&	0.98	&	1.06	&	1.04	\\
$\delta_5$	&	0.41	&	0.29	&	0.65	&	0.86	&	0.79	&	0.62	&	0.64	&	0.97	&	1.06	&	1.06	\\
$\delta_6$	&	0.40	&	0.27	&	0.67	&	0.95	&	0.96	&	1.08	&	0.59	&	0.88	&	0.96	&	0.98	\\
$\delta_7$	&	0.39	&	0.27	&	0.64	&	0.92	&	0.93	&	0.90	&	0.56	&	0.83	&	0.91	&	1.20	\\
$\delta_8$	&	0.39	&	0.30	&	0.60	&	0.79	&	0.80	&	0.94	&	0.56	&	0.81	&	0.89	&	0.15	\\
$\delta_9$	&	0.42	&	0.37	&	0.58	&	0.73	&	0.72	&	0.74	&	0.57	&	0.77	&	0.83	&	1.08	\\
$\delta_10$	&	0.62	&	0.75	&	0.70	&	0.81	&	0.77	&	0.73	&	0.77	&	0.88	&	0.92	&	0.96	\\
		\hline
	\end{tabular}
	\label{TAB:marathon-estimates}
\end{table}

\end{document}